\documentclass{scrartcl}

\usepackage{geometry}                		
\usepackage{graphicx}				
\usepackage{subcaption} 
\graphicspath{{Figures/}}

\RedeclareSectionCommand[
  font=\normalfont\scshape\Large\centering
]{section}

\setkomafont{subsection}{\normalfont\bfseries\large}       
\setkomafont{subsubsection}{\normalfont\itshape\normalsize} 
\setkomafont{paragraph}{\normalfont\normalsize\itshape}
\RedeclareSectionCommand[
  beforeskip=1ex plus .2ex,
  afterskip=1ex,
  afterindent=false
]{paragraph}
\setkomafont{subparagraph}{\normalfont\normalsize\itshape}

\setkomafont{sectionentry}{\scshape}       
\setkomafont{pagenumber}{\normalfont}       
\setkomafont{title}{\scshape\Large\centering}  
\usepackage{mathrsfs}
\usepackage{enumitem}
\usepackage{mathtools}
\usepackage{amssymb}
\usepackage{amsmath}
\usepackage{stmaryrd}
\usepackage{blkarray}
\usepackage{booktabs}
\usepackage{siunitx}
\usepackage{longtable}
\usepackage{array}
\usepackage{threeparttable} 
\usepackage{musicography}

\newcolumntype{L}[1]{>{\raggedright\arraybackslash}p{#1}}

\usepackage{rotating, graphicx}

\usepackage[T1]{fontenc}
\usepackage[french,english]{babel}

\usepackage{tikz-cd}
\usetikzlibrary{calc,babel} 

\usepackage[english]{babel}
\usepackage{amsthm}

\usepackage{biblatex}
\usepackage[colorlinks=true,
            linkcolor=blue,      
            citecolor=teal,      
            urlcolor=magenta,    
            filecolor=purple     
]{hyperref}

\renewenvironment{abstract}{
  \small
  \noindent\textbf{\textsc{Abstract.}}\ } 
{\par\vspace{1em}}
\newcommand{\Act}{\mathbf{Act}}

\newcommand{\longhookrightarrow}{\lhook\joinrel\longrightarrow}

\newcommand{\Cat}{\mathbf{Cat}}
\newcommand{\core}{\operatorname{core}}
\newcommand{\Def}{\operatorname{Def}}
\newcommand{\DefFam}{\mathbf{DefFam}}
\newcommand{\FuncHarmSys}{\mathbf{FuncHarmSys}}
\newcommand{\Gpd}{\mathbf{Gpd}}
\newcommand{\HarmVocab}{\mathbf{HarmVocab}}
\newcommand{\HarmPractice}{\mathbf{HarmPractice}}
\newcommand{\Hept}{\mathcal{H}}
\newcommand{\HDef}{\mathbf{HDef}}
\newcommand{\HMode}{\mathbf{HMode}}

\newcommand{\Mode}{\mathbf{Mode}}

\newcommand{\ModeOrbCov}{\mathbf{ModeOrbCov}}
\newcommand{\ModeOrbCovSys}{\mathbf{ModeOrbCovSys}}

\newcommand{\bfR}{\mathbf{R}}
\newcommand{\Scale}{\mathbf{Scale}}
\newcommand{\ScaleOrbCov}{\mathbf{ScaleOrbCov}}
\newcommand{\Set}{\mathbf{Set}}

\newcommand{\spec}{\operatorname{spec}}
\newcommand{\T}{\mathrm{T}}
\newcommand{\bfT}{\mathbf{T}}
\newcommand{\Z}{\mathbb{Z}}

\newcommand{\ActT}{\mathbf{Act}_{\bfT}}
\newcommand{\ActR}{\mathbf{Act}_{\bfR}}
\newcommand{\ActTR}{\mathbf{Act}_{\bfT \times \bfR}}
\newcommand{\Func}{\mathbf{Func}}

\newcommand{\OrbCovT}{\mathbf{OrbCov}_\bfT}
\newcommand{\OrbCovTR}{\mathbf{OrbCov}_{\bfT \times \bfR}}

\theoremstyle{definition}

\newtheorem{definition}{Definition}[section]
\newtheorem{convention}{Convention}[section]
\newtheorem{example}{Example}[section]

\theoremstyle{plain}

\newtheorem{lemma}{Lemma}[section]
\newtheorem{proposition}{Proposition}[section]
\newtheorem{theorem}{Theorem}[section]

\theoremstyle{remark}
\newtheorem{remark}{Remark}[section]

\title{Mathematical Principles of a Generalized Tonal Practice}
\author{Drew Flieder}
\date{}

\begin{document}
\maketitle

\begin{abstract}
We develop mathematical principles for a generalized tonal practice. We begin with an algebraic theory of modes and scales, formalized as groups and torsors over cyclic groups, respectively. We identify three musically salient classes of mode isomorphism---translations, rotations, and deformations. Upon this basis we organize successive layers of harmonic structure into a hierarchy of Grothendieck fibrations, built from scale types (generalizing the diatonic scale), modes (generalizing the diatonic modes), orbit covers (generalizing their covering by tertian triads), deformation families (generalizing chromatic alteration, such as modal mixture and secondary chords), and harmonic-function assignments (generalizing the assignment of tonic, subdominant, and dominant function). Equipping the resulting structure with a generative syntax---generalizing Rohrmeier's generative grammar of tonal harmony---yields a formal definition of a \emph{harmonic practice}: a harmonic vocabulary together with an assignment of harmonic functions and a generative syntax accommodating hierarchical harmonic organization. We show that harmonic practices form a category, with morphisms constituting an inheritance and enrichment of structure from one harmonic practice to another. The theory originates in the author's own compositional practice, developed initially through the series of compositions Op.~26--31.
\end{abstract}

\pagebreak

\tableofcontents

\section{Introduction}
\label{sec:introduction}

For roughly a decade, my work as a composer followed in the lineage of formal approaches developed by composers like Stockhausen, Xenakis, Babbitt, and Cage.\footnote{The account of this reorientation given in the remainder of this section closely follows, and in places directly adapts, the introduction of \cite{flieder2026functional}, a paper documenting the early relationship between this compositional turn and the preliminary orbit-cover theory of \cite{flieder2026scales}.} I experimented with a variety of systems, mostly of my own devising, for organizing musical materials: parameters governing register, rhythm, articulation, and timbral character, each contributing its own layer of organization to a piece. My approach to harmony throughout this time could broadly be characterized as belonging to the broader ``set-theoretic'' tradition of post-tonal harmonic organization. Then, at a certain point—abruptly, or perhaps only abruptly from where I stood—I experienced a reorientation in my musical thinking. I became convinced that there was something compelling, here and now, in the idea of functional harmony, and that this idea could be recovered without resorting to pastiche or stylistic imitation of older musics. The first piece I wrote afterward (\emph{Tonalities A}) was almost exceedingly plain: a single tonic-subdominant-dominant progression, repeated for the length of the piece.

To say why this reorientation mattered, it is worth first saying something about what I was turning away from. Serial systems of organization tend to suggest an approach to form based on contrast, even where the contrasting parts have affinities with one another. Stockhausen's concept of moment form is a natural outgrowth of this tendency: when a section is determined by a conjunction of parameter values, and many parameters are in play, one expects highly contrasting aesthetic qualities from moment to moment. In much of this music, harmony functions as one parameter among several, contributing to an overall aesthetic gestalt rather than following a process of elaborated harmonic progression in its own right.

In America, Babbitt developed a genuinely foundational theory of serialism, mathematically rigorous in its own right, and later extended by other American serial composers and theorists. Morris in particular has written extensively on the aesthetic ends serial thought can serve, which are often non-teleological, as in much twentieth-century music more broadly: such music rewards a listener who attends to the \emph{present} moment with continual openness, awareness, and sensitivity, rather than one who is always looking forward to a point of arrival, such as a cadence. This is a demanding way of listening, and the music that results from it has produced some of the most beautiful moments I know in any music---passages in Webern, Feldman, Morris, and Stockhausen among them, achieving aesthetic qualities with no real precedent in tonal harmony. A hexachord, for instance, voiced a certain way, has a sonorous quality unlike anything available in a diatonic tertian harmonic vocabulary; such sonorities are simply unavailable within common-practice tonality's harmonic resources.

What common-practice tonality offers instead is a particular economy. The music of Bach is no less demanding in its craft, but it asks comparatively little of the listener, while returning a great deal: a small investment of attention yields a correspondingly large sense of coherence and, often, awe. This is a different achievement from the one described above, and it is not, on reflection, an achievement that many post-tonal harmonic practices set out to reproduce.\footnote{Feldman \cite[174]{feldman2000give} was explicit that this was an intended refusal: functional harmony, in his account, ``hears for us''---a tonal progression does the listener's work in advance, arriving at its effects the way an accountant arrives at a sum, correctly and without requiring the listener's own attention. His own music seeks to withhold that convenience, asking to be heard rather than doing the hearing on the listener's behalf.}

What I have wanted, and what this theory is built to make possible, is to inherit both traditions: the economy of common-practice tonality, in which a little attention returns a great deal, together with the expanded harmonic palette that post-tonal music has made available. This synthesis is the most concise statement of this theory's motivation. Clearly, this did not mean returning to the harmonic vocabulary of diatonic common-practice tonality itself. Common-practice tonality is a richly developed system, optimized over centuries for a particular class of compositional aims, but its harmonic vocabulary---tertian chords built on diatonic scales---need not be the only vocabulary capable of supporting functional syntax. The challenge, as I came to see it, was to build a harmonic framework capable of supporting the core devices of functional tonality---cadentiality, modulation, the recursive embedding of functional regions within one another, prolongation, and so on---while substantially expanding the harmonic resources available to a composer working with them.

My first response to this challenge was a preliminary theory of orbit covers, developed in \cite{flieder2026scales} and used as the basis for a set of compositions titled \emph{Tonalities} (Op.~26). In these pieces I attempted to translate aspects of common-practice functional syntax into non-diatonic, non-tertian harmonic environments generated by orbit covers, drawing largely on the generative approach to harmonic syntax developed by Rohrmeier \cite{rohrmeier2011towards} as a framework for constructing functional progressions. The piece made clear that this was possible: that one could write music retaining a sense of harmonic teleology---in particular, a real sense of cadence---while working with harmonies well outside the diatonic and tertian norm, and that the underlying framework, while structurally richer than familiar post-tonal set-theoretic systems, remained straightforward to compose with in practice.

I continued applying the preliminary theory in my compositional work through \emph{Quartet After Haydn} (Op.~29), and only afterward returned to the theory itself, producing an initial systematic draft. This draft then served as the basis for two further works, \emph{Chorales} (Op.~30) and \emph{Variations for Piano} (Op.~31), by which point the framework had grown substantially more systematic than the one underlying \emph{Tonalities} through \emph{Quartet After Haydn}. The relationship between \emph{Tonalities} and \emph{Chorales}---the way the compositional practice and the theory developed each other in turn---is documented in \cite{flieder2026functional}, presented at the 2026 conference on Mathematics and Computation in Music; I do not repeat that account here. Following that conference, I undertook a complete rewrite of the theory, organized now around a hierarchy of Grothendieck fibrations (\S~\ref{sec:a-hierarchy-of-fibrations-over-heptatonic-scales}) and culminating in the formal definition of a \emph{harmonic practice} (\S~\ref{sec:formalizing-harmonic-practices}). This paper presents that theory. Where the earlier draft made certain compositional possibilities available only implicitly, the present framework makes them explicit and tractable: a harmonic practice now lives in a category of harmonic practices, with morphisms giving a formal notion of inheritance between one harmonic practice and another; this is the very relationship I had already been enacting, piece by piece, through \emph{Variations for Piano} (Op.~31), but had not yet made explicit.

\subsection*{Roadmap}

We now outline the paper's construction. We begin by modeling a mode as a pitch-class set equipped with a group structure whose identity marks a distinguished tonic; forgetting that tonic yields a torsor over a cyclic group---a group structure without a distinguished identity, and equivalent to a generalized interval system in the sense of Lewin \cite{lewin2007generalized}---which we take as our formalism for a scale (\S~\ref{sec:modes-and-scales}). We then isolate three canonical classes of mode isomorphism: translations, rotations, and deformations. Translations and rotations each extend to endofunctors on the category of modes, but deformations do not compose canonically and so form only a quiver over the set of $n$-note scales. Fixing $n = 7$ and taking as objects the translation orbits of heptatonic scales, we close this quiver under composition of deformations to obtain a well-defined category, $\HDef$, whose morphisms are the resulting deformation maps (\S~\ref{sec:three-classes-of-mode-isomorphisms}); $\HDef$ forms the base of the hierarchy of fibrations constructed in \S~\ref{sec:a-hierarchy-of-fibrations-over-heptatonic-scales}. We then equip scales with orbit covers---chords generated by translating a seed subset through that scale---and, extending the classification given in \cite{flieder2026scales}, define a category of orbit covers and characterize precisely when a morphism between two orbit covers exists (\S~\ref{sec:orbit-covers}).

We then assemble harmonic data into a hierarchy of Grothendieck fibrations over $\HDef$, successively enriching each scale with modal, orbit-cover, deformation-family, and harmonic-function data (\S~\ref{sec:a-hierarchy-of-fibrations-over-heptatonic-scales}). Using this hierarchy, we define a \emph{harmonic practice}, which consists of the following dependent data:
\begin{enumerate}
\item a choice of scale class,
\item an orbit-cover structure,
\item a choice of modes,
\item a family of chromatic alterations,
\item an assignment of harmonic functions, and
\item a generative syntax.
\end{enumerate}
It is at this point that the preceding formal apparatus comes together as a tool in the service of composition: a harmonic practice is the object a composer actually works within, and its data are exactly the choices such a composer must make. We then formalize what constitutes a morphism between harmonic practices, expressing when one practice inherits or enriches another (\S~\ref{sec:formalizing-harmonic-practices}). It will become evident that Rohrmeier's generative grammar of common-practice \cite{rohrmeier2011towards} tonal harmony is recovered as a single instance of this general construction.

Rohrmeier's grammar is adopted here as the basis for a generative harmonic syntax because it is thorough in how it organizes tonal progression across four distinct levels---phrase, function, scale-degree, and surface---each with its own rewrite rules, rather than treating harmonic syntax as a single flat grammar of chord-to-chord transitions. This hierarchical structure is what the present framework inherits and generalizes in \S~\ref{sec:formalizing-harmonic-practices}. 

\subsection*{Related Work}
The torsor construction underlying our notion of scale (\S~\ref{sub:scales}) is a descendant of (and in fact equivalent to) Lewin's generalized interval systems \cite{lewin2007generalized}, and we recall his definition explicitly before showing that it recasts as a simply transitive group action. A recent and complementary approach is given by Harasim et al.\ \cite{harasim2020axiomatic}, who provide a rigorous axiomatic account of scales as cyclic embeddings between finite cyclically ordered sets; although the present work does not build on that framework, it shares a related aim of giving scale theory a structural mathematical foundation. Riemannian function theory \cite{riemann1896harmony}, de la Motte's functional assignment \cite{motte1991study}, and Rohrmeier's generative grammar \cite{rohrmeier2011towards} are the principal examples of common-practice harmonic function that we use throughout. Rohrmeier's syntax in particular is what our own constructions are built to generalize and, ultimately, to recover. The cadential-set construction of \S~\ref{subsub:the-cadential-set-construction}, following Mazzola \cite{mazzola2002topos} and Muzzulini \cite{muzzulini1995musical} and motivated by Schoenberg's account of modulation \cite{schoenberg1978theory}, is included for its compositional utility in targeting a modulation.

A number of prior efforts share this paper's general aim of extending tonal syntax beyond the diatonic and tertian setting in which it originated: Lewin's own generalization of tonal functions \cite{lewin1982formal}, Lerdahl's tonal-distance framework \cite{lerdahl2001tonal,lerdahl1999composing}, Scotto's hybrid compositional system \cite{scotto2000hybrid}, and Tymoczko's account of tonality in \cite{tymoczko2023tonality}. All share this paper's interest in transplanting aspects of tonal structure onto non-common-practice-tonal material, but the reader familiar with those works will find that their respective approaches to generalizing functional tonality diverge significantly from our own. The scalar orbit-cover classification of \S~\ref{sec:orbit-covers} extends the affine and nerve-isomorphism classification developed in the author's own prior work \cite{flieder2026scales}, which did not yet incorporate a notion of morphism between orbit covers.

\subsection*{Contributions}

This paper makes the following contributions.
\begin{enumerate}
\item We develop the mode and scale formalism and the three canonical classes of mode isomorphism---translations, rotations, and deformations---as the musical and categorical groundwork on which everything that follows depends.
\item We give a complete characterization (Theorem~\ref{th:main-morphism-characterization}) of when a morphism between two orbit covers exists.
\item We assemble the musically relevant data of a harmonic practice into a hierarchy of Grothendieck fibrations---successively the action-groupoid, orbit-cover, deformation-family, and harmonic-function fibrations---each layer enriching the one beneath it, so that the top layer yields a set of harmonies equipped with harmonic functions.
\item We give a formal definition of a \emph{harmonic practice}---a harmonic vocabulary together with a functional assignment and a generative syntax---built from this hierarchy. This provides the machinery for a generalized tonal practice where modal orbit covers occupy the roles of keys of common-practice tonality, and the possibilities for modulation between modal orbit covers are derived canonically from the data of a harmonic practice. We also give a notion of morphism between harmonic practices, which allows for the inheritance and enrichment of one harmonic practice by another. 
\item We demonstrate the generality of the definition of a harmonic practice by recovering Rohrmeier's generative grammar of common-practice tonal harmony as a single instance of it.
\end{enumerate}

Section~\ref{sec:formalizing-harmonic-practices} closes with examples spanning common-practice harmony, a non-diatonic modal orbit cover realized musically via a slight adjustment of the opening phrase of the theme of my \emph{Variations for Piano}, Op.~31, made for pedagogical purposes (Example~\ref{ex:piano-vars}), and a morphism of harmonic practices transporting a Bach-chorale-derived progression into a new tonal environment (Example~\ref{ex:morphism-of-harmonic-practices}). The constructions of \S~\ref{sec:three-classes-of-mode-isomorphisms} onward are developed for heptatonic scales in $\Z_{12}$, the setting in which common-practice tonal harmony lives; the underlying constructions extend without difficulty to $n$-note scales in arbitrary $N$-tone equal temperament.

\section{Modes and Scales}
\label{sec:modes-and-scales}

Musical scales are traditionally conceived as pitch collections, often associated with concepts such as stepwise adjacency, directionality, and modality. In this section, we develop a formal apparatus for encoding such structures. We begin with modes, which we model as pitch-class sets equipped with a distinguished tonic and a compatible group structure; forgetting this tonic then yields a torsor over a cyclic group,\footnote{The reader familiar with Lewin's definition of a \emph{generalized interval system} (see \cite{lewin2007generalized}) will recognize that these structures correspond to simply transitive group actions on a set, and hence may be understood as torsors. This correspondence is discussed at the beginning of \cite[Chapter 7]{lewin2007generalized}, as well as in \cite[Section 2.14]{vuza1988lewin}.} which we take as our definition of a scale. This prepares the ground for our later development  of orbit covers (\S~\ref{sec:orbit-covers}), which rely on this stepwise structure to generate chords across the scale.

\subsection{Modes}
\label{sub:modes}

To give a preliminary account of a scale, let $X \subseteq \Z_N$ be a subset of the $N$-tone pitch-class universe. For example, the C major scale is typically represented as
\[
X = \{ 0, 2, 4, 5, 7, 9, 11 \} \subseteq  \Z_{12}.
\]
This bare set-theoretic notion omits the structural features usually associated with scales, such as adjacency and directionality. For instance, given $x, y \in X$, we may wish to ask how many steps apart they are. These questions cannot be answered if $X$ is treated merely as a set. Additional structure is needed.

A first observation is that a scale $X$ appears to carry a cyclic structure, and hence may be modeled by the group $\Z_n$, where $n = |X|$. Thus, we seek a bijection
\[
g : X \longrightarrow \Z_n
\]
that assigns to each $x \in X$ a position in a stepwise group. This allows us to speak meaningfully about adjacency: $y$ lies $k$ steps above $x$ precisely when $g(y) - g(x) = k$.

Clearly, not every bijection reflects the conventional scalar order. For example, if $X$ is the C major scale and $h : X \to \Z_7$ sends $2 \mapsto 0$ and $11 \mapsto 1$, then $h(11) - h(2) = 1$, implying that B$\natural$ is adjacent to D$\natural$, which is a clearly atypical assignment. While such maps may have interest for the experimental composer, they do not represent the standard step structure. We thus seek a canonical class of bijections $g : X \to \Z_n$ that preserve scalar adjacency.

Although listing the elements of $X$ in ascending order provides an initial ordering, this ordering is arbitrary up to cyclic rotation. We therefore choose the \emph{normal order} of $X$, namely the rotation whose elements occupy the smallest span, with ties broken by the standard left-packing convention of pitch-class set theory.\footnote{See, for instance, \cite[4]{forte1973structure}.} This determines a canonical ordering of the elements of $X$, and hence a map
\begin{equation}
\label{eq:normal-order}
\mathrm{norm}_X : [n] \longrightarrow X,
\end{equation}
where $[n] = \{0,1,\ldots,n-1\}$ indexes the elements of $X$ in normal order. Identifying $[n]$ with $\Z_n$ via 
\[
\begin{matrix}
\eta : & [n] & \longrightarrow & \Z_n \\
& i & \longmapsto & i,
\end{matrix}
\] 
we canonically label the elements of $X$ with group coordinates. This induces a unique map 
\begin{equation}
\label{eq:mu}
\mu : X \longrightarrow \Z_n
\end{equation} 
that makes the following diagram commute:
\[\begin{tikzcd}
	{[n]} && X \\
	\\
	&& {\Z_n}
	\arrow["{\mathrm{norm}_X}", from=1-1, to=1-3]
	\arrow["\eta"', from=1-1, to=3-3]
	\arrow["\mu", from=1-3, to=3-3]
\end{tikzcd}\]
Using $\mu$, we may define a group operation on $X$ by
\[
x \oplus y \coloneqq \mu^{-1}\big(\mu(x) + \mu(y)\big),
\]
which equips $X$ with a group structure isomorphic to $\mathbb Z_n$.

We are now in a position to define a notion of mode.

\begin{definition}[Mode]
\label{def:mode}
Let $\Z_N$ be a cyclic group representing the $N$-element pitch-class universe, and let $X \subseteq \Z_N$ with $|X| = n$. For each $i \in \Z_n$, let $\T_i : \Z_n \to \Z_n$ denote translation by $i$. Define $\mu_i : X \to \Z_n$ by
\begin{equation}
\label{eq:mu_i}
\mu_i \coloneqq \T_i \circ \mu,
\end{equation}
where $\mu : X \to \Z_n$ is the unique map defined by \eqref{eq:mu}. We define the \emph{$i$th mode of $X$} as the group $(X, \oplus_{\mu_i})$, where
\[
x \oplus_{\mu_i} y \coloneqq {\mu_i}^{-1}({\mu_i}(x) + {\mu_i}(y)).
\]
We call the identity element $t \in X$ the \emph{tonic} and refer to $\Z_n$ as the \emph{scale-degree group} of the mode.
\end{definition}

\begin{example}[Modes of a Major Scale]
Let $X = \{0,2,4,5,7,9,11\} \subset \Z_{12}$ denote the C major scale. Its normal order is
\[
\mathrm{norm}_X = (11, 0, 2, 4, 5, 7, 9),
\]
and thus its zeroth mode $\mu_0 : X \to \Z_7$ is Locrian, since it assigns $11$ (the B$\natural$) to the tonic position. 
\end{example}

We will often denote a mode $(X, \oplus_{\mu_i})$ simply by $(X, \mu_i)$. 

\begin{definition}[Mode Morphism]
\label{def:mode-morphism}
Let $(X, \oplus_{\mu_i})$ and $(X', \oplus_{\mu_{i'}})$ be modes with $|X| = n$ and $|X'| = n'$. A \emph{mode morphism}
\[
\phi : (X, \oplus_{\mu_i}) \longrightarrow (X', \oplus_{\mu_{i'}})
\]
is a function $\phi : X \to X'$ induced by a group homomorphism
\[
\hat{\phi} : \Z_n \longrightarrow \Z_{n'}
\]
that makes the following diagram commute:
\[\begin{tikzcd}
	X && {X'} \\
	\\
	{\Z_n} && {\Z_{n'}}
	\arrow["\phi", from=1-1, to=1-3]
	\arrow["{\mu_i}"', from=1-1, to=3-1]
	\arrow["{\mu_{i'}}", from=1-3, to=3-3]
	\arrow["{\hat{\phi}}"', from=3-1, to=3-3]
\end{tikzcd}\]
The homomorphism $\hat{\phi}$ is of the form
\begin{equation}
\label{eq:hom-condition1}
\hat{\phi}(j) = aj \pmod n',
\end{equation}
where 
\begin{equation}
\label{eq:hom-condition2}
a = \frac{n'}{\gcd(n, n')}.
\end{equation}
\end{definition}

\begin{remark}
The coefficient $a=n'/\gcd(n,n')$ is the smallest positive integer satisfying $na\equiv0\pmod{n'}$. Consequently, $\hat{\phi}$ has image of maximal possible order, namely $\gcd(n,n')$.
\end{remark}

Mode morphisms compose associatively and admit identities, so the class of all modes and their morphisms forms a category, denoted $\Mode$.

\subsection{Scales}
\label{sub:scales}

While it is formally convenient to model a mode as a group structure on a pitch-class set, this formulation is somewhat artificial. In practice, one is not usually interested in adding scale degrees; rather, one cares about the distances—or step relations—between them. The group identity of a mode functions only to designate a distinguished tonic, providing an ``origin'' for the scale. The binary operation $\oplus_{\mu_i}$ then allows us to speak of stepwise motion, typically via subtraction.

If we ``forget the tonic'' $t \in X$ in a mode $(X, \oplus_{\mu_i})$, we are left with a structure known as a \emph{torsor}. Informally, a torsor is like a group without a specified identity element. This more accurately reflects the way scalar relationships are typically understood, namely, as origin-independent. Indeed, chromatic pitch-class space is better modeled as a torsor rather than a group, since no pitch class functions intrinsically as an identity.\footnote{Likewise for chromatic pitch space; see \cite{baez2009torsors}.} When we speak of pitch-class relations, we are describing intervals as relative differences, rather than as coordinates anchored to a fixed reference point.

The torsor structure corresponds to Lewin’s concept of a generalized interval system (GIS), in the sense that a GIS may be recast as a simply transitive group action. Let us therefore recall the definition of a GIS (Definition 2.3.1 in \cite{lewin2007generalized}).

\begin{definition}
\label{def:gis}
A \emph{generalized interval system (GIS)} is a triple $(S, \mathrm{IVLS}, \mathrm{int})$ consisting of:
\begin{enumerate}
\item a set $S$, called the \emph{space};
\item a group $\mathrm{IVLS}$, called the \emph{interval group};
\item a function $\mathrm{int} : S \times S \to \mathrm{IVLS}$ satisfying the following axioms:
	\begin{enumerate}
		\item For all $r, s, t \in S$,
			\[ 
			\mathrm{int}(r, s) \star \mathrm{int}(s, t) = \mathrm{int}(r, t),
			\] 
		where $\star$ denotes the binary operation on $\mathrm{IVLS}$.
		\item For every $s \in S$ and every $i \in \mathrm{IVLS}$, there exists a unique $t \in S$ such that 
			\[ 
			\mathrm{int}(s, t) = i.
			\]
	\end{enumerate}
\end{enumerate}
\end{definition}

A GIS allows one to associate to every pair of elements $s, t \in S$ a unique interval from $s$ to $t$. Lewin then states at the beginning of Chapter 7 of \cite{lewin2007generalized} that an earlier theorem (6.7.1) allows one ``to replace entirely the concept of interval-in-a-GIS by the concept of transposition-operation-on-a-space,'' and thereby to ``replace the idea of GIS structure by the idea of a space $S$ together with'' a simply transitive group action on $S$ \cite[p.\ 157]{lewin2007generalized}. This is precisely a torsor, as defined below.

\begin{definition}[Torsor]
Let $G$ be a group. A \emph{$G$-torsor} is a set $X$ equipped with a simply transitive action $\tau : G \times X \to X$. 
\end{definition}

In our setting, any mode $(X, \oplus_{\mu_i})$ with $|X| = n$ induces a $\Z_n$-torsor:
\begin{equation}
\label{eq:torsor}
\begin{matrix}
\tau : & \Z_n \times X &  \longrightarrow &  X \\
&  \tau(g, x) & \longmapsto &  \mu_i^{-1}(g + \mu_i(x)),
\end{matrix}
\end{equation}
which is in fact the same torsor for every mode on $X$. The action $\tau$ expresses how pitch classes in $X$ are related by their stepwise distances: $\tau(g, x)$ yields the pitch class that lies $g$ steps above $x$.

This leads naturally to the definition of a scale.

\begin{definition}[Scale]
Let $X$ be a pitch-class set with $|X| = n$. The \emph{scale} of $X$ is the $\Z_n$-torsor defined in~\eqref{eq:torsor}. We denote this scale by $(X, \Z_n)$, or simply by $X$ when unambiguous.
\end{definition}

To define morphisms between scales, we recall the general notion of a torsor homomorphism. A \emph{torsor homomorphism} between a $G$-torsor $(X, G)$ and an $H$-torsor $(Y, H)$ is a pair $(f, \varphi)$ consisting of a set function $f : X \to Y$ and a group homomorphism $\varphi : G \to H$ such that
\[
f(g \cdot x) = \varphi(g) \cdot f(x)
\]
for all $g \in G$ and all $x \in X$.

This yields the corresponding notion for scales.

\begin{definition}[Scale Morphism]
\label{def:scale-morphism}
Let $(X, \Z_n)$ and $(X', \Z_{n'})$ be scales. A \emph{scale morphism} is a pair 
\[
(\phi_{i,i'}, \hat{\phi}) : (X, \Z_n) \longrightarrow (X', \Z_{n'}),
\] 
where $\hat{\phi} : \Z_n \to \Z_{n'}$ is the group homomorphism from Definition~\ref{def:mode-morphism} and $\phi_{i,i'} : X \to X'$ is a set map such that the diagram
\[\begin{tikzcd}
	X && {X'} \\
	\\
	{\Z_n} && {\Z_{n'}}
	\arrow["{\phi_{i,i'}}", from=1-1, to=1-3]
	\arrow["{\mu_i}"', from=1-1, to=3-1]
	\arrow["{\mu_{i'}}", from=1-3, to=3-3]
	\arrow["{\hat{\phi}}"', from=3-1, to=3-3]
\end{tikzcd}\]
commutes. Since any choice of modes $(\mu_i, \mu_{i'})$ determines the same underlying torsors, these maps define all scale morphisms from $(X, \Z_n)$ to $(X', \Z_{n'})$.
\end{definition}

Scales and their morphisms thus form a category, denoted $\Scale$.\footnote{We do not engage here with the detailed combinatorial properties of particular scale types---well-formedness, maximal evenness, and related notions---for which there exists a rich and well-developed literature (see, e.g., \cite{carey1989aspects,clampitt2011modes,clough1985variety,clough1986musical,clough1991maximally,clough1999scales,noll2015triads}, among many others).}

\section{Three Classes of Mode Isomomorphisms}
\label{sec:three-classes-of-mode-isomorphisms}

Having defined scales and modes, we now restrict our attention to three particularly musically salient classes of mode isomorphisms. These serve as the foundational morphisms for the hierarchy of fibrations developed in \S~\ref{sec:a-hierarchy-of-fibrations-over-heptatonic-scales}.

Since mode isomorphisms exist only between modes of equal cardinality, we restrict our attention to $7$-note modes in the $12$-tone pitch-class universe. Although the constructions extend naturally to arbitrary values of the mode cardinality $n$ and ambient equal temperament $N$, this specialization keeps the exposition concrete while focusing on the musically important case $(n,N)=(7,12)$, where diatonic harmony in $12$-tone equal temperament lives.

Let $\Mode$ be the category of modes on $\mathbb{Z}_{12}$, and denote by $\Mode_n$ the full subcategory of $n$-element modes. By Definition~\ref{def:mode-morphism} of a mode morphism, every morphism in $\Mode_n$ is an isomorphism, and therefore $\Mode_n$ is a groupoid.

Let
\[
\HMode := \Mode_7
\]
denote the category of heptatonic modes, and let 
\[
U_{\Set} : \HMode \longrightarrow \Set
\]
 be the forgetful functor sending a mode $(X,{\mu_i})$ to its underlying set $X$, and a mode morphism to its underlying set function. Define
\[
\Hept \coloneqq U_{\Set}(\operatorname{Ob}(\HMode)),
\]
the set of heptatonic scales.

Let 
\[
\bfT \coloneqq \{ \bfT_p \mid p \in \Z_{12} \}
\]
 be the group of translations on $\Z_{12}$, and denote the translation orbit by $\Hept/\bfT$.

The remainder of this section is devoted to three particularly important classes of mode isomorphisms: translations, rotations, and deformations. We introduce each class in turn.

\subsection{Translations}
\label{sub:translations}

A particularly significant class of mode isomorphisms is given by those of the form
\[
\phi : (X, {\mu_i}) \longrightarrow (X', {\mu_i})
\]
for which $U_{\Set}(\phi)=\bfT_p$ for some $\bfT_p\in\bfT$. In this case, the two modes are identical up to translation. For example, the translation carrying the C major Ionian mode to the G major Ionian mode preserves the Ionian mode. We call such mode isomorphisms \emph{translations}. Translations preserve the mode type, while changing both the underlying scale and its tonic.

\subsection{Mode Rotations}
\label{sub:mode-rotations}

Another important class of mode isomorphisms consists of those that preserve the underlying scale while changing the chosen mode. These are the isomorphisms
\[
\phi : (X,\mu_i) \longrightarrow (X,\mu_{i+q}),
\]
where $q\in\Z_7$ and addition is taken modulo $7$. For example, the map carrying the Ionian mode $\mu_1$ to the Aeolian mode $\mu_6$ on a fixed diatonic scale $X$ is a mode rotation. We call such isomorphisms \emph{mode rotations}, or simply \emph{rotations}. They form the group
\[
\bfR\coloneqq\{\bfR_q\mid q\in\Z_7\}\cong\Z_7.
\]
Each rotation preserves the underlying scale while changing both the mode type and the distinguished tonic.

\subsection{Deformations}
\label{sub:deformations}

Finally, we consider \emph{deformations}. Consider a mode isomorphism $\phi$ that flattens the 3rd, 6th, and 7th scale degrees of the C major Ionian mode. The resulting scale is still diatonic, but the underlying set map $U_{\Set}(\phi)$ is not a translation, since it does not preserve the intervallic ordering of the mode. Nonetheless, such isomorphisms are musically significant. It is natural to regard the C minor scale as arising from the C major scale by flattening the 3rd, 6th, and 7th scale degrees, rather than by translating the scale up three semitones and then rotating the mode downward two scale degrees to obtain the Aeolian mode.

The generalization of this flattening procedure is achieved by what I call \emph{deformations} of a scale. These permit both flattenings and sharpenings of scale degrees, subject to suitable constraints. To define deformations in full generality, it is first convenient to choose a canonical ordering for each $X \in \Hept$, for which we use the normal order $\mathsf{norm}_X$, as defined in \eqref{eq:normal-order}. We shall also write $\mathsf{norm}_X = (x_0, x_1, \ldots, x_6)$, where $x_i = \mathsf{norm}_X(i)$.

We next introduce the function
\begin{equation}
\label{eq:pc-intervals}
\operatorname{ints} : \Hept \longrightarrow \mathbb{Z}_{12}^7,
\end{equation}
whose coordinates are given by
\[
\operatorname{ints}(X)_i = \mathsf{norm}_X(i+1)-\mathsf{norm}_X(i),
\]
where indices are taken modulo $7$. Thus, 
\[
\operatorname{ints}(X) = (d_0,d_1,\ldots,d_6),
\] 
where $d_i=\operatorname{ints}(X)_i$ is the interval between the $i$th and $(i+1)$th pitch class in the normal order of $X$.

For instance, for the C major scale $X=\{0,2,4,5,7,9,11\}$, we have 
\[
\mathsf{norm}_X=(11,0,2,4,5,7,9),
\] 
and therefore 
\[
\operatorname{ints}(X)=(1,2,2,1,2,2,2).
\]

Using the interval tuple $\operatorname{ints}(X)$, we define two auxiliary functions. First, let
\begin{equation}
\label{eq:sharp-X}
\begin{matrix}
\sharp_X : & [7] & \longrightarrow & \mathcal{P}(\Z_{12}) \\
& i & \longmapsto & \{0,\ldots,d_i-1\},
\end{matrix}
\end{equation}
where $\mathcal{P}$ denotes the powerset operator. Thus $\sharp_X(i)$ consists of the pitch-class intervals strictly smaller than $d_i$. Equivalently, it records the chromatic displacements by which the $i$th pitch class may be sharpened without passing the $(i+1)$th pitch class.

Similarly, define
\begin{equation}
\label{eq:flat-X}
\begin{matrix}
\flat_X : & [7] & \longrightarrow & \mathcal{P}(\Z_{12}) \\
& i & \longmapsto & \{-(d_{i-1}-1),\ldots,0\},
\end{matrix}
\end{equation}
where indices are taken modulo $7$. Then $\flat_X(i)$ consists of the pitch-class intervals by which the $i$th pitch class may be flattened without passing the $(i-1)$th pitch class.

We now provide the definition of deformation spaces. 

\begin{definition}[Deformation Spaces]
\label{def:deformation-spaces}
Let $X \in \Hept$. The \emph{sharp deformation space} of $X$ is the dependent product
\[
\operatorname{Def}_\sharp(X) \coloneqq \prod_{i\in[7]} \sharp_X(i),
\]
and the \emph{flat deformation space} of $X$ is the dependent product
\[
\operatorname{Def}_\flat(X) \coloneqq \prod_{i\in[7]} \flat_X(i).
\]
The \emph{total deformation space} of $X$ is
\[
\operatorname{Def}(X)
\coloneqq
\operatorname{Def}_\sharp(X)
\cup
\operatorname{Def}_\flat(X).
\]

The elements of $\operatorname{Def}_\sharp(X)$ and $\operatorname{Def}_\flat(X)$ are called the \emph{sharp deformations} and \emph{flat deformations} of $X$, respectively. Collectively, the elements of $\operatorname{Def}(X)$ are called the \emph{deformations} of $X$.
\end{definition}

We define the  \emph{magnitude} of a deformation $\Delta \in \Def(X)$ as
\begin{equation}
\operatorname{mag}(\Delta) = \left| \sum_{i=0}^6 \Delta_i \right|
\end{equation}

Figure~\ref{fig:sharp-flat-deformations} anticipates the discussion in \S~\ref{subsub:deformations-as-mode-isomorphisms}, where we show that deformations induce transformations of pitch-class sets.

\begin{figure}[t]
    \centering

    \begin{subfigure}[b]{1\textwidth}
        \centering
        \includegraphics[width=\linewidth]{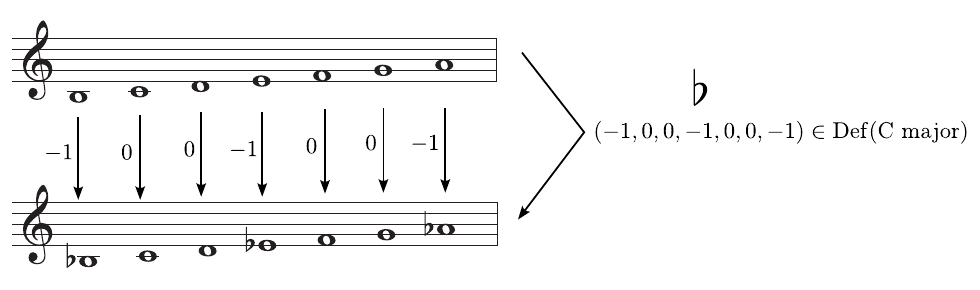}
        \caption{A flat deformation of the C-major diatonic scale.}
        \label{fig:flat-deformation}
    \end{subfigure}
    \hfill
    \begin{subfigure}[b]{1\textwidth}
        \centering
        \includegraphics[width=\linewidth]{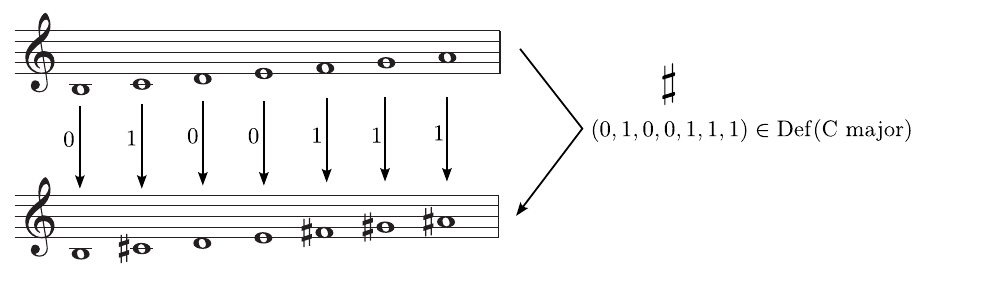}
        \caption{A sharp deformation of the C-major diatonic scale.}
        \label{fig:sharp-deformation}
    \end{subfigure}

    \caption{Examples of flat and sharp deformations acting on the C-major diatonic scale.}
    \label{fig:sharp-flat-deformations}
\end{figure}

\subsubsection{Order Structure on Deformation Spaces}
\label{subsub:order-structure-on-deformation-spaces}

The space $\operatorname{Def}(X)$ carries a natural partial order. For
$\Delta,\Delta'\in\operatorname{Def}(X)$,
\[
\Delta\le\Delta'
\iff
\Delta_i\le\Delta_i'
\quad\text{for every }i.
\]
The subspaces $\operatorname{Def}_\sharp(X)$ and $\operatorname{Def}_\flat(X)$ inherit this partial order.

\begin{proposition}
The spaces $\operatorname{Def}_\sharp(X)$ and $\operatorname{Def}_\flat(X)$ are order anti-isomorphic.
\end{proposition}

\begin{proof}
First, we will show that there exists a map $\Gamma:\operatorname{Def}_\sharp(X)\to\operatorname{Def}_\flat(X)$ that is order-reversing. We then show that $\Gamma$ is injective and surjective, and hence bijective.

Let  $\Delta = (\Delta_0, \ldots, \Delta_6) \in \operatorname{Def}_\sharp(X)$. To every sharpening of the $i$th pitch class of $X$ by $\Delta_i$ there corresponds a flattening of the $(i+1)$th pitch class by the same amount. Accordingly, define
\[
\Gamma : \operatorname{Def}_\sharp(X)
\longrightarrow
\operatorname{Def}_\flat(X)
\]
by
\[
(\Gamma(\Delta))_{i+1}=-\Delta_{i},
\]
where indices are taken modulo $7$. Clearly $\Gamma$ is order-reversing, since the more a pitch class is sharpened, the more the adjacent pitch class is flattened.

We now show that $\Gamma$ is bijective. First, suppose $\Delta,\Delta'\in\operatorname{Def}_\sharp(X)$ satisfy $\Gamma(\Delta)=\Gamma(\Delta')$. Then
\[
(-\Delta_6,-\Delta_0,\ldots,-\Delta_5)
=
(-\Delta'_6,-\Delta'_0,\ldots,-\Delta'_5),
\]
and hence $-\Delta_i=-\Delta'_i$ for every $i$. Therefore $\Delta=\Delta'$, so $\Gamma$ is injective.

Now let $\Delta'=(\Delta'_0,\ldots,\Delta'_6)\in\operatorname{Def}_\flat(X)$. To every flattening of the $i$th pitch class of $X$ by $\Delta'_i$ there corresponds a sharpening of the $(i-1)$th pitch class by the same amount. Thus $\Delta=(-\Delta'_1,-\Delta'_2,\ldots,-\Delta'_0)$ belongs to $\operatorname{Def}_\sharp(X)$ and satisfies $\Gamma(\Delta)=\Delta'$. Hence $\Gamma$ is surjective.
\end{proof}

\subsubsection{Deformations as Mode Isomorphisms}
\label{subsub:deformations-as-mode-isomorphisms}

We now show how a deformation of $X \in \Hept$ extends canonically to every scale in its translation orbit $[X]$. Let $\Delta=(\Delta_0,\ldots,\Delta_{n-1})\in\operatorname{Def}(X)$. Then $\Delta$ determines a function
\begin{equation}
\begin{matrix}
\Delta^X : & X & \longrightarrow & \Z_{12} \\
& x_i & \longmapsto & x_i+\Delta_i,
\end{matrix}
\end{equation}
where $x_i=\mathsf{norm}_X(i)$.

Since translations preserve the interval tuple $\operatorname{ints}(X)$, they also preserve the deformation spaces. Hence, if $X'\in[X]$, then $\operatorname{Def}(X')=\operatorname{Def}(X)$, and the same deformation vector $\Delta$ induces a map
\[
\begin{matrix}
\Delta^{X'} : & X' & \longrightarrow &\Z_{12},\\
& x'_i & \longmapsto & x'_i+\Delta_i.
\end{matrix}
\]
Accordingly, $\Delta$ induces a map on the translation orbit $\Delta : [X] \to \Hept$, given by
\[
\Delta(X')
=
\{\Delta^{X'}(x) \mid x\in X'\}.
\]
We may therefore denote the deformation space over the whole translation orbit $[X]$ as $\Def(X)$.

It remains to show that the image of $\Delta$ is itself a translation orbit. In particular, we show that $\Delta$ commutes with translations.

\begin{proposition}
\label{prop:equivariance}
Let $X\in\Hept$ and let $\Delta\in\Def(X)$. Then $\Delta$ is
$\mathbf T$-equivariant; that is,
\[
\Delta\circ\mathbf T_p
=
\mathbf T_p\circ\Delta
\]
for every $p\in\mathbb Z_{12}$.
Consequently,
\[
\operatorname{im}(\Delta)
=
[\Delta(X)].
\]
\end{proposition}

\begin{proof}
For every $p\in\mathbb Z_{12}$, we have $\Delta(\mathbf T_p(X)) = \mathbf T_p(\Delta(X))$, since if $x_i\in X$, then
\[
\Delta(x_i+p) = 
x_i+p+\Delta_i =
\mathbf T_p(\Delta(x_i)).
\]
Hence
\[
\operatorname{im}(\Delta)
=
\{\mathbf T_p(\Delta(X))\mid p\in\mathbb Z_{12}\}
=
[\Delta(X)],
\]
which is a translation orbit.
\end{proof}

\begin{definition}[Deformation Symmetry]
Let $X\in\Hept$ and let $\Delta\in\operatorname{Def}(X)$. If $[\Delta(X)]=[X]$, then $\Delta$ is called a \emph{deformation symmetry} of the translation orbit $[X]$.
\end{definition}

Deformations $\Delta\in\operatorname{Def}(X)$ give rise to mode isomorphisms of the form
\[
\phi : (X,{\mu_i})
\longrightarrow
(\Delta(X),{\mu_j}).
\]
We call such mode isomorphisms \emph{deformations}. Deformations act by changing both the underlying scale and the mode type. Table~\ref{tab:mode-isomorphisms} compares this behavior with the two previously introduced classes of mode isomorphism.

\begin{table}[h]
\centering
\begin{threeparttable}
\begin{tabular}{lccc}
\toprule
 & \textbf{Mode} & \textbf{Scale} & \textbf{Tonic} \\
\midrule
$\mathbf T_p \in \bfT$ & Preserved & Changed   & Changed \\
$\mathbf R_q \in \bfR$ & Changed   & Preserved & Changed \\
$\Delta \in \Def(X)$      & Changed   & Changed   & Preserved\tnote{$\dagger$} \\
\bottomrule
\end{tabular}
\begin{tablenotes}
\small
\item[$\dagger$] In the common case; deformations need not preserve the tonic in general.
\end{tablenotes}
\end{threeparttable}
\caption{The three classes of mode isomorphisms and the data they preserve.}
\label{tab:mode-isomorphisms}
\end{table}

\begin{example}
Let $X=\{0,2,4,5,7,9,11\}$ be the C major scale, so that $(X,\mu_1)$ is the C Ionian mode. Let $\Delta$ be the flat deformation that sends
\[
4\mapsto3,\qquad
9\mapsto8,\qquad
11\mapsto10,
\]
and fixes every other pitch class. Then the induced deformation
\[
\phi : (X,\mu_1)
\longrightarrow
(\Delta(X),\mu_6)
\]
carries the C Ionian mode to the C Aeolian mode. 
\end{example}

For each mode $(X,\mu_i)$, let $\tau : \Z_7 \times X \longrightarrow X$ denote the corresponding torsor action defined in \eqref{eq:torsor}. We write $q+x\coloneqq\tau(q,x)$. Thus $q+x$ is the pitch class lying $q$ steps above $x$. 

\begin{proposition}
Let $(X,\mu_i)$ and $(Y,\mu_j)$ be modes, and let $\Delta:X\to Y$ be a deformation satisfying $\mu_j\circ\Delta=\mu_i$. Then $\Delta$ induces the translation
\begin{equation}
\label{eq:delta}
\begin{matrix}
\delta : & \Z_7& \longrightarrow& \Z_7,\\
& n & \longmapsto & n+(j-i),
\end{matrix}
\end{equation}
on mode indices. Equivalently, $\mu_{\delta(n)}\circ\Delta=\mu_n$ for every $n\in\Z_7$.
\end{proposition}

\begin{proof}
Let $q\in\Z_7$. For every $x\in X$,
\begin{align*}
\mu_{i+q}(x)
&=\mu_i(q+x) \\
&=\mu_j\bigl(\Delta(q+x)\bigr) \\
&=\mu_j\bigl(q+\Delta(x)\bigr) \\
&=\mu_{j+q}\bigl(\Delta(x)\bigr).
\end{align*}
Hence $\mu_{j+q}\circ\Delta=\mu_{i+q}$. Therefore the induced map on mode indices is the translation $n\mapsto n+(j-i)$.
\end{proof}

\begin{example}[Iterated Parallel-Minor Deformation]
\label{ex:iterated-parallel-minor-deformation}
Let $X=\{0,2,4,5,7,9,11\}$ be the C-major scale, and consider the deformation
\[
\Delta=(-1,0,0,-1,0,0,-1).
\]
As in the preceding example, this deformation carries the C Ionian mode $(X,\mu_1)$ to the C Aeolian mode $(\mathbf T_3(X),\mu_6)$. Thus, on the scale component, $\Delta$ acts by the chromatic translation $\mathbf T_3$, while its induced action on mode indices is $\delta(i)=i+5\pmod 7$.

The deformation may therefore be iterated. If
\[
X_k\coloneqq\mathbf T_{3k}(X)
\qquad\text{and}\qquad
i_k\coloneqq i+5k\pmod 7,
\]
then
\[
\Delta^k(X,\mu_1)=(X_k,\mu_{1_k}).
\]
The first seven iterations are shown in Figure~\ref{fig:iterated-parallel-minor-deformation}. Explicitly, the mode indices follow the cycle
\[
1\longmapsto6\longmapsto4\longmapsto2
\longmapsto7\longmapsto5\longmapsto3
\longmapsto1,
\]
whereas the scale components follow the sequence
\[
X\longmapsto
\mathbf T_3(X)\longmapsto
\mathbf T_6(X)\longmapsto
\mathbf T_9(X)\longmapsto
X\longmapsto
\mathbf T_3(X)\longmapsto
\mathbf T_6(X)\longmapsto
\mathbf T_9(X).
\]
Consequently, after seven iterations the mode index has returned to its initial value, but the scale component has not. This reflects the fact that the induced translation $\delta(i)=i+5$ has order $7$ on the mode indices, whereas $\mathbf T_3$ has order $4$ in the group $\bfT$ of chromatic translations; therefore the deformation itself has order $\operatorname{lcm}(7,4)=28$ on the corresponding modes.

\begin{figure}[h!]
\centering
\includegraphics[width=0.6\textwidth]{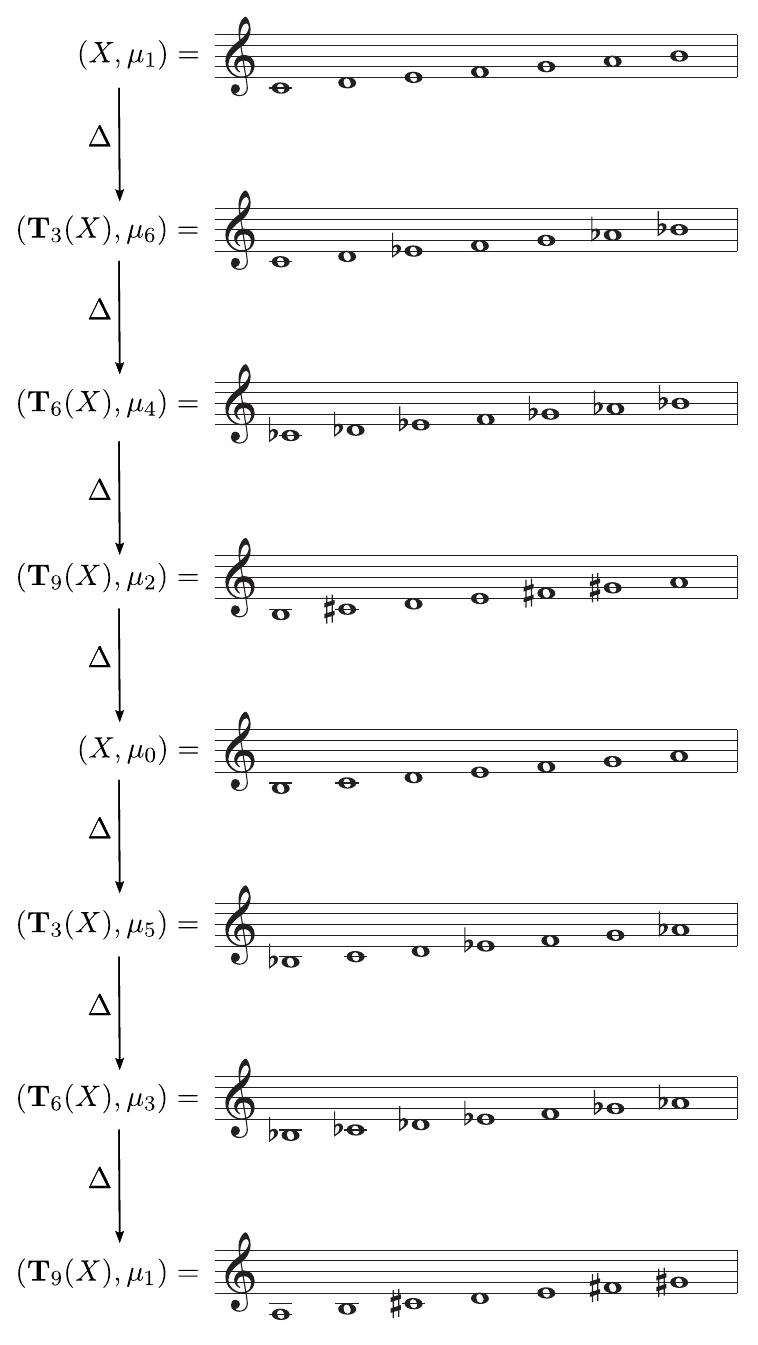}
\caption{The first seven iterations of the deformation $\Delta=(-1,0,0,-1,0,0,-1)$, beginning with the C Ionian mode $(X,\mu_1)$. At each step, the underlying scale is translated upward by three semitones, while the mode index is translated by five elements of $\Z_7$. Thus the $k$th iteration of $\Delta$ is $(\mathbf T_{3k}(X),\mu_{1+5k})$. After seven iterations, the mode index returns to $1$, while the scale component is $\mathbf T_9(X)$.}
\label{fig:iterated-parallel-minor-deformation}
\end{figure}
\end{example}

\subsubsection{The Deformation Quiver}
\label{subsub:the-deformation-quiver}

Translations and mode rotations are defined uniformly on every object in the category $\HMode$ of heptatonic modes with mode isomorphisms. Thus each translation $\bfT_p$ and each mode rotation $\bfR_q$ determines an endofunctor
\[
\bfT_p,\bfR_q:\HMode\longrightarrow\HMode.
\]
For example, the translation $\bfT_p$ acts on objects by $(X,\mu_i)\mapsto(\bfT_p(X),\mu_i)$, and sends every mode isomorphism $\phi:(X,\mu_i) \to(Y,\mu_j)$ to the mode isomorphism
\[
\bfT_p(\phi) : 
(\bfT_p(X),\mu_i)
\longrightarrow
(\bfT_p(Y),\mu_j),
\]
so that the diagram
\[\begin{tikzcd}
	{(X, \mu_i)} && {(Y, \mu_j)} \\
	\\
	{(\bfT_p(X), \mu_i)} && {(\bfT_p(Y), \mu_j)}
	\arrow["\phi", from=1-1, to=1-3]
	\arrow["{\bfT_p}"', from=1-1, to=3-1]
	\arrow["{\bfT_p}", from=1-3, to=3-3]
	\arrow["{\bfT_p(\phi)}"', from=3-1, to=3-3]
\end{tikzcd}\]
commutes. An analogous construction applies to the rotations $\bfR_q$.

A deformation $\Delta\in\Def(X)$, however, is defined only relative to a single translation orbit of scales and therefore does not extend naturally to an endofunctor on $\HMode$. Moreover, deformations do not admit a canonical composition. Given $\Delta : X \to Y$ and $\Delta' : Y \to Z$, there need not exist a distinguished deformation $X \to Z$ representing their composite. Thus, unlike translations and rotations, whose morphisms form categories of modes, deformations do not endow a set of isomorphic modes with the structure of a category. Instead, they are more naturally viewed as the arrows of a quiver. 

\begin{definition}[Deformation Quiver]
Define
\[
\mathcal{D} \coloneqq  \coprod_{X\in\mathcal H}\Def(X).
\]
Each deformation has a source and target scale, giving maps 
\[
s,t:\mathcal{D} \longrightarrow \mathcal H, \qquad s(\Delta) = X, \quad t(\Delta) = Y.
\]
The quadruple  $Q \coloneqq (\mathcal H, \mathcal{D},s,t)$ therefore defines a quiver, which we call the \emph{deformation quiver} over $\Hept$.
\end{definition}

Since deformations in $\Def(X)$ extend to all $X' \in [X]$, it is natural to pass to the quotient by the translation action. Writing $\mathcal H/\mathbf T$ for the set of translation orbits of heptatonic scales, define
\[
\mathcal{D}/\mathbf T
=
\coprod_{[X]\in\mathcal H/\mathbf T} \Def(X),
\]
together with the induced source and target maps
\[
s/\mathbf T,\;
t/\mathbf T:
\mathcal{D}/\mathbf T
\longrightarrow
\mathcal H/\mathbf T.
\]
The resulting quiver $Q/\bfT$ records the deformations between translation classes of scales. In \S~\ref{sec:a-hierarchy-of-fibrations-over-heptatonic-scales} we will extend $Q/\bfT$ to a category $\HDef$, which underlies a hierarchy of fibrations.

\subsubsection{A Constrained Construction Scheme for Compositional Applications}
\label{subsub:a-constrained-construction-scheme}

The full deformation quiver is extremely large and therefore of limited practical use as a compositional tool. In practice, one may instead work with subquivers obtained by imposing musically salient constraints. For example, one may restrict to deformations of maximum magnitude $\epsilon$, or discard translation orbits not of interest for a given application.  

The following construction scheme is one that I have employed in my own compositional practice.
\begin{enumerate}
    \item Choose a collection $H \subset \Hept/\bfT$ of translation orbits satisfying desired intervallic constraints.
    \item Fix a maximum deformation magnitude $\epsilon$.
\end{enumerate}
These two choices give rise to a canonical quiver construction. First, rather than labeling arrows by deformations, we first record only the existence of sufficiently small deformations. Thus we introduce an arrow from $[X]$ to $[Y]$ whenever there exists a deformation $\Delta \in \Def(X)$ whose magnitude satisfies $\operatorname{mag}(\Delta)\le\epsilon$. In other words, we connect two translation orbits whenever they differ by a sufficiently small deformation. The resulting quiver therefore records which translation orbits are ``close'' with respect to the bound $\epsilon$.

For translation orbits $[X],[Y]\in H$, say that $[Y]$ is \emph{$\epsilon$-reachable} from $[X]$ if there exists $\Delta\in\Def(X)$ such that $\operatorname{mag}(\Delta)\leq\epsilon$ and $[\Delta(X)]=[Y]$. This defines a relation
\[
\mathscr{R}\ \subseteq\ H\times H,
\qquad
[X]\,\mathscr{R}\,[Y]
\iff
[Y]\text{ is $\epsilon$-reachable from }[X].
\]

\begin{definition}
The \emph{$\epsilon$-adjacency quiver} of $H$ is the quiver $H^\epsilon=(H,A,s,t)$, where the arrow set is
\[
A\coloneqq
\{\,([X],[Y])\in H\times H \mid [X]\,\mathscr{R}\,[Y]\,\},
\]
and the source and target maps $s,t:A\to H$ are the coordinate projections,
\[
s([X],[Y])=[X],
\qquad
t([X],[Y])=[Y].
\]
\end{definition}

Thus there is at most one arrow between any ordered pair of translation orbits. The quiver $H^\epsilon$ therefore records only the existence of admissible deformations, without distinguishing between different deformations realizing the same adjacency. Figure~\ref{fig:heptatonic-deformations-constrained} illustrates the $\epsilon$-adjacency quiver $H^\epsilon$, where $H\subset\Hept/\bfT$ consists of the translation orbits whose successive scale steps are at most three semitones, and $\epsilon=1$. Thus an arrow records the existence of a deformation from one translation orbit to another obtained by raising or lowering a single scale degree by one semitone.

\begin{sidewaysfigure}
    \centering
    \includegraphics[width=\textheight]{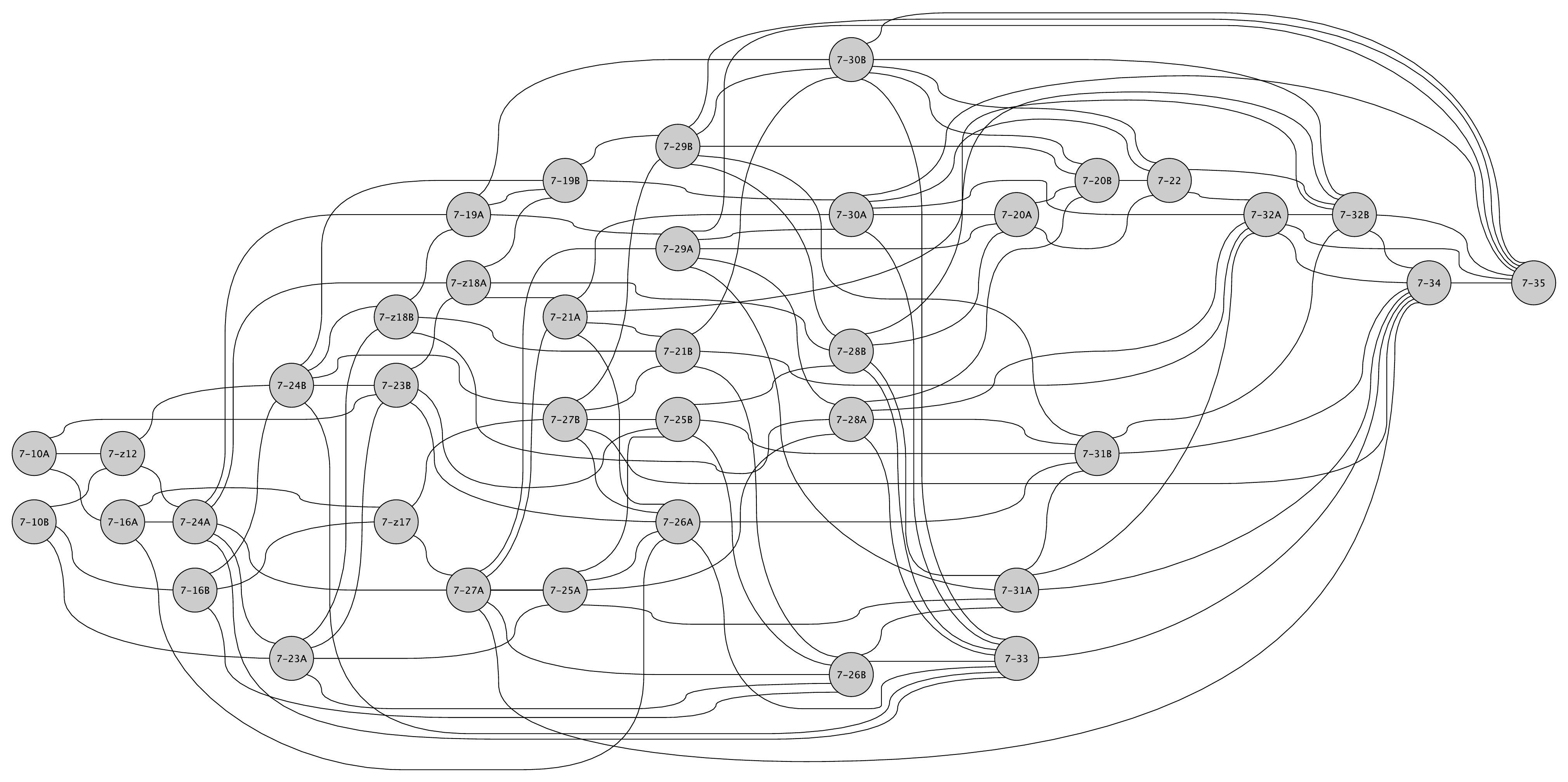}
    \caption{The $\epsilon$-adjacency quiver $H^\epsilon$ for heptatonic scales with $\epsilon=1$. Vertices are translation orbits whose successive scale steps are all at most three semitones. Two vertices are joined by an arrow whenever they are related by a deformation of magnitude at most $1$, that is, by raising or lowering a single pitch class by one semitone.}
    \label{fig:heptatonic-deformations-constrained}
\end{sidewaysfigure}

\subsubsection{Fibered Deformation Quivers}
\label{subsub:fibered-deformation-quivers}

The $\epsilon$-adjacency quiver $H^\epsilon$ records only which pairs of translation orbits are sufficiently close. One may now enlarge this quiver by replacing each arrow with a collection of deformations between the corresponding translation orbits. For example, in common-practice harmony, one may take $H$ to consist of the translation orbits of the diatonic, harmonic minor, and melodic minor scales. These three translation orbits are pairwise connected in $H^1$, since each pair differs by a deformation of magnitude $1$. Relative to the Aeolian mode, the harmonic minor is obtained by raising the seventh degree by one semitone; relative to the Ionian mode, the melodic minor is obtained by lowering the third degree by one semitone; and relative to the harmonic minor, the melodic minor is obtained by raising the sixth degree by one semitone.

However, the deformations used in practice are not restricted to those of magnitude at most $\epsilon$. For example, within the natural minor mode one may raise both the sixth and seventh degrees to obtain the melodic minor mode, producing a deformation of magnitude $2$. Likewise, from the Ionian mode one may obtain the parallel minor by lowering the third, sixth, and seventh degrees, producing a deformation of magnitude $3$. The role of the $\epsilon$-adjacency quiver is therefore not necessarily to constrain the deformations themselves, but to constrain which pairs of translation orbits may be related by deformation. Once an edge of $H^\epsilon$ has been chosen, deformations of arbitrary magnitude between the corresponding modes may be considered. Hence we may enlarge the $\epsilon$-adjacency quiver by assigning to each of its arrows a collection of deformations realizing the corresponding adjacency.

Let $\Def(X,Y)$ denote the set of deformations $\Delta\in\Def(X)$ such that $\Delta(X)\in[Y]$. We then specify, for each arrow $a:[X]\to[Y]$ of $H^\epsilon$, a subset
\[
\mathcal E(a) \subseteq \Def(X,Y),
\]
so that $\mathcal E$ is an indexed family of sets over $A$. The expanded arrow set is the dependent sum of this family,
\[
A^{\mathcal E}
\coloneqq
\sum_{a\in A}\mathcal{E}(a),
\]
whose elements are pairs $(a,\Delta)$ consisting of an arrow $a$ of $H^\epsilon$ together with a deformation belonging to $\mathcal{E}(a)$. 

There is a canonical projection
\[
\pi : A^\mathcal{E} \longrightarrow A
\]
sending $(a,\Delta)$ to $a$, and the fiber over an adjacency arrow $a$ is therefore precisely the chosen collection $\mathcal{E}(a)$ of deformations. Thus every arrow of the $\epsilon$-adjacency quiver is replaced by a fiber of deformation arrows between the same pair of translation orbits.

\begin{definition}
The \emph{$\mathcal{E}$-expansion} of the $\epsilon$-adjacency quiver $H^\epsilon$ is the quiver
\[
H^{\epsilon,\mathcal{E}}\coloneqq(H,A^\mathcal{E},s_\mathcal{E},t_\mathcal{E}),
\]
where
\[
s_\mathcal{E}(a,\Delta)=s(a),
\qquad
t_\mathcal{E}(a,\Delta)=t(a).
\]
\end{definition}

\begin{example}
Let $H^\epsilon=(H,A,s,t)$ be the $\epsilon$-adjacency quiver of Figure~\ref{fig:heptatonic-deformations-constrained}. Define
\[
\mathcal E(a)
=
\{\Delta\in\Def(s(a),t(a))\mid \operatorname{mag}(\Delta)\le3\}
\]
for each $a \in A$, so that $\mathcal E(a)$ consists of all deformations between the corresponding translation orbits whose magnitude is at most $3$.
\end{example}

\section{Orbit Covers}
\label{sec:orbit-covers}

In \cite{flieder2026scales}, orbit covers were introduced as a distinguished class of coverings of a scale, namely those obtained by translating a fixed subset under the action of the scale's translation group (Figure~\ref{fig:orbit-cover-examples}). Although that paper classified orbit covers, it did not define morphisms between them. Consequently, it did not endow orbit covers with the structure of a category.

\begin{figure}[h!]
\centering

\begin{subfigure}[t]{0.7\textwidth}
    \centering
    \includegraphics[width=\linewidth]{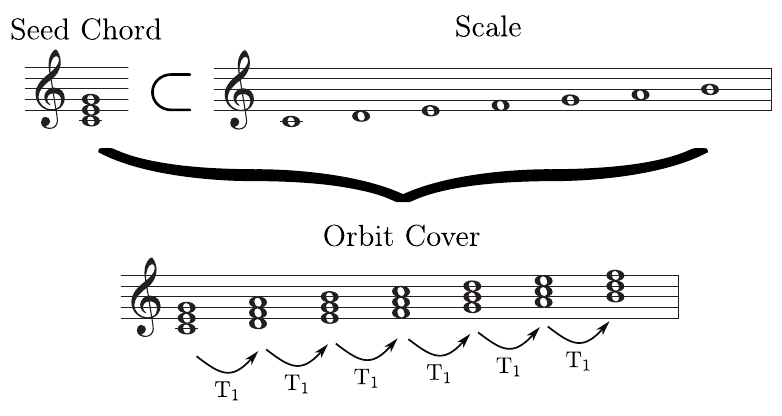}
    \caption{The orbit cover of the C-major scale generated by a tertian triad.}
\end{subfigure}
\hfill
\begin{subfigure}[t]{0.7\textwidth}
    \centering
    \includegraphics[width=\linewidth]{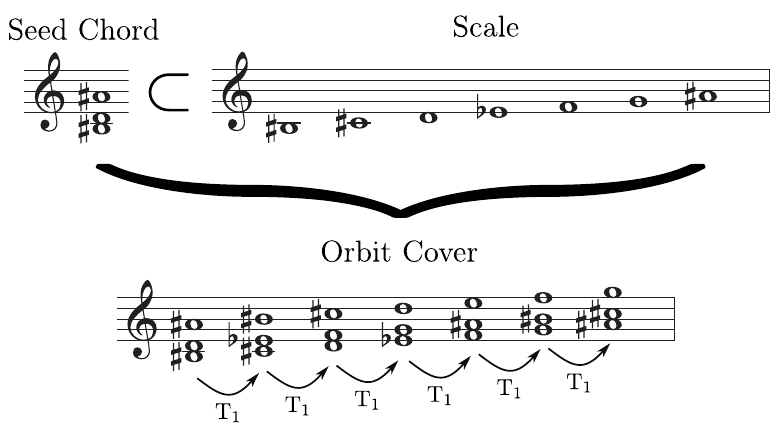}
    \caption{An orbit cover of a non-diatonic heptatonic scale generated by the subset consisting of the first, third, and seventh scale degrees.}
\end{subfigure}

\caption{Examples of scalar orbit covers. In each case, a subset (labeled ``seed chord'') of the scale is translated through the action of the scale's translation group, producing an orbit of subsets that covers the scale.}
\label{fig:orbit-cover-examples}
\end{figure}

The main subtlety in forming such a category is that scale and mode morphisms preserve the cyclic ordering of a scale, whereas orbit cover morphisms arise from arbitrary group homomorphisms of the scale-degree group and need not preserve this ordering. Consequently, orbit covers possess a notion of morphism that differs from that of scales or modes. We begin by defining the category of scalar orbit covers, and then show how to incorporate modal information.

\subsection{The Category of Scalar Orbit Covers}
\label{sub:the-category-of-scalar-orbit-covers}

In \cite{flieder2026scales}, orbit covers were classified with respect to their common-tone structure. Our objective here is to endow scalar orbit covers with the structure of a category $\ScaleOrbCov$ by introducing an appropriate notion of morphism. One consequence, established as Theorem~\ref{th:nerve-isomorphism} below, is that the common-tone structure exhibited by the chords of an orbit cover is captured by the isomorphism classes in $\ScaleOrbCov$. This recovers, in categorical language, the principal classification result of \cite[\S~4.2]{flieder2026scales}.

\begin{definition}[Scalar Orbit Cover]
Let $X$ be a scale with $|X| = n$, and let $\varnothing\neq A\subseteq X$. The \emph{scalar orbit cover} generated by $A$ is the collection
\[
X^{(A)} = \{\T_i(A)\mid i\in\Z_n\}.
\]
\end{definition}

To define morphisms between scalar orbit covers, note that every group homomorphism $h:\Z_n\to\Z_m$ induces, after choosing modal orderings $\mu_i:X\to\Z_n$ and $\mu_j:Y\to\Z_m$, a unique map $h_{i,j}:X\to Y$ making the diagram
\begin{equation}
\label{eq:h_ij}
\begin{tikzcd}
X && Y\\
\\
\Z_n && \Z_m
\arrow["{h_{i,j}}", from=1-1, to=1-3]
\arrow["{\mu_i}"', from=1-1, to=3-1]
\arrow["{\mu_j}", from=1-3, to=3-3]
\arrow["{ h}"', from=3-1, to=3-3]
\end{tikzcd}
\end{equation}
commute. This map transports $h$---a homomorphism of scale-degree groups---into the corresponding group homomorphism between the modes $(X,\oplus_{\mu_i})$ and $(Y,\oplus_{\mu_j})$.

\begin{remark}
Fix a modal ordering $\mu_i:X\to\Z_n$, and let $\mu:Y\to\Z_m$ denote the normal modal ordering on $Y$. Every modal ordering on $Y$ is uniquely of the form $\mu_j=\T_j\circ\mu$ for some $j\in\Z_m$ (Definition~\ref{def:mode}). Consequently, the data of a group homomorphism $h:\Z_n\to\Z_m$ together with a target modal ordering $\mu_j$ is equivalent to the data of an affine map
\[
a:\Z_n\longrightarrow \Z_m,\qquad a(x)=h(x)+j.
\]
\end{remark}

\begin{definition}[Morphisms of Scalar Orbit Covers]
\label{def:morphisms-Scaorb}
Let $X^{(A)}$ and $Y^{(B)}$ be scalar orbit covers. A \emph{morphism of scalar orbit covers} 
\[
h_{i,j}:X^{(A)} \longrightarrow Y^{(B)}
\] 
is such an induced map $h_{i,j} : X \to Y$ satisfying the property that, for every chord $A'\in X^{(A)}$, there exists a chord $B'\in Y^{(B)}$ such that $h_{i,j}(A')\subseteq B'$.
\end{definition}

Composition is associative, and identity maps are again morphisms. Scalar orbit covers therefore form a category, denoted $\ScaleOrbCov$.

\begin{example}
Let $X=\{0,2,4,5,7,9,11\}$ be the C-major scale. Let $A=\{0,4,7\}$ be the C-major triad, and let $B=\{0,5,9,11\}$. Define the group automorphism
\[
\begin{matrix}
h : & \Z_7 & \longrightarrow & \Z_7 \\
& x & \longmapsto & 5x \pmod 7.
\end{matrix}
\]
Then $h$ induces the map
\[
h_{1,1}:X\longrightarrow X
\]
between the Ionian modes given by
\[
\begin{matrix}
X & \xlongrightarrow{h_{1,1}} & X \\
0 & \longmapsto & 0 \\
2 & \longmapsto & 9 \\
4 & \longmapsto & 5 \\
5 & \longmapsto & 2 \\
7 & \longmapsto & 11 \\
9 & \longmapsto & 7 \\
11 & \longmapsto & 4.
\end{matrix}
\]
The image of the seed chord is
\[
h_{1,1}(A)=\{0,5,11\},
\]
which is contained in the subset $B=\{0,5,9,11\}$. Consequently,
\[
h_{1,1}:X^{(A)}\longrightarrow X^{(B)}
\]
extends to a morphism of scalar orbit covers, since for every chord $A'\in X^{(A)}$ there exists a chord $B'\in X^{(B)}$ such that $h_{1,1}(A')\subseteq B'$.

Figure~\ref{fig:scalar-orbit-cover-morphism} illustrates this morphism in musical notation. For each translated tertian triad in the source orbit cover, its image is contained in the corresponding translated quartal tetrachord of the target orbit cover.
\end{example}

\begin{figure}[ht]
\centering
\includegraphics[width=0.6\textwidth]{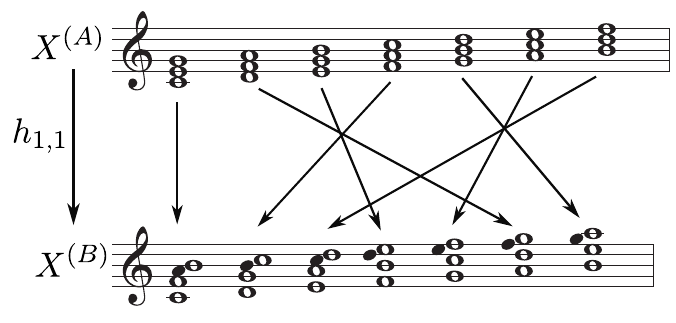}

\caption{A morphism of scalar orbit covers induced by the automorphism $k\mapsto5k$ of $\Z_7$. For each $k\in\Z_7$, the $k$th chord of the tertian orbit cover $X^{(A)}$ is mapped into the $(5k)$th chord of the quartal orbit cover $X^{(B)}$, where indices are taken modulo $7$. The filled noteheads indicate pitches of the target chord that are not contained in the image of the corresponding source chord. Thus, the image of each source chord is contained in its target chord, illustrating the defining condition $h_{1,1}(A_k)\subseteq B_{5k}$.}
\label{fig:scalar-orbit-cover-morphism}
\end{figure}

\subsection{The Category of Modal Orbit Covers}
\label{sub:the-category-of-modal-orbit-covers}

A scalar orbit cover records only the collection of translated chords, forgetting the choice of tonic and hence the underlying modal structure. To recover this information, we introduce \emph{modal orbit covers}, which enrich scalar orbit covers by equipping them with a mode.

\begin{definition}[Modal Orbit Cover]
Let $(X,\mu_i)$ be a mode with tonic $t\in X$, and let $\varnothing\neq A\subseteq X$ with $t\in A$. The corresponding \emph{modal orbit cover} is the pair $(X^{(A)},\mu_i)$.
\end{definition}

Morphisms of modal orbit covers are given in the same way as morphisms of scalar orbit covers (Definition~\ref{def:morphisms-Scaorb}).  Therefore modal orbit covers form a category, denoted $\ModeOrbCov$.

There are canonical forgetful functors
\[
U_{\ScaleOrbCov}:\ModeOrbCov\longrightarrow\ScaleOrbCov
\]
and
\[
U_{\Mode}:\ModeOrbCov\longrightarrow\Mode,
\]
which forget the chosen mode and the orbit cover, respectively.

\subsection{Rooted Orbit Covers}
\label{sub:rooted-orbit-covers}

The requirement that $t\in A$ endows every modal orbit cover with a canonical \emph{rooting}, which assigns a distinguished root pitch class to each translate $\T_i(A)$ in the orbit cover. This rooting further equips every chord with canonical degree and inversion data.

As shown in \cite{flieder2026scales}, $k$-chordal scalar orbit covers of an $n$-note scale are classified by interval compositions
\begin{equation}
\label{eq:int-comp-type}
\mathbf C(n,k)
=
\left\{
(\mathbf c_1,\ldots,\mathbf c_k) \in \Z_{>0}^k
\;\middle|\;
\sum_{j=1}^k \mathbf c_j=n
\right\}
\end{equation}
modulo cyclic rotation.\footnote{This classification has precedent in the work of Clough \cite{clough1979aspects}.} For a modal orbit cover $(X^{(A)},\mu_i)$, the subset $\mu_i(A)\subseteq\Z_n$ contains $0$. Traversing its elements in positive cyclic order beginning at $0$ therefore determines a unique interval composition $\mathbf c\in\mathbf C(n,k)$. Writing $\mathbf c=(\mathbf c_1,\ldots,\mathbf c_k)$, let
\begin{equation}
\label{eq:part-sums}
S_b=\sum_{j=1}^b\mathbf c_j,
\qquad
S_0=0.
\end{equation}
The interval composition $\mathbf c$ determines the chord
\begin{equation}
\label{eq:int-comp}
\hat{\mathbf c}(x)
=
\{ \T_{S_b}(x) \mid 0\le b<k \},
\end{equation}
and hence a bijection
\[
\begin{matrix}
\hat{\mathbf c} : & X& \longrightarrow & X^{(A)},\\
& x& \longmapsto & \hat{\mathbf c}(x).
\end{matrix}
\]
The inverse of this bijection gives the \emph{chordal degree map}
\begin{equation}
\label{eq:chordal-degree}
\deg
=
\mu_i\circ\hat{\mathbf c}^{-1} : X^{(A)} \longrightarrow \Z_n.
\end{equation}

Lastly, the interval composition $\mathbf c$ determines a canonical ordering of every chord,
\[
\bigl( x, \T_{S_1}(x), \dots, \T_{S_{k-1}}(x) \bigr),
\]
whose first element is the distinguished root. The inversions of a rooted chord are therefore obtained by cyclically rotating this ordered tuple, agreeing with the classical notion of chord inversion in tonal harmony.

\begin{example}
Let $X=\{0,2,4,5,7,9,11\}$ be the C-major scale, and let $A=\{0,4,7\}$ be the C-major triad. Then $(X^{(A)},\mu_1)$ is the C-major Ionian tertian orbit cover. The subset $\mu_1(A)=\{0,2,4\}\subset\Z_7$ has consecutive intervals $2$, $2$, $3$, and therefore determines the interval composition $\mathbf c=(2,2,3)$. Consequently, each rooted tertian triad with root note $x\in X$ is given by
\[
\hat{\mathbf c}(x)
=
\{x,\T_2(x),\T_4(x)\}.
\]

Now let $B=\{0,4,9\}$. Then $\mu_1(B)=\{0,2,5\}\subset\Z_7$ has consecutive intervals $2$, $3$, $2$, and therefore determines the interval composition $\mathbf d=(2,3,2)$. Accordingly, each rooted chord with root note $x\in X$ is given by
\[
\hat{\mathbf d}(x)
=
\{x,\T_2(x),\T_5(x)\}.
\]

Thus, $X^{(A)}=X^{(B)}$ as scalar orbit covers, since they consist of the same collection of translated chords. However, the modal orbit covers $(X^{(A)},\mu_1)$ and $(X^{(B)},\mu_1)$ are distinct, because they induce different interval compositions, $\mathbf c=(2,2,3)$ and $\mathbf d=(2,3,2)$, respectively. Thus, although the underlying collection of chords is identical, the distinguished root of each chord—and hence its associated degree and inversion data—differs.

Figure~\ref{fig:rooted-orbit-covers} compares the two modal orbit covers in musical notation. Although they determine the same scalar orbit cover, the different choices of rooted seed chord induce distinct interval compositions and therefore different root, degree, and inversion data.

\begin{figure}[ht]
\centering
\begin{subfigure}[t]{0.6\textwidth}
    \centering
    \includegraphics[width=\linewidth]{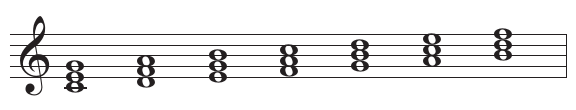}
    \caption{The rooted orbit cover $X^{(A)}$ determined by $\mathbf c=(2,2,3)$.}
\end{subfigure}
\hfill
\begin{subfigure}[t]{0.6\textwidth}
    \centering
    \includegraphics[width=\linewidth]{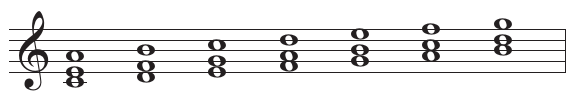}
    \caption{The rooted orbit cover $X^{(B)}$ determined by $\mathbf d=(2,3,2)$.}
\end{subfigure}
\caption{Two modal orbit covers having the same underlying scalar orbit cover but different rootings.}
\label{fig:rooted-orbit-covers}
\end{figure}
\end{example}

\subsection{Nerves of Orbit Covers}
\label{sub:nerves-of-orbit-covers}

While an orbit cover records a collection of chords, it does not record how those chords intersect. This common-tone structure is expressed by the nerve of the orbit cover.

\begin{definition}[Nerve]
Let $X^{(A)}$ be an orbit cover. Its \emph{nerve} is the simplicial complex whose vertices are the chords of $X^{(A)}$, and whose simplices are the finite collections of chords having nonempty intersection.
\end{definition}

Thus an edge of the nerve records a pair of chords sharing at least one pitch class, a triangle records three mutually intersecting chords, and higher-dimensional simplices record larger collections of chords with a common tone.

\begin{example}
\label{ex:diatonic-nerve}
Figure~\ref{fig:diatonic-nerve} shows the nerve of the triadic orbit cover of the C-major scale. The seven vertices correspond to the seven diatonic triads, while an edge joins two triads whenever they share a pitch class. For example, the triads $\{\mathrm{C},\mathrm{E},\mathrm{G}\}$ and $\{\mathrm{E},\mathrm{G},\mathrm{B}\}$ intersect in the pitches $\mathrm{E}$ and $\mathrm{G}$, and therefore are joined by an edge.

More generally, each pitch class determines a triangle. For instance, the pitch
class $\mathrm{C}$ belongs to the triads $\{\mathrm{C},\mathrm{E},\mathrm{G}\}$,  $\{\mathrm{A},\mathrm{C},\mathrm{E}\}$, and $\{\mathrm{F},\mathrm{A},\mathrm{C}\}$, so these three vertices span one of the shaded triangles in the figure. Similarly, every pitch class of the scale determines a unique maximal simplex, giving a total of seven triangles.
\end{example}

\begin{figure}[ht]
\centering
\includegraphics[width=0.65\textwidth]{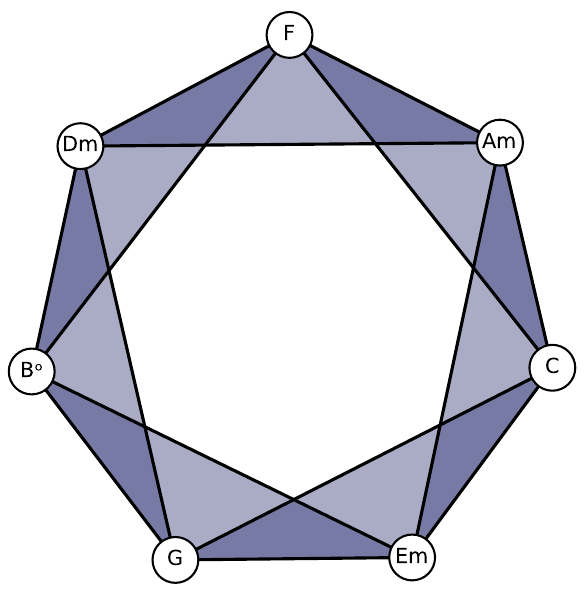}
\caption{The nerve of the C-major triadic orbit cover. Vertices correspond to
the seven triads of the cover, while each shaded triangle represents the three
triads containing a common pitch class.}
\label{fig:diatonic-nerve}
\end{figure}

Example~\ref{ex:diatonic-nerve} suggests the following general phenomenon for rooted orbit covers.

\begin{proposition}
Let $(X^{(A)},\mu_i)$ be a rooted $k$-chordal orbit cover. Then each pitch class $x\in X$ determines a unique $(k-1)$-simplex of $N(X^{(A)})$, namely the simplex whose vertices are the chords containing $x$.
\end{proposition}

\begin{proof}
Since every chord is a translate of the generating $k$-chord $A$, each pitch class belongs to exactly $k$ chords of the orbit cover. These chords have nonempty intersection, namely $\{x\}$, and therefore determine a $(k-1)$-simplex of the nerve.
\end{proof}

Thus the top-dimensional simplices of the nerve are canonically indexed by the pitch classes of the underlying scale.

Every morphism of scalar orbit covers induces a simplicial map between the corresponding nerves. Thus the nerve construction defines a functor
\[
N:\ScaleOrbCov\longrightarrow\mathbf{SimpComp}.
\]

\begin{theorem}[{\cite[\S~4.2]{flieder2026scales}}]
\label{th:nerve-isomorphism}
Two scalar orbit covers are isomorphic in $\ScaleOrbCov$ if and only if their nerves are isomorphic as simplicial complexes.
\end{theorem}

\begin{proof}
By Definition~\ref{def:morphisms-Scaorb}, isomorphisms in $\ScaleOrbCov$ are precisely the group automorphisms considered in \cite{flieder2026scales}. The theorem is therefore an immediate reformulation of the results established in \cite[\S~4.2]{flieder2026scales}.
\end{proof}

Thus, the common-tone topology of an orbit cover is completely determined by its isomorphism class in $\ScaleOrbCov$.

\section{A Hierarchy of Fibrations over Heptatonic Scales}
\label{sec:a-hierarchy-of-fibrations-over-heptatonic-scales}

The preceding sections introduced the principal musical structures studied in this paper. We first developed the theory of modes and scales, then restricted attention to seven-note scales and their three classes of isomorphisms. We next introduced orbit covers, which consist of harmonic subsets of a scale, and showed that they naturally form categories. Finally, we observed that the nerve construction captures the common-tone topology of an orbit cover.

We now assemble these constructions, along with some additional ones, into the foundational framework of the paper. Our objective is to build a hierarchy of fibrations over the category of heptatonic scale types and their deformations, with each successive layer adding further musical structure. Beginning with translation classes of scales, we equip them successively with action groupoids, orbit covers, deformation families, and harmonic functions. The resulting hierarchy culminates in a category that simultaneously encodes scale types, modes, orbit covers, chromatic alterations, and harmonic functions (Figure~\ref{fig:hierarchy-of-fibrations}).

Specifically, over each scale type we form an action-groupoid category, generalizing the common-practice operations of transposing between diatonic scales and shifting among their modes---for instance, transposing to the dominant and then moving to the relative minor. Over each mode in this fiber, we then assign orbit-cover data, as in common-practice harmony, where tertian triads are built over each diatonic mode. Next, we equip deformation families, which formalize the alteration of scale tones: raising the seventh degree in minor to obtain a leading tone, lowering the second and sixth degrees in major to obtain a Neapolitan chord, raising the fourth degree in major to obtain a secondary dominant, and so on. Finally, we assign harmonic functions, generalizing how common-practice tonality labels the chords of a key as tonic, subdominant, or dominant. This hierarchy of fibrations thus encodes all of these levels of structure coherently within a single total category, forming the basis for the ultimate object of the paper: the formalization of \emph{harmonic practices} in \S~\ref{sec:formalizing-harmonic-practices}---a general theory from which common-practice harmony arises as a special case.

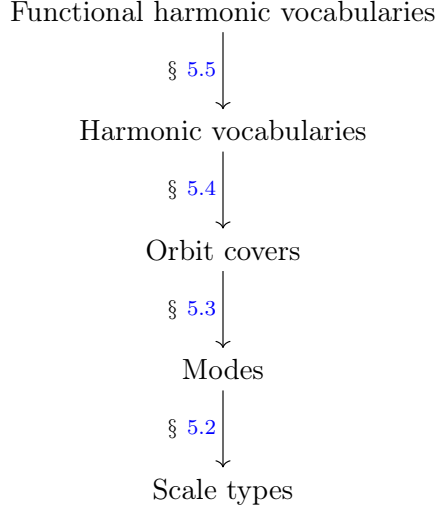
\begin{figure}[h]
\centering
\[\begin{tikzcd}[row sep=large]
	{\text{Functional harmonic vocabularies}} \\
	{\text{Harmonic vocabularies}} \\
	{\text{Orbit covers}} \\
	{\text{Modes}} \\
	{\text{Scale types}}
	\arrow[from=1-1, to=2-1, "\S~\ref{sub:harmonic-function-fibrations}"']
	\arrow[from=2-1, to=3-1, "\S~\ref{sub:the-deformation-family-fibrations}"']
	\arrow[from=3-1, to=4-1, "\S~\ref{sub:orbit-cover-fibrations}"']
	\arrow[from=4-1, to=5-1, "\S~\ref{sub:action-groupoid-fibrations}"']
\end{tikzcd}\]
\caption{The hierarchy of fibrations built in this section, read from base to summit. A choice of scale instance and tonic turns a scale type into a mode (\S~\ref{sub:action-groupoid-fibrations}, the action-groupoid fibrations). A choice of generating chord turns a mode into an orbit cover (\S~\ref{sub:orbit-cover-fibrations}, the orbit-cover fibrations). A choice of admissible chromatic alterations---a \emph{deformation family}---turns an orbit cover into a harmonic vocabulary (\S~\ref{sub:the-deformation-family-fibrations}, the deformation-family fibration), and a choice of functional assignment labels the chords of a harmonic vocabulary with harmonic functions, yielding a functional harmonic vocabulary (\S~\ref{sub:harmonic-function-fibrations}, the harmonic-function fibrations). Each downward arrow is the forgetful projection of the corresponding fibration, recovering the layer beneath.}
\label{fig:hierarchy-of-fibrations}
\end{figure}

For the remainder of this section, we restrict attention to seven-note scales and modes in $\Z_{12}$, since these are of particular musical interest. The constructions below, however, generalize in the evident way to $n$-note scales and modes.

Let $\Hept$ denote the set of heptatonic pitch-class subsets of $\Z_{12}$, and let $\Hept/\mathbf T$ denote the set of their translation orbits. For each $[X]\in\Hept/\mathbf T$, recall (Definition~\ref{def:deformation-spaces}) the type $\Def(X)$ of deformations of $[X]$. By Proposition~\ref{prop:equivariance}, every deformation $\Delta\in\Def(X)$ extends uniquely to a map $\Delta:[X] \to [Y]$ between translation orbits. Closing these maps under composition yields the category $\HDef$, whose objects are the translation orbits in $\Hept/\mathbf T$ and whose morphisms are the resulting deformation maps.

The remainder of this section constructs the hierarchy of fibrations over $\HDef$.

\subsection{Grothendieck Fibrations}
\label{sub:grothendieck-fibrations}

Let us briefly recall the notion of a Grothendieck fibration.\footnote{For fuller accounts, see \cite{borceux1994handbook,jacobs1999categorical}.} Suppose $\mathcal C$ is a category, and let $F:\mathcal C\to\Cat$ be a functor assigning to each $C\in\operatorname{Ob}(\mathcal C)$ a category $F(C)$, and to each morphism $f:C\to C'$ a functor $F(f):F(C)\to F(C')$. One may think of $F$ as assigning an additional layer of structure to each object of $\mathcal C$. An object $X\in\operatorname{Ob}(F(C))$ represents a choice of additional structure attached to the base object $C$, while a morphism $f:C\to C'$ transports this structure via the functor $F(f)$. Thus every pair $(C,X)$, where $X\in\operatorname{Ob}(F(C))$, is carried by $f$ to the pair $(C',F(f)(X))$. Morphisms within the category $F(C)$ alter the additional structure while leaving the underlying object $C$ fixed.

The Grothendieck construction assembles all such pairs into a single category, whose objects consist of base objects together with their chosen additional structure, and whose morphisms simultaneously transport both the underlying objects and their attached structure. The resulting projection back to the base category is called a \emph{Grothendieck fibration}.

Given a functor $F:\mathcal C\to\Cat$, we denote its Grothendieck construction by $\int F$, yielding a fibration $\int F\to\mathcal C$. For each object $C\in\operatorname{Ob}(\mathcal C)$, the fiber over $C$ is canonically equivalent to the category $F(C)$.

\begin{example}
The first level of the hierarchy constructed below arises this way. Let $\Gpd$ denote the category of groupoids, and define $\Act_\bfT:\HDef\to\Gpd$ by sending each translation orbit $[X]$ to its translation action groupoid $[X]\sslash\bfT$ (Definition~\ref{def:action-groupoid}), and each deformation $\Delta:[X]\to[Y]$ to the induced functor between action groupoids. The Grothendieck construction then yields a fibration
\[
P_\bfT:\int\ActT\longrightarrow\HDef,
\]
whose objects are pairs $([X],X')$ with $X'\in[X]$, and whose fiber over $[X]$ is the action groupoid $[X]\sslash\bfT$. A base morphism $\Delta:[X]\to[Y]$ lifts to a functor $\Delta$ between the corresponding fibers (see Figure~\ref{fig:act-groupoid-fibration}). This fibration, together with three further fibrations built the same way, is developed in detail below (\S~\ref{sub:action-groupoid-fibrations}--\ref{sub:harmonic-function-fibrations}).
\end{example}

\begin{sidewaysfigure}
\centering
\includegraphics[width=1\textheight]{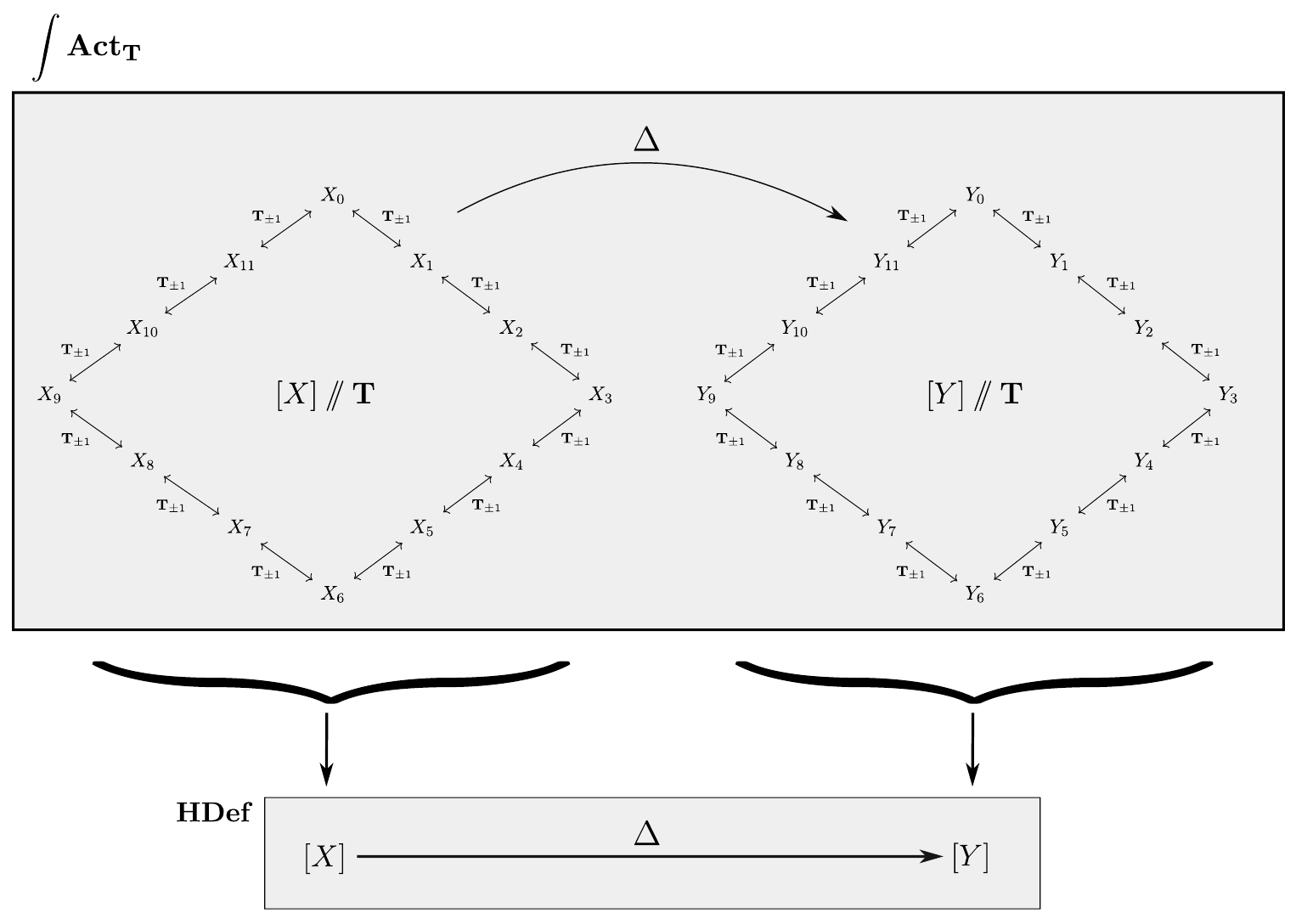}
\caption{The fibration $P_\bfT:\int\Act_\bfT\to\HDef$. The fibers over $[X]$ and $[Y]$ are the translation action groupoids $[X]\sslash\bfT$ and $[Y]\sslash\bfT$, each a 12-cycle under $\bfT_{\pm1}$. A morphism $\Delta:[X]\to[Y]$ in the base $\HDef$ lifts to the functor $\Delta$ between fibers, which the Grothendieck construction assembles into the total category $\int\Act_\bfT$.}
\label{fig:act-groupoid-fibration}
\end{sidewaysfigure}

Applying this construction repeatedly to $\HDef$ yields the hierarchy of fibrations shown in Figure~\ref{fig:fibration-hierarchy}, in which each level enriches the preceding one with additional musical structure. Reading from bottom to top, the hierarchy successively equips each scale class with action groupoids, orbit covers, deformation families, and harmonic functions.

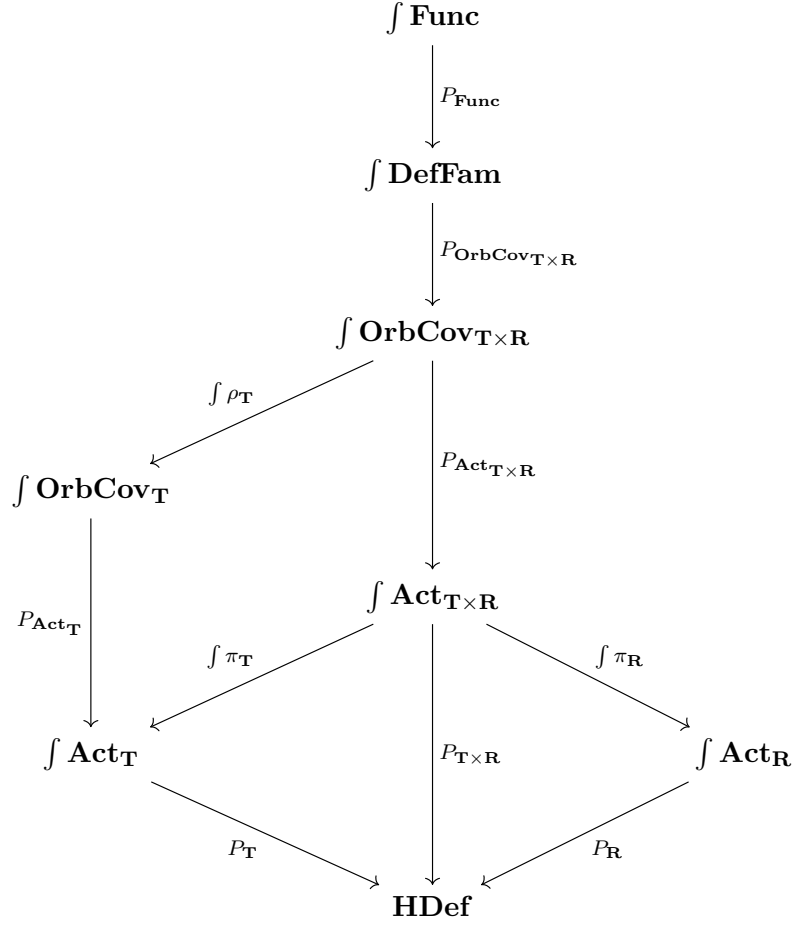
\begin{figure}[!h]
\[\begin{tikzcd}
	&& {\int \Func} && \\
	\\
	&& {\int \DefFam} \\
	\\
	&& {\int \OrbCovTR} \\
	\\
	{\int \OrbCovT} \\
	&& {\int \ActTR} \\
	\\
	{\int \ActT} &&&& {\int \ActR} \\
	\\
	&& \HDef
	\arrow["{P_\Func}", from=1-3, to=3-3]
	\arrow["{P_{\OrbCovTR}}", from=3-3, to=5-3]
	\arrow["{\int\rho_\bfT}"', from=5-3, to=7-1]
	\arrow["{P_{\ActTR}}", from=5-3, to=8-3]
	\arrow["{P_{\ActT}}"', from=7-1, to=10-1]
	\arrow["{\int\pi_\bfT}"', from=8-3, to=10-1]
	\arrow["{\int\pi_\bfR}", from=8-3, to=10-5]
	\arrow["{P_{\bfT \times \bfR}}", from=8-3, to=12-3]
	\arrow["{P_\bfT}"', from=10-1, to=12-3]
	\arrow["{P_\bfR}", from=10-5, to=12-3]
\end{tikzcd}\]
\caption{The hierarchy of fibrations constructed in this paper. Each level enriches the preceding one with additional musical structure.}
\label{fig:fibration-hierarchy}
\end{figure}

\subsection{Action Groupoid Fibrations}
\label{sub:action-groupoid-fibrations}

We begin by recalling the notion of an action groupoid.

\begin{definition}[Action Groupoid]
\label{def:action-groupoid}
Given a set $S$ and a group action $\alpha:G\times S\to S$, the \emph{action groupoid} $S\sslash G$ is the category whose objects are the elements $s\in S$ and whose morphisms $g:s\to t$ are the group elements satisfying $\alpha(g,s)=t$.
\end{definition}

We now define the action-groupoid fibrations over $\HDef$.

\subsubsection{The Translation-Action-Groupoid Fibration}
\label{subsub:the-translation-action-groupoid-fibration}

We define the functor
\[
\Act_{\mathbf T}:\HDef\longrightarrow\Gpd
\]
as follows.

\begin{enumerate}
\item On objects, it sends each translation orbit $[X]\in\Hept/\mathbf T$ to the translation action groupoid  $[X]\sslash\mathbf T$.

\item On morphisms, it sends each deformation map $\Delta:[X]\to[Y]$ to the induced deformation functor $\Delta:[X]\sslash\mathbf T \to [Y]\sslash\mathbf T$,
defined by:
\begin{enumerate}
\item on objects, $X' \to \Delta(X')$;
\item on morphisms, it acts as the identity.
\end{enumerate}
\end{enumerate}

The functor $\Act_\bfT$ gives rise, via the Grothendieck construction, to the fibration
\[
P_\bfT : \int \ActT \longrightarrow \HDef. 
\]
An object of $\int\Act_\bfT$ is a pair $([X],Y)$, where $Y\in[X]$. A morphism
\[
(\Delta,\bfT_p):([X],Y)\longrightarrow([W],Z)
\]
consists of a deformation $\Delta:[X]\to[W]$ together with a translation satisfying 
\[
\Delta(\bfT_p(Y))=Z.
\]
The fiber over $[X]$ is therefore the translation action groupoid $[X]\sslash\bfT$. The deformation component $\Delta$ of a morphism induces the corresponding map between translation action groupoids.

\subsubsection{The Rotation-Action-Groupoid Fibration}
\label{subsub:the-rotation-action-groupoid-fibration}

We define the functor
\[
\Act_{\bfR}:\HDef\longrightarrow\Gpd
\]
as follows.

\begin{enumerate}
\item On objects, it sends each translation orbit $[X]\in\Hept/\bfT$ to the rotation action groupoid $\Z_7\sslash\bfR$, whose objects are the modal degrees $i\in\Z_7$, and whose morphisms are the rotations $\bfR_q:i\to i+q$.

\item On morphisms, it sends each deformation map $\Delta:[X]\to[Y]$ to the induced deformation functor $\Delta:\Z_7\sslash\bfR \to \Z_7\sslash\bfR$, defined by:
\begin{enumerate}
\item on objects, $i \to \delta(i)$, where $\delta:\Z_7\to\Z_7$ is the translation of modal degrees induced by $\Delta$ (see \eqref{eq:delta});

\item on morphisms, it acts as the identity. 
\end{enumerate}
\end{enumerate}

The functor $\Act_{\bfR}$ gives rise, via the Grothendieck construction, to the fibration
\[
P_\bfR:\int\Act_{\bfR}\longrightarrow\HDef.
\]
An object of $\int\Act_{\bfR}$ is a pair $([X],i)$, which specifies the $i$th mode of every scale in the translation orbit $[X]$. A morphism
\[
(\Delta,\bfR_q):([X],\mu_i)\longrightarrow([Y],\mu_j)
\]
consists of a deformation $\Delta:[X]\to[Y]$ together with a rotation satisfying 
\[
\delta(\bfR_q(i)) = \delta(i + q) = j.
\]
The fiber over $[X]$ is therefore the rotation action groupoid $\Z_7\sslash\bfR$. The deformation component $\Delta$ of a morphism induces the corresponding map between rotation action groupoids.

\subsubsection{The Translation-Rotation-Action-Groupoid Fibration}
\label{subsub:the-translation-rotation-action-groupoid-fibration}

This fibration combines the translation- and rotation-action-groupoid fibrations. We define the functor
\[
\Act_{\bfT\times\bfR}:\HDef\longrightarrow\Gpd
\]
as follows.

\begin{enumerate}
\item On objects, it sends each translation orbit $[X]\in\Hept/\bfT$ to the action groupoid 
\[
([X]\times\Z_7)\sslash(\bfT\times\bfR),
\] 
which is canonically isomorphic to the product $\Act_{\bfT}([X])\times\Act_{\bfR}([X])$. Its objects may therefore be identified with the modes $(X',\mu_i)$, where $X'\in[X]$ and $i\in\Z_7$. A morphism is a pair
\[
(\bfT_p,\bfR_q) : (X',\mu_i) \longrightarrow (\bfT_p(X'),\mu_{i+q}).
\]

\item On morphisms, it sends each deformation map $\Delta:[X] \to [Y]$ to the induced deformation functor
\[
\Delta:
([X]\times\Z_7)\sslash(\bfT\times\bfR)
\longrightarrow
([Y]\times\Z_7)\sslash(\bfT\times\bfR),
\]
defined by:
\begin{enumerate}
\item on objects, $(X',\mu_i) \to (\Delta(X'),\mu_{\delta(i)})$, where $\delta:\Z_7\to\Z_7$ is the translation of scale degrees induced by~$\Delta$;

\item on morphisms, it acts as the identity. Indeed, since $\delta(i+q)=\delta(i)+q$, we have
\[
\Delta(\bfT_p,\bfR_q)
=
(\bfT_p,\bfR_q):
(\Delta(X'),\mu_{\delta(i)})
\longrightarrow
(\Delta(\bfT_p(X')),\mu_{\delta(i+q)}).
\]
\end{enumerate}
\end{enumerate}

The functor $\Act_{\bfT\times\bfR}$ gives rise, via the Grothendieck construction, to the fibration
\[
P_{\bfT \times \bfR} : \int\Act_{\bfT\times\bfR} \longrightarrow \HDef.
\]
Moreover, there are natural projection functors
\[
\int \pi_{\bfT}: \int\Act_{\bfT\times\bfR} \longrightarrow \int\Act_{\bfT}
\]
and
\[
\int \pi_{\bfR}:\int\Act_{\bfT\times\bfR}\longrightarrow\int\Act_{\bfR},
\]
defined on fibers by 
\[ 
([X]\times\Z_7)\sslash(\bfT\times\bfR)
\longmapsto
[X]\sslash\bfT
\] 
and 
\[
([X]\times\Z_7)\sslash(\bfT\times\bfR)
\longmapsto
\Z_7\sslash\bfR,
\]
respectively.

Since all three fibrations are indexed by the common base category $\HDef$, the projection functors send each lifted deformation morphism to the corresponding lifted deformation morphism. The projections are morphisms of fibrations over $\HDef$; equivalently, the following diagram commutes:
\[\begin{tikzcd}
	& {\int \Act_{\bfT \times \bfR}} & \\
	\\
	{\int \Act_{\bfT}} && {\int \Act_{\bfR}} \\
	\\
	& \HDef
	\arrow["{\int\pi_\bfT}"', from=1-2, to=3-1]
	\arrow["{\int\pi_\bfR}", from=1-2, to=3-3]
	\arrow["{P_{\bfT \times \bfR}}", from=1-2, to=5-2]
	\arrow["{P_\bfT}"', from=3-1, to=5-2]
	\arrow["{P_\bfR}", from=3-3, to=5-2]
\end{tikzcd}\]

\subsection{Orbit Cover Fibrations}
\label{sub:orbit-cover-fibrations}

The next layer of our hierarchy of fibrations enriches the action groupoids by equipping their objects with orbit cover data.

The following equivariance result shows that group automorphisms commute with deformations, and is an essential component in constructing the orbit-cover fibrations.

\begin{proposition}
\label{prop:equivariance-of-deformations}
Let $h:\Z_7 \to \Z_7$ be a group automorphism, and let $h_{i,j}:X\to X$ be the induced map satisfying $\mu_j\circ h_{i,j}= h\circ \mu_i$ (see \eqref{eq:h_ij}). Let $\Delta:X\to \Delta(X)$ be a deformation, and let $\delta:\Z_7\to \Z_7$ be its induced translation on mode indices, as defined in \eqref{eq:delta}. Then the induced map $h_{\delta(i),\delta(j)}:\Delta(X)\to \Delta(X)$ satisfies
\[
h_{\delta(i),\delta(j)}\circ \Delta
=
\Delta\circ h_{i,j},
\]
i.e., the diagram
\[\begin{tikzcd}
	X && X \\
	\\
	{\Delta(X)} && {\Delta(X)}
	\arrow["{h_{i,j}}", from=1-1, to=1-3]
	\arrow["\Delta"', from=1-1, to=3-1]
	\arrow["\Delta", from=1-3, to=3-3]
	\arrow["{h_{\delta(i),\delta(j)}}"', from=3-1, to=3-3]
\end{tikzcd}\]
commutes.
\end{proposition}

\begin{proof}
The claim is precisely the commutativity of the following diagram. Each of its constituent regions commutes by the defining properties of $\Delta$, $h_{i,j}$, and $\delta$. Hence the outer rectangle commutes.
\[\begin{tikzcd}
	{\Delta(X)} &&& {\Delta(X)} \\
	& X & X \\
	\\
	{\Z_7} &&& {\Z_7}
	\arrow["{h_{\delta(i), \delta(j)}}", from=1-1, to=1-4]
	\arrow["{\mu_{\delta(i)}}"', from=1-1, to=4-1]
	\arrow["{\mu_{\delta(j)}}", from=1-4, to=4-4]
	\arrow["\Delta"', from=2-2, to=1-1]
	\arrow["{h_{i,j}}", from=2-2, to=2-3]
	\arrow["{\mu_i}", from=2-2, to=4-1]
	\arrow["\Delta", from=2-3, to=1-4]
	\arrow["{\mu_j}"', from=2-3, to=4-4]
	\arrow["h"', from=4-1, to=4-4]
\end{tikzcd}\]
\end{proof}

\subsubsection{The Scalar Orbit-Cover Fibration}
\label{subsub:the-scalar-orbit-cover-fibration}

We first enrich the translation action groupoids by assigning a scalar orbit cover to each translate. Accordingly, we define a functor
\[
\OrbCovT:
\int\ActT
\longrightarrow
\Cat.
\]
An object of $\int\ActT$ is a pair $([X],Y)$, where $Y\in[X]$. To such an object we assign the category $\OrbCovT([X],Y)$, whose
\begin{enumerate}
\item objects are scalar orbit covers $Y^{(A)}$;
\item morphisms are scalar orbit-cover morphisms $h_{i,j}: Y^{(A)} \to Y^{(B)}$,
as defined in Definition~\ref{def:morphisms-Scaorb}.
\end{enumerate}

Now let $(\Delta,\bfT_p):
([X],Y)
\to
([W],Z)$ be a morphism of $\int\ActT$, so that $\Delta(\bfT_p(Y))=Z$. We define
\[
\OrbCovT(\Delta,\bfT_p):
\OrbCovT([X],Y)
\longrightarrow
\OrbCovT([W],Z)
\]
by
\begin{enumerate}
\item on objects, $Y^{(A)} \mapsto (\Delta(\bfT_p(Y^{(A)}))$;
\item on morphisms, $h_{i,j}
\mapsto
h_{\delta(i),\delta(j)}$,
by Proposition \ref{prop:equivariance-of-deformations}.
\end{enumerate}
Functoriality follows from the fact that  $\Delta \circ h_{i,j} = h_{\delta(i),\delta(j)} \circ \Delta$, from Proposition \ref{prop:equivariance-of-deformations}. Applying the Grothendieck construction therefore yields a fibration
\[
P_{\ActT}:
\int\OrbCovT
\longrightarrow
\int\ActT.
\]
An object of $\int\OrbCovT$ is a triple $([X],Y,Y^{(A)})$, which we abbreviate as $([X],Y^{(A)})$. The fiber over $([X],Y)$ is precisely the category $\OrbCovT([X],Y)$ of scalar orbit covers on the translate $Y$.

\subsubsection{Classifying and Characterizing Orbit-Cover Morphisms}
\label{subsub:classifying-characterizing-orbit-cover-morphisms}

Fix orbit covers $X^{(A)}$ and $Y^{(B)}$, generated respectively by interval compositions $\mathbf{c}$ and $\mathbf{d}$ (recalling the discussion of interval compositions from \S~\ref{sub:rooted-orbit-covers}). The condition of Definition~\ref{def:morphisms-Scaorb} states that a group automorphism $h : \Z_n \to \Z_n$ extends to an orbit cover morphism $X^{(A)} \to Y^{(B)}$ whenever, for every chord $A' \in X^{(A)}$, there exists a chord $B' \in Y^{(B)}$ such that $h_{i,j}(A') \subseteq B'$. This leaves open the question of \emph{which} such $B'$ is chosen as the witness for the subset condition. Choosing, for each $A'$, a specific witness $B'$ yields not merely a verification that the condition holds, but an actual bijection between the chords of $X^{(A)}$ and $Y^{(B)}$. This is the data we will need in \S~\ref{sub:harmonic-function-fibrations} to transport harmonic-function assignments from chords in one orbit cover to chords in another. We now show that such chord maps are classified precisely by the morphisms in a category $\mathbf C(n)$ of interval compositions of an $n$-note scale, which we introduce next.

\paragraph{The Category of Interval Compositions}

We construct a category $\mathbf{C}(n)$ whose objects are interval compositions of an $n$-note scale. Its morphisms $\mathbf c \to \mathbf d$ are exactly the data needed to induce a bijection between the chords of $X^{(A)}$ and $Y^{(B)}$, realizing the morphism condition of Definition~\ref{def:morphisms-Scaorb}.

We start with some preliminary work. For an interval composition $\mathbf c \in \mathbf{C}(n,k)$, recall its partial sums $S_b(\mathbf c)$ from \eqref{eq:part-sums}, and the associated set $\hat{\mathbf c}(0) = \{S_0(\mathbf c),\ldots,S_{k-1}(\mathbf c)\}$ from \eqref{eq:int-comp}. Since the parts of $\mathbf c$ are positive, $S_0=0<S_1<\cdots<S_{k-1}<n$ are strictly increasing, so $\hat{\mathbf c}(0)$ is a $k$-element subset of $\Z_n$ containing $0$; write $\binom{\Z_n}{k}_0$ for the set of all such subsets.

\begin{proposition}
\label{prop:intcomp-subset-bijection}
The correspondence $\mathbf c \mapsto \hat{\mathbf c}(0)$ defines a bijection
\[
\mathbf C(n,k) \; \xrightarrow{\ \sim\ } \; \binom{\Z_n}{k}_0,
\]
with inverse sending $A = \{0=A_0<A_1<\cdots<A_{k-1}\} \in \binom{\Z_n}{k}_0$ to
\[
\mathbf a \coloneqq (A_1-A_0,\ldots, n-A_{k-1}) \in \mathbf C(n,k).
\]
\end{proposition}

\begin{proof}
Injectivity and well-definedness of $\mathbf c \mapsto \hat{\mathbf c}(0)$ were established above. Conversely, given $A=\{0=A_0<\cdots<A_{k-1}\}\in\binom{\Z_n}{k}_0$, the entries of $\mathbf a=(A_1-A_0,\ldots,n-A_{k-1})$ are positive integers summing to $n$, with partial sums $S_b(\mathbf a)=A_b$ for $0\le b<k$, so $\hat{\mathbf a}(0)=A$. These two constructions are mutually inverse.
\end{proof}

Via this identification, we may treat interval compositions $\mathbf c\in\mathbf C(n,k)$ interchangeably with objects of $\binom{\Z_n}{k}_0$. We now turn to morphisms, building toward Definition~\ref{def:intcomp-category} below, which assembles $\mathbf C(n)$ as a category once the right notion of morphism is in hand.

\begin{convention}
\label{conv:scale-degree-identification}
Equipping $X$ and $Y$ with modes $\mu_i : X \to \Z_n$ and $\mu_j : Y \to \Z_n$, we  identify a chord $A \subseteq X$ (resp.\ $B \subseteq Y$) with its image $\mu_i(A) \subseteq \Z_n$ (resp.\ $\mu_j(B)$), writing both simply as $A$ (resp.\ $B$), and using $X^{(A)}$ and $\Z_n^{(A)}$ interchangeably for the corresponding orbit cover.
\end{convention}

Fix orbit covers $X^{(A)}$ and $Y^{(B)}$, and suppose $h : \Z_n \to \Z_n$ is a group automorphism inducing an orbit cover morphism $h_{i,j} : X^{(A)} \to Y^{(B)}$, in the sense of Definition~\ref{def:morphisms-Scaorb}. The modal identifications $\mu_i, \mu_j$ transport $h_{i,j}$ to $h$ itself (Convention~\ref{conv:scale-degree-identification}), so the covering condition of Definition~\ref{def:morphisms-Scaorb} becomes a condition on chords of $\Z_n^{(A)}$ and $\Z_n^{(B)}$.

Since $0 \in A$ and $h$ is a group automorphism, $0 = h(0) \in h(A)$. Consequently, any chord $B'\in \Z_n^{(B)}$ witnessing $h(A) \subseteq B'$ must itself contain $0$. It therefore suffices to search among translates of $B$ containing $0$; these are precisely the translates $\T_{-b}(B)$ for $b \in B$, each itself an object of $\mathbf C(n)$. We will show below that the existence of such a translate $B'$ with $h(A) \subseteq B'$ is not only sufficient but \emph{necessary} for the existence of an orbit cover morphism $X^{(A)} \to Y^{(B)}$, and moreover that a choice of such $B'$ determines an induced bijection on chords.

We record what the aforementioned data amounts to. Seeking a translate $\T_{-b}(B)$ with $h(A) \subseteq \T_{-b}(B)$ is the same as seeking $b \in B$ such that the affine automorphism
\[
f(x) \coloneqq h(x) + b
\]
satisfies $f(A) \subseteq B$. Thus the data of a chord-level witness is exactly the data of an affine automorphism $f : \Z_n \to \Z_n$ with linear part $h$, satisfying $f(A) \subseteq B$. This is more than $h$ alone provides: $h$ only tells us that \emph{some} chord $B' \in \Z_n^{(B)}$ exists for each $A' \in \Z_n^{(A)}$ such that $h(A') \subseteq B'$, whereas $f$ additionally selects the witnessing translate, and so, as we shall soon see, determines a specific bijection of chords $\Z_n^{(A)} \to \Z_n^{(B)}$. Since $h$ is a group automorphism of $\Z_n$, it is of the form $h(x) = ux$ for a unique unit $u \in \Z_n^\times$; we write $u$ for this unit going forward, so that $f(x) = ux+b$.

\begin{lemma}
\label{lem:affine-translate-commute}
Let $u \in \Z_n^\times$ and $j \in \Z_n$. Then
\[
u \circ \T_j = \T_{uj} \circ u.
\]
\end{lemma}

\begin{proof}
For $x \in \Z_n$, we have
\begin{align*}
u \circ \T_j(x) & = u(x + j) \\
& = ux + uj \\
& = \T_{uj} \circ u(x). 
\end{align*}
Hence $u \circ \T_j = \T_{uj} \circ u$.
\end{proof}

\begin{proposition}
\label{prop:induced-chord-bijection}
Let $f \coloneqq \T_b \circ u$, i.e.\ $f(x) = ux+b$, and suppose $f(A) \subseteq B$, equivalently $u(A) \subseteq \T_{-b}(B)$. Then
\[
\begin{matrix}
\hat f : & \Z_n^{(A)} & \longrightarrow & \Z_n^{(B)} \\
&  \T_m(A)& \longmapsto &\T_{um-b}(B)
\end{matrix}
\]
is a well-defined bijection.
\end{proposition}

\begin{proof}
For each $m \in \Z_n$, Lemma~\ref{lem:affine-translate-commute} gives $u(\T_m(A)) = \T_{um}(u(A))$, and applying $\T_{um}$ to both sides of $u(A) \subseteq \T_{-b}(B)$ gives
\[
u(\T_m(A)) = \T_{um}(u(A)) \subseteq \T_{um}(\T_{-b}(B)) = \T_{um-b}(B),
\]
so $\hat f$ is well-defined. Since $u \in \Z_n^\times$, the map $m \mapsto um - b$ is a bijection of $\Z_n$; composing with the parametrizations $m \mapsto \T_m(A)$ and $m \mapsto \T_m(B)$ of $\Z_n^{(A)}$ and $\Z_n^{(B)}$ shows $\hat f$ is a bijection.
\end{proof}

\begin{definition}[The Category $\mathbf C(n)$]
\label{def:intcomp-category}
The category $\mathbf C(n)$ of interval compositions of $n$ is defined as follows.
\begin{enumerate}
    \item The object class is $\mathrm{Ob}(\mathbf C(n)) \coloneqq \coprod_{k=1}^n \binom{\Z_n}{k}_0$, identified via Proposition~\ref{prop:intcomp-subset-bijection} with interval compositions $\mathbf c \in \mathbf C(n,k)$, $1\le k\le n$; accordingly, we  write $\mathbf c$ for the object $\hat{\mathbf c}(0)$, and $\mathrm{Hom}_{\mathbf C(n)}(\mathbf c,\mathbf d)$ for $\mathrm{Hom}_{\mathbf C(n)}(\hat{\mathbf c}(0),\hat{\mathbf d}(0))$.
    \item For objects $A, B$, a \emph{morphism} $A \to B$ is an affine automorphism $f : \Z_n \to \Z_n$, $f(x) = ux+b$ for some $u\in\Z_n^\times$, $b\in\Z_n$, satisfying $f(A) \subseteq B$.
\end{enumerate}
Composition is composition of the underlying affine automorphisms, and the identity morphism on $A$ is $\mathrm{id}$ (the case $u=1$, $b=0$). Composition is well-defined, since a composite of affine automorphisms is again an affine automorphism, and $g(f(A)) \subseteq g(B) \subseteq C$ whenever $f(A)\subseteq B$ and $g(B)\subseteq C$. Hence $\mathbf C(n)$ is a category.
\end{definition}

\begin{remark}
By Proposition~\ref{prop:induced-chord-bijection}, every morphism $f\in\mathrm{Hom}_{\mathbf C(n)}(A,B)$ determines a canonical induced bijection $\hat f : \Z_n^{(A)} \to \Z_n^{(B)}$ on chords.
\end{remark}

\begin{theorem}
\label{th:main-morphism-characterization}
An orbit cover morphism $X^{(A)} \to Y^{(B)}$ exists if and only if there is an affine automorphism $f : \Z_n \to \Z_n$, $f(x) = ux+b$ for some $u\in\Z_n^\times$, $b\in\Z_n$, satisfying
\[
f(A) \subseteq B,
\]
equivalently
\[
u(A) \subseteq \T_{-b}(B).
\]
That is, $X^{(A)} \to Y^{(B)}$ admits an orbit cover morphism if and only if $\mathrm{Hom}_{\mathbf C(n)}(A,B) \neq \varnothing$.
\end{theorem}

\begin{proof}
\emph{Sufficiency.} Suppose $f(x)=ux+b$ satisfies $f(A)\subseteq B$, equivalently $u(A) \subseteq \T_{-b}(B)$. Setting $h(x) \coloneqq ux$, Proposition~\ref{prop:induced-chord-bijection} shows that $h$ satisfies the covering condition of Definition~\ref{def:morphisms-Scaorb} (with witness $\T_{-b}(B) \in Y^{(B)}$ for the chord $A$, and, by translation-equivariance, the corresponding translate of $B$ for every other chord of $X^{(A)}$) and moreover induces a bijection $\hat f : \Z_n^{(A)} \to \Z_n^{(B)}$ on chords. Hence $h$ extends to an orbit cover morphism $X^{(A)} \to Y^{(B)}$.

\emph{Necessity.} Suppose $h:\Z_n\to\Z_n$ is a group automorphism inducing an orbit cover morphism $X^{(A)} \to Y^{(B)}$. By Definition~\ref{def:morphisms-Scaorb}, there exists $B' \in \Z_n^{(B)}$ with $h(A) \subseteq B'$. Since $0 \in A$ and $h(0)=0$, we have $0 \in h(A) \subseteq B'$, so $B' = \T_{-b}(B)$ for some $b \in B$, as established above. Setting $f(x) \coloneqq ux+b$, this is exactly the condition $u(A) \subseteq \T_{-b}(B)$, i.e.\ $f(A)\subseteq B$.
\end{proof}

\begin{convention}
\label{conv:tonic-rooted-notation}
For an interval composition $\mathbf c\in\mathbf C(n,k)$ and a scale $X$ with chosen tonic $t\in X$, i.e.\ $\mu_i(t)=0$, we write $(X^{(\mathbf c)},\mu_i)$ for the modal orbit cover $X^{(\hat{\mathbf c}(t))}$, where $\hat{\mathbf c}(t)$ is as in \eqref{eq:int-comp}.
\end{convention}

\begin{remark}
\label{rem:induced-map-in-composition-notation}
Under Convention~\ref{conv:tonic-rooted-notation}, $\T_m(A) = \hat{\mathbf c}(m)$ and $\T_{um-b}(B) = \hat{\mathbf d}(um-b)$, so the induced bijection of Proposition~\ref{prop:induced-chord-bijection} may be written
\[
\hat f(\hat{\mathbf c}(z)) = \hat{\mathbf d}(uz - b), \qquad z \in \Z_n.
\]
\end{remark}

\begin{example}
\label{ex:rotation-witness-nonuniqueness}
Let $\mathbf c=(2,2,3)$, the tertian triad type, so $\hat{\mathbf c}(0)=\{0,2,4\}$, and let $\mathbf d=(2,2,1,2)$, so $\hat{\mathbf d}(0)=\{0,2,4,5\}$. Take $u=1$.

The morphism $f_0(x)=x$ (i.e.\ $u=1$, $b=0$) witnesses containment trivially, since $\hat{\mathbf c}(0)\subseteq\hat{\mathbf d}(0)$. But so does $f_1(x)=x-2$ (i.e.\ $u=1$, $b=-2$), since
\[
\hat{\mathbf c}(0) - 2 = \{-2,0,2\} \pmod 7 = \{0,2,5\} \subseteq \hat{\mathbf d}(0).
\]
Both $f_0,f_1 \in \mathrm{Hom}_{\mathbf C(7)}(\mathbf c,\mathbf d)$, and by Remark~\ref{rem:induced-map-in-composition-notation} induce distinct bijections
\[
\hat f_0(\hat{\mathbf c}(z)) = \hat{\mathbf d}(z),
\qquad
\hat f_1(\hat{\mathbf c}(z)) = \hat{\mathbf d}(z+2).
\]
Identifying $\Z_7$ with the degrees of the Ionian mode ($0,1,2,3,4,5,6\leftrightarrow \mathrm C,\mathrm D,\mathrm E,\mathrm F,\mathrm G,\mathrm A,\mathrm B$), the root triad $\hat{\mathbf c}(0) = \{0,2,4\} = \{\mathrm C,\mathrm E,\mathrm G\}$ is sent by the two witnesses to
\[
\hat f_0(\{0,2,4\}) = \{0,2,4,5\} = \{\mathrm C,\mathrm E,\mathrm G,\mathrm A\},
\qquad
\hat f_1(\{0,2,4\}) = \{0,2,4,6\} = \{\mathrm C,\mathrm E,\mathrm G,\mathrm B\}.
\]
That is, $\hat f_0$ extends the C-major triad to $\mathrm C6$, while $\hat f_1$ extends it to $\mathrm C\mathrm{maj}7$---and by the same computation applied to every $z\in\Z_7$, $\hat f_0$ uniformly adds the sixth to every tertian triad in the orbit cover while $\hat f_1$ uniformly adds the seventh. Both are legitimate instances of the canonical induced chord map for the same unit $u=1$.
\end{example}

Figure~\ref{fig:isoclasses} shows the isomorphism classes of orbit covers over a heptatonic scale, according to the interval compositions that generate them. Rows are arranged by $k$, the number of parts, from $k=7$ at the top to $k=1$ at the bottom. Each node lists an isomorphism class as a set of rotation classes $[\mathbf c]$, and an arrow between two nodes indicates the existence of an orbit cover morphism between the orbit covers generated by representatives of the respective classes.

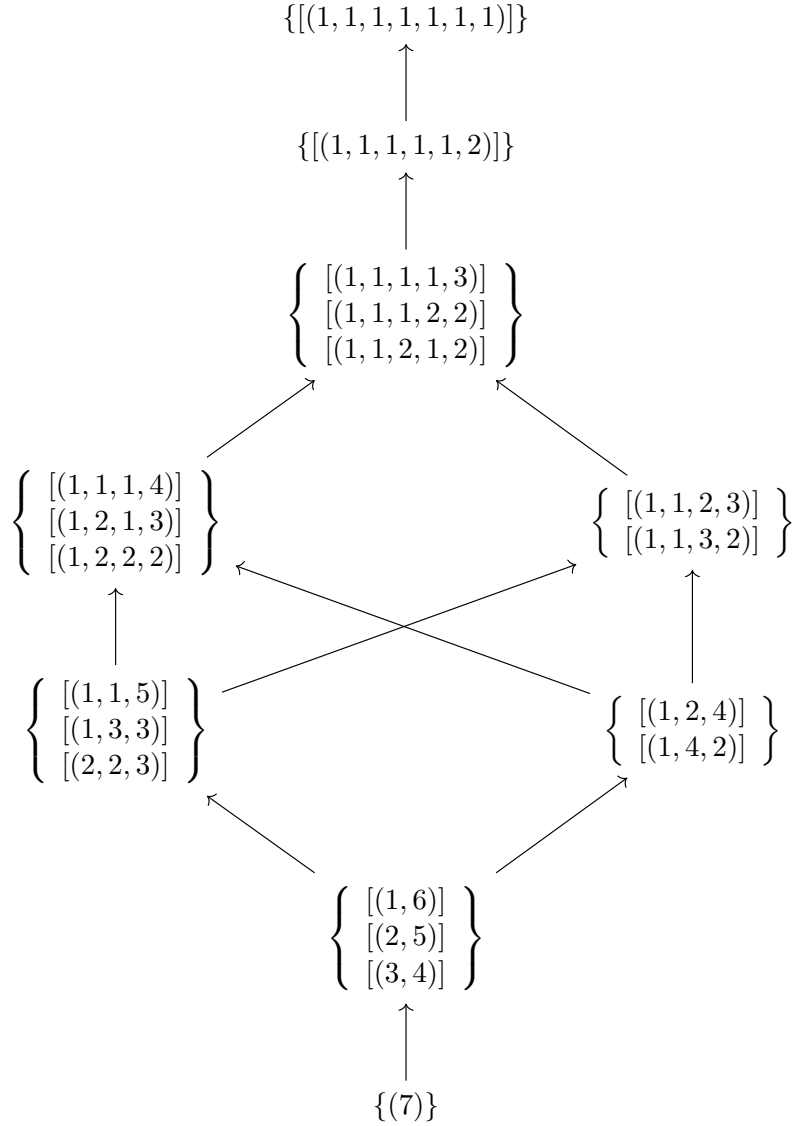
\begin{figure}[h!]
\[
\begin{tikzcd}[column sep=small, row sep=large]
& \left\{[(1,1,1,1,1,1,1)]\right\} & \\
& \left\{[(1,1,1,1,1,2)]\right\} & \\
& {\left\{\begin{array}{l} [(1,1,1,1,3)] \\{} [(1,1,1,2,2)] \\{} [(1,1,2,1,2)] \end{array}\right\}} & \\
{\left\{\begin{array}{l} [(1,1,1,4)] \\{} [(1,2,1,3)] \\{} [(1,2,2,2)] \end{array}\right\}} & & {\left\{\begin{array}{l} [(1,1,2,3)] \\{} [(1,1,3,2)] \end{array}\right\}} \\
{\left\{\begin{array}{l} [(1,1,5)] \\{} [(1,3,3)] \\{} [(2,2,3)] \end{array}\right\}} & & {\left\{\begin{array}{l} [(1,2,4)] \\{} [(1,4,2)] \end{array}\right\}} \\
& {\left\{\begin{array}{l} [(1,6)] \\{} [(2,5)] \\{} [(3,4)] \end{array}\right\}} & \\
& \left\{(7)\right\} &
	\arrow[from=2-2, to=1-2]
	\arrow[from=3-2, to=2-2]
	\arrow[from=4-1, to=3-2]
	\arrow[from=4-3, to=3-2]
	\arrow[from=5-1, to=4-1]
	\arrow[from=5-1, to=4-3]
	\arrow[from=5-3, to=4-1]
	\arrow[from=5-3, to=4-3]
	\arrow[from=6-2, to=5-1]
	\arrow[from=6-2, to=5-3]
	\arrow[from=7-2, to=6-2]
\end{tikzcd}
\]
\caption{Isomorphism classes of scalar orbit covers over a heptatonic scale, by number of parts $k=1,\dots,7$ (bottom to top).}
\label{fig:isoclasses}
\end{figure}

\subsubsection{The Modal Orbit-Cover Fibration}
\label{subsub:the-modal-orbit-cover-fibration}

The scalar orbit-cover fibration admits a modal analogue, obtained by replacing translation action groupoids with translation-rotation action groupoids. We define a functor
\[
\OrbCovTR:
\int\ActTR
\longrightarrow
\Cat.
\]
An object of $\int\ActTR$ is a triple $([X],Y,\mu_i)$, where $Y\in[X]$ and $\mu_i$ is a mode on $Y$. To such an object we assign the category $\OrbCovTR([X],Y,\mu_i)$, whose objects are modal orbit covers $(Y^{(\mathbf{c})},\mu_i)$, for $\mathbf c \in \mathbf C(n,k)$, and whose morphisms, for fixed $(Y,\mu_i)$, are given by
\[
\mathrm{Hom}_{\OrbCovTR([X],Y,\mu_i)}\bigl((Y^{(\mathbf c)},\mu_i),\,(Y^{(\mathbf d)},\mu_i)\bigr) \coloneqq \mathrm{Hom}_{\mathbf C(n)}(\mathbf c,\mathbf d).
\]
For $f \in \mathrm{Hom}_{\mathbf C(n)}(\mathbf c,\mathbf d)$, we write $\hat f_{i,i} : Y^{(\mathbf c)} \to Y^{(\mathbf d)}$ for the canonical chord bijection it determines via Proposition~\ref{prop:induced-chord-bijection}.

Now let
\[
(\Delta,\bfT_p,\bfR_q):
([X],Y,\mu_i)
\longrightarrow
([W],Z,\mu_j)
\]
be a morphism of $\int\ActTR$. We define
\[
\OrbCovTR(\Delta,\bfT_p,\bfR_q):
\OrbCovTR([X],Y,\mu_i)
\longrightarrow
\OrbCovTR([W],Z,\mu_j)
\]
by
\begin{enumerate}
\item on objects, $(Y^{(\mathbf{c})},\mu_i) \mapsto (Z^{(\mathbf{c})},\mu_j)$;
\item on morphisms, it acts as the identity $f \mapsto f$.
\end{enumerate}
Again, functoriality follows from the equivariance of modal orbit-cover morphisms under deformations (Proposition~\ref{prop:equivariance-of-deformations}).

Applying the Grothendieck construction therefore yields a fibration
\[
P_{\ActTR}:
\int\OrbCovTR
\longrightarrow
\int\ActTR.
\]
An object of $\int\OrbCovTR$ is thus a quadruple $([X],Y,\mu_i,Y^{(\mathbf{c})})$, which we abbreviate as $([X],Y^{(\mathbf{c})},\mu_i)$. The fiber over $([X],Y,\mu_i)$ is precisely the category $\OrbCovTR([X],Y,\mu_i) \cong \mathbf{C}(7)$ of modal orbit covers on the mode $(Y,\mu_i)$. A morphism of $\int\OrbCovTR$ is correspondingly a quadruple
\[
(\Delta,\bfT_p,\bfR_q,f):
([X],Y^{(\mathbf{c})},\mu_i)
\longrightarrow
([W],Z^{(\mathbf{d})},\mu_j),
\]
consisting of a base morphism 
\[
(\Delta,\bfT_p,\bfR_q):([X],Y,\mu_i)\longrightarrow([W],Z,\mu_j)
\] 
of $\int\ActTR$ together with a fiber morphism
\[
f : \bigl(Z^{(\mathbf c)},\mu_j\bigr) \longrightarrow (Z^{(\mathbf d)},\mu_j)
\]
in $\OrbCovTR([W],Z,\mu_j)$, where $f\in\mathrm{Hom}_{\mathbf{C}(7)}(\mathbf c,\mathbf d)$. By Proposition~\ref{prop:induced-chord-bijection}, every such morphism determines a canonical bijection $\hat f_{i,j}:Y^{(\mathbf c)}\to Z^{(\mathbf d)}$, which we use in \S~\ref{sub:harmonic-function-fibrations} to define the harmonic-function fibration $\Func$.

Finally, forgetting the translational part of $f$ and retaining only the linear part $\ell(f)$ defines a functor
\[
\OrbCovTR([X],Y,\mu_i)
\longrightarrow
\OrbCovT([X],Y),
\qquad
(Y^{(\mathbf c)}, \mu_i) \longmapsto Y^{(\mathbf c)}, \qquad
f\longmapsto \ell(f),
\]
which assemble into a natural transformation
\[
\rho_{\bfT}:
\OrbCovTR
\Longrightarrow
\OrbCovT\circ\pi_{\bfT}.
\]
Applying the Grothendieck construction yields a morphism of fibrations
\[
\int\rho_{\bfT}:
\int\OrbCovTR
\longrightarrow
\int\OrbCovT,
\]
covering the projection $\int \pi_{\bfT}:\int\ActTR\to\int\ActT$.

The hierarchy of fibrations constructed thus far is summarized in Figure~\ref{fig:hierarchy-orbcov}.
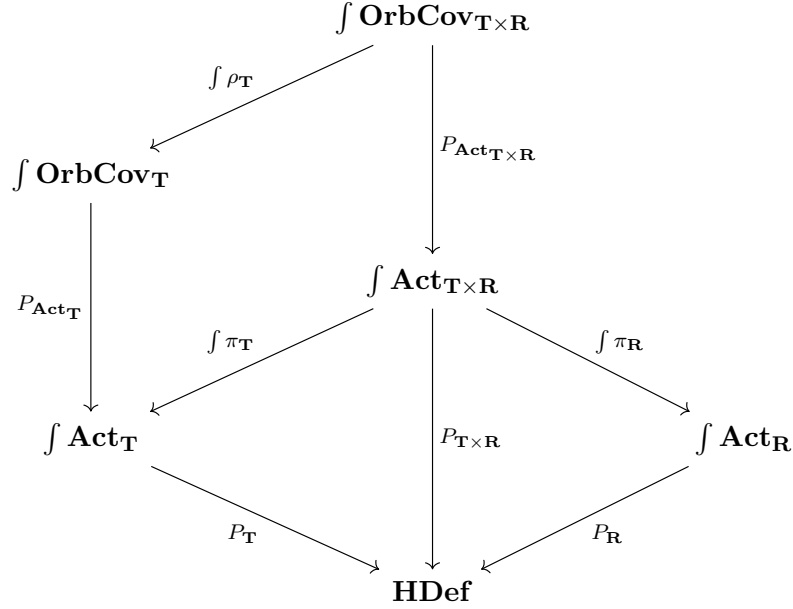
\begin{figure}[h]
\[\begin{tikzcd}
	&& {\int \OrbCovTR} && \\
	\\
	{\int \OrbCovT} \\
	&& {\int \ActTR} \\
	\\
	{\int \ActT} &&&& {\int \ActR} \\
	\\
	&& \HDef
	\arrow["{\int \rho_\bfT}"', from=1-3, to=3-1]
	\arrow["{P_{\ActTR}}", from=1-3, to=4-3]
	\arrow["{P_{\ActT}}"', from=3-1, to=6-1]
	\arrow["{\int \pi_\bfT}"', from=4-3, to=6-1]
	\arrow["{\int \pi_\bfR}", from=4-3, to=6-5]
	\arrow["{P_{\bfT \times \bfR}}", from=4-3, to=8-3]
	\arrow["{P_\bfT}"', from=6-1, to=8-3]
	\arrow["{P_\bfR}", from=6-5, to=8-3]
\end{tikzcd}\]
\caption{The hierarchy of fibrations constructed through \S~\ref{sub:orbit-cover-fibrations}.}
\label{fig:hierarchy-orbcov}
\end{figure}

\paragraph{Restriction to Modal Families}
\label{para:restriction-to-modal-families}

The modal orbit-cover fibration above applies to all seven modes over each translation orbit. In many applications, however, one wishes to work only with a selected collection of modes. For example, common-practice harmony typically uses only the Ionian and Aeolian modes. Such choices are modeled by restricting the ambient translation-rotation action groupoid to a full subgroupoid.

Let $I\subseteq\Z_7$. For each translation orbit $[X]$, let 
\[
\ActTR([X];I) \longhookrightarrow \ActTR([X])
\] 
denote the full subgroupoid of $\ActTR([X])$ whose objects are the modes $(Y,\mu_i)$, for \mbox{$Y\in[X]$} and $i\in I$. Restricting the local orbit-cover functor
\[
\OrbCovTR([X]):
\ActTR([X])
\longrightarrow
\Cat
\]
to this subgroupoid yields a functor
\[
\OrbCovTR([X];I):
\ActTR([X];I)
\longrightarrow
\Cat.
\]
Applying the Grothendieck construction gives a restricted modal orbit-cover fibration
\[
\int\OrbCovTR([X];I)
\longrightarrow
\ActTR([X];I).
\]
Its objects are precisely the modal orbit covers whose modal indices belong to $I$.

If $I\subseteq J$, then the inclusion $\ActTR([X];I) \hookrightarrow \ActTR([X];J)$ induces an inclusion of restricted modal orbit-cover fibrations
\[
\int\OrbCovTR([X];I)
\longhookrightarrow
\int\OrbCovTR([X];J).
\]
Now let $\Delta: \ActTR([X]) \to \ActTR([Y])$ be a deformation functor, and let $\delta:\Z_7\to\Z_7$ be the induced map on modal indices. Then $\Delta$ restricts to a functor
\[
\Delta_I:
\ActTR([X];I)
\longrightarrow
\ActTR([Y];\delta(I)).
\]
Since $\OrbCovTR$ is functorial, this restriction induces a functor
\[
\widetilde{\Delta}_I:
\int\OrbCovTR([X];I)
\longrightarrow
\int\OrbCovTR([Y];\delta(I)),
\]
making the diagram
\[\begin{tikzcd}
	{\int \OrbCovTR([X]; I)} && {\int \OrbCovTR([Y]; \delta(I))} \\
	\\
	{\ActTR([X]; I)} && {\ActTR([Y]; \delta(I))}
	\arrow["{\tilde{\Delta}_I}", from=1-1, to=1-3] 
	\arrow["{P_{\ActTR}}"', from=1-1, to=3-1]
	\arrow["{P_{\ActTR}}", from=1-3, to=3-3]
	\arrow["{\Delta_I}"', from=3-1, to=3-3]
\end{tikzcd}\]
commute.

Thus every deformation induces a morphism between the restricted modal orbit-cover fibrations. Thus modal families arise as full subfibrations of the ambient modal orbit-cover fibration.

\subsubsection{The Cadential Set Construction}
\label{subsub:the-cadential-set-construction}

We now present the cadential-set construction for orbit covers. Cadential sets were formulated by Mazzola \cite[Chapters 26 and 27]{mazzola2002topos} and Muzzulini \cite{muzzulini1995musical} to explain the process of modulation as described by Schoenberg \cite{schoenberg1978theory}; in Mazzola's study, these were identified for diatonic triadic orbit covers. The cadential sets $\mathscr{C}(Y^{(\mathbf c)})$ of an orbit cover $Y^{(\mathbf c)}$ in the translation orbit $[X^{(\mathbf c)}]$ are those subsets $\Gamma\subseteq Y^{(\mathbf c)}$ of triads such that only $Y^{(\mathbf c)}\in[X^{(\mathbf c)}]$ contains these chords. They are called \emph{cadential} because they uniquely identify $Y^{(\mathbf c)}$: a modulation from some other $Z^{(\mathbf c)}$ to $Y^{(\mathbf c)}$ can be regarded as fully accomplished once the chords of $\Gamma$ have occurred, since their occurrence alone firmly establishes the new key $Y^{(\mathbf c)}$ (see Figure~\ref{fig:cadential-set}).

\begin{figure}[htbp]
\centering
\includegraphics[width=0.6\textwidth]{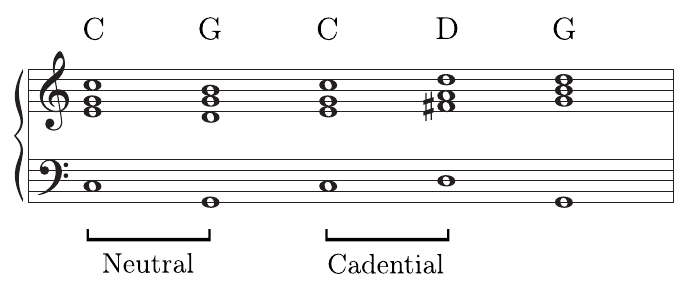}
\caption{The progression $\mathrm{C}$--$\mathrm{G}$--$\mathrm{C}$--$\mathrm{D}$--$\mathrm{G}$, realizing a modulation from C major to G major. The opening $\mathrm{C}$--$\mathrm{G}$ progression is \emph{neutral}: the $\mathrm{C}$-major and $\mathrm{G}$-major triads each occur diatonically in both $\mathrm{C}$ major and $\mathrm{G}$ major, so together they do not yet single out the $\mathrm{G}$ major tonality. Once the second $\mathrm{C}$ chord and the $\mathrm{D}$ chord have sounded, however, $\Gamma=\{\mathrm{C},\mathrm{D}\}$ is \emph{cadential}: $\mathrm{G}$ major is the only orbit cover in the translation orbit whose diatonic triads include both a $\mathrm{C}$-major and a $\mathrm{D}$-major chord, so the pair $\Gamma$ uniquely identifies it. The closing $\mathrm{G}$ chord is the tonic arrival a listener hears as confirming the cadence, but $\mathrm{G}$ major is already determined the moment every chord in $\Gamma$ has sounded.}
\label{fig:cadential-set}
\end{figure}

While we make no use of the cadential-set construction in the theoretical developments that follow, it is nonetheless useful data to have on hand for actually composing modulations within our framework. Later, when we formulate harmonic practices as generalizations of common-practice tonality (\S~\ref{sub:harmonic-practices}), it will be helpful to have cadential sets on hand, so that a composer working within this framework knows how to uniquely target the orbit cover of a modulation.

Note that the cadential-set construction makes no use of modal data, as it depends only on the scalar orbit-cover fibration. Specifically, for each scalar orbit cover $X^{(\mathbf c)}$, the corresponding translation action groupoid $[X^{(\mathbf c)}]\sslash\mathbf T$ forms a canonical subgroupoid of $\int\OrbCovT([X])$, whose objects are the translates of $X^{(\mathbf c)}$ and whose morphisms are translations. It is on these translation action groupoids that the cadential-set construction is defined.

\begin{definition}[Cadential Set]
Let $Y^{(\mathbf c)}\in[X^{(\mathbf c)}]$. A subset $\Gamma\subseteq Y^{(\mathbf c)}$ is called a \emph{cadential set} for $Y^{(\mathbf c)}$ if
\[
\Gamma\subseteq Z^{(\mathbf c)}
\implies
Z^{(\mathbf c)}=Y^{(\mathbf c)}
\]
for every $Z^{(\mathbf c)}\in[X^{(\mathbf c)}]$. A cadential set is called \emph{minimal} if none of its proper subsets is cadential.
\end{definition}

Thus a cadential set is a collection of chords that uniquely identifies an orbit cover within its translation orbit. 

In what follows, the term ``cadential set'' will always mean ``minimal cadential set.''

For each orbit-cover groupoid  $[X^{(\mathbf c)}] \sslash \bfT \hookrightarrow \int \OrbCovT([X])$ and $Y^{(\mathbf c)} \in [X^{(\mathbf c)}]$, define
\[
\mathscr{C}(Y^{(\mathbf c)}) \coloneqq
\left\{
\Gamma\subseteq Y^{(\mathbf c)}
\mid
\Gamma\text{ is a cadential set for }Y^{(\mathbf c)}
\right\}.
\]

\begin{proposition}
Let $\bfT_p : Y^{(\mathbf c)} \to Z^{(\mathbf c)}$ be a translation in $[X^{(\mathbf c)}] \sslash \bfT$. If $\Gamma$ is a cadential set for $Y^{(\mathbf c)}$, then
\[
\bfT_p(\Gamma)=\{\bfT_p(\gamma)\mid\gamma\in\Gamma\}
\]
is a cadential set for $Z^{(\mathbf c)}$.
\end{proposition}

\begin{proof}
Suppose $\bfT_p(\Gamma)\subseteq W^{(\mathbf c)}$ for some $W^{(\mathbf c)}\in[X^{(\mathbf c)}]$. Applying $\bfT_{-p}$ to $\bfT_p(\Gamma)\subseteq W^{(\mathbf c)}$ and using that $\Gamma$ is cadential for $Y^{(\mathbf c)}$ forces $\bfT_{-p}(W^{(\mathbf c)})=Y^{(\mathbf c)}$, hence 
\[
W^{(\mathbf c)}=\bfT_p(Y^{(\mathbf c)})=Z^{(\mathbf c)}; 
\]
so $\bfT_p(\Gamma)$ is cadential.
\end{proof}

Hence every morphism $\bfT_p : Y^{(\mathbf c)} \to Z^{(\mathbf c)}$ induces a map
\[
\begin{matrix}
\bfT_p: & \mathscr{C}(Y^{(\mathbf c)}) & \longrightarrow & \mathscr{C}(Z^{(\mathbf c)}) \\
& \Gamma & \longmapsto & \{\,\bfT_p(\gamma)\mid\gamma\in\Gamma\,\}.
\end{matrix}
\]
Thus,
\[
\mathscr{C}:[X^{(\mathbf c)}] \sslash \bfT \longrightarrow [\mathscr{C}(X^{(\mathbf c)})] \sslash \bfT
\]
defines a functor (in fact, a groupoid isomorphism), assigning to each orbit cover its collection of cadential sets and transporting cadential sets along translations.

\paragraph{No Extension to Orbit-Cover Morphisms}
\label{no-extension-to-orbit-cover-morphisms}

In contrast with translations, an orbit-cover morphism $\hat f_{i,j}:X^{(\mathbf c)}\to X^{(\mathbf d)}$ between two orbit covers of a fixed scale $X$ does not, in general, induce a map $\mathscr{C}(X^{(\mathbf c)})\to\mathscr{C}(X^{(\mathbf d)})$ between the corresponding collections of cadential sets. Thus orbit-cover morphisms do not in general preserve cadentiality.

\begin{example}
Let $X = \{0,2,4,5,7,9,11\}$ be the C diatonic scale. Take $\mathbf c = (2,2,3)$, $\mathbf d= (3,3,1)$, so that $X^{(\mathbf c)}$ is the usual tertian orbit cover and $X^{(\mathbf d)}$ the quartal orbit cover. Define the affine automorphism $f(x)=5x$, which induces the orbit-cover morphism
\[
\begin{matrix}
X^{(\mathbf c)} & \xlongrightarrow{\hat f_{1,1}} & X^{(\mathbf d)} \\
\{0,4,7\} & \longmapsto & \{0,5,11\} \\ 
\{2,5,9\} & \longmapsto & \{9,2,7\} \\
\{4,7,11\} & \longmapsto & \{5,11,4\} \\ 
\{5,9,0\} & \longmapsto & \{2,7,0\} \\
\{7,11,2\} & \longmapsto & \{11,4,9\} \\
\{9,0,4\} &\longmapsto & \{7,0,5\} \\
\{11,2,5\} & \longmapsto & \{4,9,2\}.
\end{matrix}
\]
The cadential sets of $X^{(\mathbf c)}$ are
\[
\begin{matrix}
\{\{11,2,5\}\},\\
\{\{2,5,9\},\{4,7,11\}\},\\
\{\{2,5,9\},\{7,11,2\}\},\\
\{\{5,9,0\},\{4,7,11\}\},\\
\{\{5,9,0\},\{7,11,2\}\},
\end{matrix}
\]
whereas the cadential sets of $X^{(\mathbf d)}$ are
\[
\begin{matrix}
\{\{0,5,11\}\},\\
\{\{5,11,4\}\},\\
\{\{7,0,5\},\{11,4,9\}\}.
\end{matrix}
\]
Under the action of $\hat f_{i,j}$, the singleton cadential set $\{\{11,2,5\}\}$ is mapped to $\{\{4,9,2\}\}$, which is not a cadential set for $X^{(\mathbf d)}$. Hence the image of a cadential set need not be cadential, so $h$ does not induce a map $\mathscr{C}(X^{(\mathbf c)}) \to \mathscr{C}(X^{(\mathbf d)})$.

This has to do with the fundamentally chromatic nature of cadentiality. Whether a collection of chords is cadential depends on the interaction between the chromatic intervallic structure of a scale and the way its orbit cover is realized in that chromatic space. By contrast, an orbit-cover morphism is determined entirely by its action on the abstract scale-degree group $\Z_7$, and therefore does not preserve the chromatic information on which cadentiality depends.
\end{example}

We treat here the case where $h$ relates two covers of a single fixed scale $X$; combined with the translation action of \S~\ref{subsub:the-cadential-set-construction} and the deformation-equivariance of Proposition~\ref{prop:equivariance-of-deformations}, this recovers the general morphism $(\Delta,\bfT_p,h)$ of the ambient category in the evident, componentwise way. 

Every such orbit-cover morphism $h_{i,j}$ from $X^{(\mathbf c)}$ to $X^{(\mathbf d)}$ extends canonically to a functor
\[
h_{i,j}:[X^{(\mathbf c)}]\sslash\mathbf T
\longrightarrow
[X^{(\mathbf d)}]\sslash\mathbf T,
\]
and hence, composing with $\mathscr C$, to a functor $\mathscr C(h_{i,j})$ between the corresponding cadential-set groupoids. This extension costs nothing further. Since $h$ acts on $\Z_7$ while $\bfT_p$ is chromatic transposition, $h_{i,j}$ commutes with every $\bfT_p$ and simply retypes the $\mathbf c$-cover on each translate as the $\mathbf d$-cover on that same translate. Combined with the translation-equivariance of $\mathscr C$ established above, the square
\[\begin{tikzcd}
	{[X^{(\mathbf c)}] \sslash \bfT} && {[X^{(\mathbf d)}] \sslash \bfT} \\
	\\
	{[\mathscr{C}(X^{(\mathbf c)})] \sslash \bfT } && {[\mathscr{C}(X^{(\mathbf d)})] \sslash \bfT }
	\arrow["{h_{i,j}}", from=1-1, to=1-3]
	\arrow["{\mathscr{C}}"', from=1-1, to=3-1]
	\arrow["{\mathscr{C}}", from=1-3, to=3-3]
	\arrow["{\mathscr{C}(h_{i,j})}"', from=3-1, to=3-3]
\end{tikzcd}\]
commutes immediately: both composites send $\bfT_p$ to $\bfT_p$.

\begin{remark}
Let
\[
\widetilde{\Delta}:\int\OrbCovT([X])
\longrightarrow
\int\OrbCovT([Y])
\]
be the morphism of orbit-cover fibrations induced by a deformation $\Delta$, and let
\[
h_{i,j}:[X^{(\mathbf c)}]\sslash\bfT
\longrightarrow
[X^{(\mathbf d)}]\sslash\bfT
\]
be a functor between translation action subgroupoids of $\int\OrbCovT([X])$. These data determine a commuting cube (Figure~\ref{fig:commuting-cube}): the vertical direction is given by $h_{i,j}$, the rightward direction by the deformation-induced functor $\widetilde{\Delta}$, and the leftward direction by the cadential-set construction $\mathscr{C}$. 

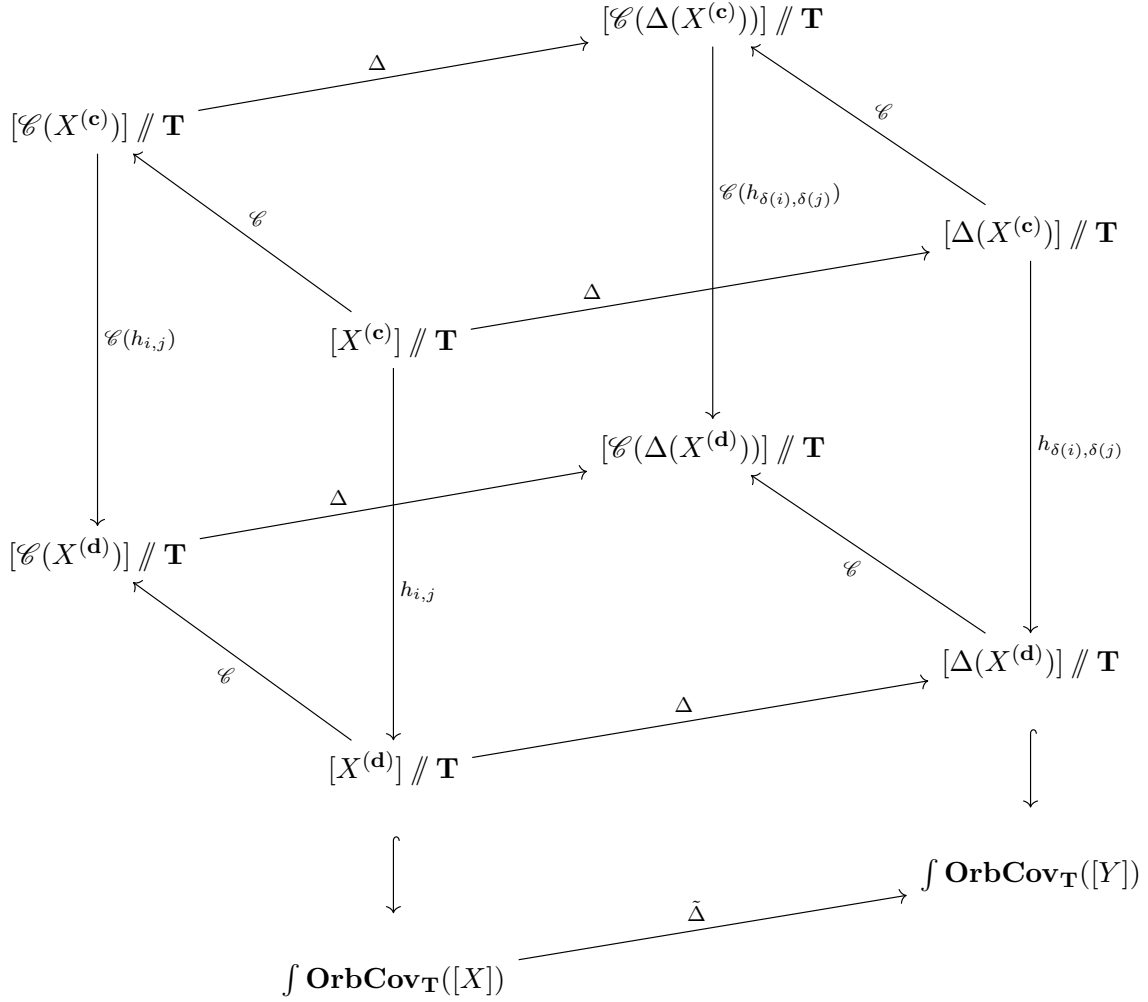
\begin{figure}[htbp]
\[\begin{tikzcd}
	&& {[\mathscr{C}(\Delta(X^{(\mathbf c)}))] \sslash \bfT } & \\
	{[\mathscr{C}(X^{(\mathbf c)})] \sslash \bfT } \\
	&&& {[\Delta(X^{(\mathbf c)})] \sslash \bfT} \\
	& {[X^{(\mathbf c)}] \sslash \bfT} \\
	&& {[\mathscr{C}(\Delta(X^{(\mathbf d)}))] \sslash \bfT } \\
	{[\mathscr{C}(X^{(\mathbf d)})] \sslash \bfT } \\
	&&& {[\Delta(X^{(\mathbf d)})] \sslash \bfT} \\
	& {[X^{(\mathbf d)}] \sslash \bfT} \\
	&&& {\int \OrbCovT([Y])} \\
	& {\int \OrbCovT([X])}
	\arrow["{\mathscr{C}(h_{\delta(i),\delta(j)})}"{pos=0.4}, from=1-3, to=5-3]
	\arrow["\Delta", from=2-1, to=1-3]
	\arrow["{\mathscr{C}(h_{i,j})}", from=2-1, to=6-1]
	\arrow["{\mathscr{C}}"', from=3-4, to=1-3]
	\arrow["{h_{\delta(i),\delta(j)}}", from=3-4, to=7-4]
	\arrow["{\mathscr{C}}"', from=4-2, to=2-1]
	\arrow["\Delta"{pos=0.3}, from=4-2, to=3-4]
	\arrow["{h_{i,j}}"{pos=0.6}, from=4-2, to=8-2]
	\arrow["\Delta"{pos=0.4}, from=6-1, to=5-3]
	\arrow["{\mathscr{C}}", from=7-4, to=5-3]
	\arrow[hook, shorten <=15pt, shorten >=15pt, from=7-4, to=9-4]
	\arrow["{\mathscr{C}}", from=8-2, to=6-1]
	\arrow["\Delta", from=8-2, to=7-4]
	\arrow[hook, shorten <=15pt, shorten >=15pt, from=8-2, to=10-2]
	\arrow["{\tilde{\Delta}}", from=10-2, to=9-4]
\end{tikzcd}\]
\caption{The commuting cube relating orbit-cover morphisms, deformations, and cadential sets. Front vertical arrows are the orbit-cover morphism $h_{i,j}:[X^{(\mathbf c)}]\sslash\bfT\to[X^{(\mathbf d)}]\sslash\bfT$ and its $\Delta$-transport $h_{\delta(i),\delta(j)}$; arrows to the right are induced by the deformation $\Delta$, together with the functor $\widetilde\Delta$ it induces on the ambient total categories; arrows to the left are the cadential-set functor $\mathscr{C}$. Hooked arrows are the canonical inclusions of $[X^{(\mathbf d)}]\sslash\bfT$ and $[\Delta(X^{(\mathbf d)})]\sslash\bfT$ into $\int\OrbCovT([X])$ and $\int\OrbCovT([Y])$, respectively.}
\label{fig:commuting-cube}
\end{figure}
\end{remark}

\subsection{The Deformation Family Fibration}
\label{sub:the-deformation-family-fibrations}

We now enrich modal orbit covers with a family of deformations attached to each. A deformation family determines the allowable substitutions of notes and chordal degrees on both the underlying scale and its orbit cover, furnishing the data needed for harmonic substitution, pivot chords, and modulation, treated in \S~\ref{sub:harmonic-vocabularies}--\ref{sub:inheritance-of-harmonic-practices}.

We begin with a definition.

\begin{definition}[Deformation Family]
Let $([X],Y^{(\mathbf c)},\mu_i)$ be a modal orbit cover. A \emph{deformation family} for $([X],Y^{(\mathbf c)},\mu_i)$ is a subset 
\[
\mathcal{D}\subseteq \coprod_{[Y] \in \mathrm{Ob}(\HDef)} \operatorname{Hom}_{\HDef}([X],[Y]) 
\]
containing the identity morphism $\operatorname{id}_{[X]}$.
\end{definition}

A deformation family induces substitution structures on both the underlying scale and its orbit cover. Define
\begin{equation}
\label{eq:deformation-substitutions}
\mathcal{D}(Y)
\coloneqq
\coprod_{\Delta\in\mathcal D}\Delta(Y),
\qquad
\mathcal{D}(Y^{(\mathbf c)})
\coloneqq
\coprod_{\Delta\in\mathcal D}\Delta(Y^{(\mathbf c)}).
\end{equation}
These come equipped with projection maps
\begin{equation}
\begin{matrix}
p_Y : & \mathcal D(Y) & \longrightarrow & Y \\
& \Delta(y) & \longmapsto & y,
\end{matrix}
\end{equation}
and
\begin{equation}
\label{eq:proj-orb-cov}
\begin{matrix}
p_{Y^{(\mathbf c)}} : & \mathcal D(Y^{(\mathbf c)}) & \longrightarrow & Y^{(\mathbf c)} \\
& \Delta(A) & \longmapsto & A.
\end{matrix}
\end{equation}
Composing these projections with the degree maps $\mu_i:Y\to\mathbb Z_n$ and $\deg:Y^{(\mathbf c)}\to\mathbb Z_n$ (recalling \eqref{eq:chordal-degree}) relates substitutions to their scalar and chordal degrees.

The following definition identifies those chords common to two orbit covers as \emph{pivot chords}---the notion underlying the secondary chords and pivot-chord modulations discussed later (\S~\ref{sub:harmonic-vocabularies}--\ref{sub:inheritance-of-harmonic-practices}).

\begin{definition}[Direct Pivot Chords]
\label{def:direct-pivot-chords}
Let $Y^{(\mathbf c)}, Z^{(\mathbf c)} \in [X^{(\mathbf c)}]$ be orbit covers over the same translation orbit. Their \emph{direct pivot chords} are $Y^{(\mathbf c)}\cap Z^{(\mathbf c)}$.
\end{definition}
After equipping orbit covers with deformation families, additional pivot constructions become available. For deformation-equipped orbit covers
\[
M_1=([X],Y^{(\mathbf c)},\mu_i,\mathcal D),
\qquad
M_2=([X],Z^{(\mathbf c)},\mu_j,\mathcal E),
\]
there are four natural intersections:
\[
Y^{(\mathbf c)}\cap Z^{(\mathbf c)},\qquad
\mathcal D(Y^{(\mathbf c)})\cap Z^{(\mathbf c)},\qquad
Y^{(\mathbf c)}\cap\mathcal E(Z^{(\mathbf c)}),\qquad
\mathcal D(Y^{(\mathbf c)})\cap\mathcal E(Z^{(\mathbf c)}).
\]
The first of these is just the direct pivot chords of $Y^{(\mathbf c)}$ and $Z^{(\mathbf c)}$ (Definition~\ref{def:direct-pivot-chords}). Each of the middle two is a \emph{one-sided deformation pivot}: the \emph{left} one-sided pivot $\mathcal D(Y^{(\mathbf c)})\cap Z^{(\mathbf c)}$ and the \emph{right} one-sided pivot $Y^{(\mathbf c)}\cap\mathcal E(Z^{(\mathbf c)})$. The last is the \emph{two-sided deformation pivot}. We therefore distinguish three types of pivot chords: direct pivots, one-sided deformation pivots (of which there are two, left and right, per pair $M_1,M_2$), and two-sided deformation pivots.

\begin{example}[Secondary Chords]
\label{ex:secondary-chords}
Consider C major and G major as modal orbit covers
\[
C=(Y^{(\mathbf c)},\mu_1)
\qquad\text{and}\qquad
G=(Z^{(\mathbf c)},\mu_1).
\]
Let $\mathcal D$ consist of the single deformation that raises the fourth degree of the C-major scale, replacing $\mathrm{F}$ by $\mathrm{F}^\sharp$. Then the left-sided deformation pivots $\mathcal D(Y^{(\mathbf c)})\cap Z^{(\mathbf c)}$ contain the D-major triad (Figure~\ref{fig:secondary-chord-projection}). This chord is not a member of the original C-major orbit cover, but it becomes available after deformation and is simultaneously the diatonic dominant triad of G major.

Thus the left-sided deformation pivots identify chords obtained by altering the source harmonic vocabulary so that they become degrees of the target orbit cover. In the common-practice setting, these correspond precisely to secondary chords of the target key. (For a more thorough formulation, see Example~\ref{ex:secondary-functions}.)
\end{example}

\begin{figure}[htbp]
\centering
\includegraphics[width=0.75\textwidth]{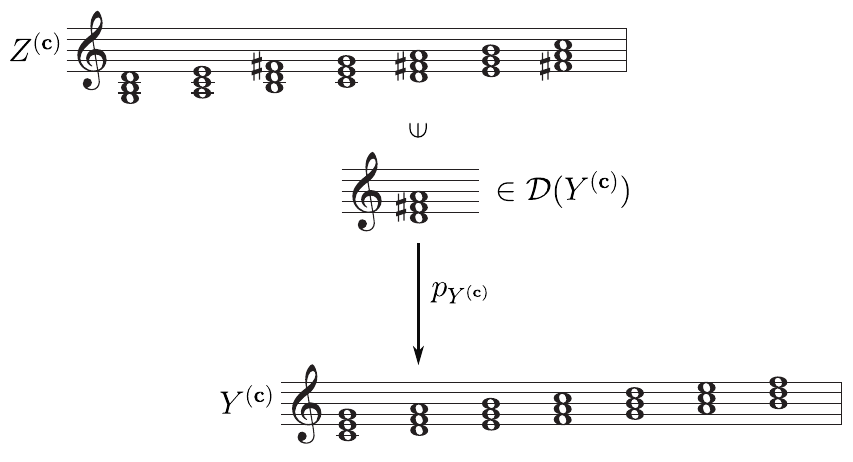}
\caption{The pivot chord of Example~\ref{ex:secondary-chords}. Applying the deformation $\mathcal D$ (which raises $\mathrm{F}$ to $\mathrm{F}^\sharp$) to the ii chord of $Y^{(\mathbf c)}$ produces the D-major triad shown in $\mathcal D(Y^{(\mathbf c)})$, which is exhibited here as one of the seven diatonic triads of $Z^{(\mathbf c)}$. The projection $p_{Y^{(\mathbf c)}}:\mathcal D(Y^{(\mathbf c)})\to Y^{(\mathbf c)}$ sends this deformed chord back to its undeformed counterpart, the D-minor triad of $Y^{(\mathbf c)}$.}
\label{fig:secondary-chord-projection}
\end{figure}

The interaction between deformation families and the cadential set construction (\S~\ref{subsub:the-cadential-set-construction}) provides a natural framework for generalized modulation within the orbit cover setting. A composition may begin in the modal orbit cover $M_1$, pass through pivot chords between $M_1$ and $M_2$, and then establish the destination orbit cover $M_2$ by means of a cadential set in $M_2$. Modulation is one of the principal motivations for introducing deformation families.

We now show that deformation families give rise to a new fibrational layer in our theory.

\subsubsection{Definition of the Deformation-Family Fibration}
\label{subsub:the-deformation-family-fibration}

We define a functor
\[
\DefFam:\int\OrbCovTR\longrightarrow\mathbf{Pos},
\]
into the category of posets as follows.  For a modal orbit cover $M=([X],Y^{(\mathbf c)},\mu_i)$, let $\DefFam(M)$ denote the category whose
\begin{enumerate}
\item objects are deformation families $\mathcal D$ over $M$;
\item morphisms are inclusions $\mathcal D\hookrightarrow\mathcal D'$.
\end{enumerate}
Given a morphism
\[
(\Delta,\bfT_p,\bfR_q,f):
([X],Y^{(\mathbf c)},\mu_i)
\longrightarrow
([W],Z^{(\mathbf d)},\mu_j)
\]
in $\int\OrbCovTR$, define
\[
\DefFam(\Delta,\bfT_p,\bfR_q,f):
\DefFam(M)
\longrightarrow
\DefFam(N),
\]
where $N=([W],Z^{(\mathbf d)},\mu_j)$, by
\begin{enumerate}
\item on objects, 
\begin{equation}
\label{eq:def-on-objects} 
\mathcal D
\longmapsto
\Delta\circ\mathcal D \coloneqq
\{\Gamma \circ \Delta^{-1}  \mid\Gamma\in\mathcal D\};
\end{equation}

\item on morphisms, $\mathcal D \to \mathcal D'$ is sent to the induced inclusion $\Delta\circ\mathcal D
\hookrightarrow
\Delta\circ\mathcal D'$.
\end{enumerate}
This determines a functor
\[
\DefFam:\int\OrbCovTR\longrightarrow\mathbf{Pos},
\]
and hence a Grothendieck fibration
\[
P_{\DefFam} : \int\DefFam\longrightarrow\int\OrbCovTR.
\]
Thus the objects of $\int\DefFam$ are quadruples $([X],Y^{(\mathbf c)},\mu_i,\mathcal D)$.

The hierarchy of fibrations constructed thus far is summarized in Figure~\ref{fig:hierarchy-deffam}.
\begin{figure}[h]
\[\begin{tikzcd}
	&& {\int \DefFam} && \\
	\\
	&& {\int \OrbCovTR} \\
	\\
	{\int \OrbCovT} \\
	&& {\int \ActTR} \\
	\\
	{\int \ActT} &&&& {\int \ActR} \\
	\\
	&& \HDef
	\arrow["{P_{\OrbCovTR}}", from=1-3, to=3-3]
	\arrow["{\int\rho_\bfT}"', from=3-3, to=5-1]
	\arrow["{P_{\ActTR}}", from=3-3, to=6-3]
	\arrow["{P_{\ActT}}"', from=5-1, to=8-1]
	\arrow["{\int\pi_\bfT}"', from=6-3, to=8-1]
	\arrow["{\int\pi_\bfR}", from=6-3, to=8-5]
	\arrow["{P_{\bfT \times \bfR}}", from=6-3, to=10-3]
	\arrow["{P_\bfT}"', from=8-1, to=10-3]
	\arrow["{P_\bfR}", from=8-5, to=10-3]
\end{tikzcd}\]
\caption{The hierarchy of fibrations constructed through \S~\ref{sub:the-deformation-family-fibrations}.}
\label{fig:hierarchy-deffam}
\end{figure}
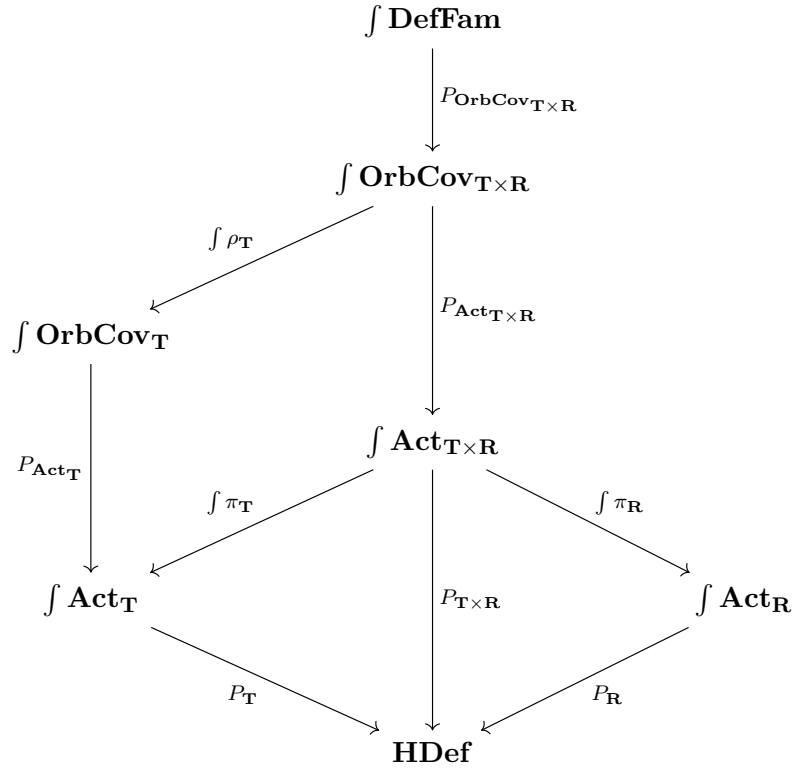

\subsection{Harmonic Function Fibrations}
\label{sub:harmonic-function-fibrations}

The final fibrational layer equips the chords of an orbit cover with harmonic functions. We will not attempt to derive harmonic function from more primitive musical principles; instead we treat it as an abstract specification assigned to chords. This matters because, as we will see in \S~\ref{sub:harmonic-practices}, harmonic syntax is defined at the level of harmonic functions rather than directly on chords. Once functions have been assigned to the chords of an orbit cover, the syntax transports to those chords, so the same abstract syntax can be realized on any orbit cover through a suitable functional assignment.  We now formalize this idea by introducing a fibrational layer whose objects consist of orbit covers together with harmonic function assignments.

Before presenting the construction, we briefly recall the basic notions of harmonic function. In common-practice tonality, the principal harmonic functions are tonic (T), subdominant (S), and dominant (D), represented by the I, IV, and V triads. Other diatonic chords are typically regarded as refinements of these core functions. For example, in Riemannian function theory \cite{riemann1896harmony}, the vi chord in a major key is called the \emph{tonic parallel} (Tp), the ii chord the \emph{subdominant parallel} (Sp), and the iii chord the \emph{dominant parallel} (Dp), although iii is also sometimes interpreted as the \emph{tonic counter-parallel} (Tcp).

Abstracting from this situation, we will define a \emph{harmonic function schema} as a set $F_{\operatorname{core}}$ of core harmonic functions, a set $F_{\operatorname{spec}}$ of specialized harmonic functions, and a surjective map
\[
\varphi : F_{\operatorname{spec}} \longrightarrow F_{\operatorname{core}}
\]
assigning each specialized function to its underlying core function. Then, given a harmonic function schema $(F_{\operatorname{core}}, F_{\operatorname{spec}}, \varphi)$ and an orbit cover $Y^{(\mathbf c)}$, a \emph{functional assignment} is a map
\[
\Phi:Y^{(\mathbf c)}\longrightarrow\mathcal{P}(F_{\operatorname{spec}}),
\]
assigning to each chord $A\in Y^{(\mathbf c)}$ its set of harmonic functions. Composing with $\varphi$, extended pointwise to power sets, recovers the corresponding set of core harmonic functions.

This leads to two harmonic function fibrations. The first is defined over $\int\OrbCovTR$, where functions are assigned to the chords of a modal orbit cover. The second is defined over $\int\DefFam$, allowing the original chords and their deformation images to carry independent functional assignments. 

\subsubsection{Harmonic Functions Independent of Deformation Families}
\label{subsub:harmonic-functions-independent-of-deformation-families}

\begin{definition}[Harmonic Function Schema]
A \emph{harmonic function schema} is a triple  $\mathcal{F} = (F_{\operatorname{core}}, F_{\operatorname{spec}}, \varphi)$, consisting of:
\begin{enumerate}
\item a set $F_{\operatorname{core}}$ of core harmonic functions;
\item a set $F_{\operatorname{spec}}$ of specialized harmonic functions;
\item a surjection $\varphi : F_{\operatorname{spec}} \to F_{\operatorname{core}}$.
\end{enumerate}
\end{definition}

In the principal examples of this paper, the core functions form a subset $F_{\core}\subseteq F_{\spec}$. In this situation we assume that $\varphi(\alpha)=\alpha$, for every $\alpha\in F_{\core}$, so that the core functions are fixed points of $\varphi$. Equivalently, $\varphi$ is an idempotent retraction of $F_{\spec}$ onto $F_{\core}$:
\[\begin{tikzcd}
	{F_{\core}} && {F_{\spec}} \\
	\\
	&& {F_{\core}}
	\arrow[hook, from=1-1, to=1-3]
	\arrow["{\mathrm{id}}"', from=1-1, to=3-3]
	\arrow["\varphi", from=1-3, to=3-3]
\end{tikzcd}\]

\begin{example} 
\label{ex:harmonic-function-schema}
Let
\[
F_{\spec} =
\{\mathrm{T},\mathrm{Tp},\mathrm{Tcp},
  \mathrm{S},\mathrm{Sp},
  \mathrm{D},\mathrm{Dp}\},
\qquad
F_{\core} =
\{\mathrm{T},\mathrm{S},\mathrm{D}\},
\]
and define $\varphi$ so that it fixes each element of $F_{\core}$ while sending every specialized function to its underlying core function. We will use this later in
\S~\ref{subsub:rewrite-rules} when formalizing harmonic grammars.
\end{example}

\begin{definition}[Functional Assignment]
Let $M=([X],Y^{(\mathbf c)},\mu_i)$ be a modal orbit cover, and let $\mathcal{F} = (F_{\operatorname{core}},F_{\operatorname{spec}},\varphi)$ be a harmonic function schema. A \emph{functional assignment} for $M$ with respect to this schema is a map
\[
\Phi:Y^{(\mathbf c)}\longrightarrow\mathcal P(F_{\spec}).
\]
By abuse of notation, we also write
\[
\begin{matrix}
\varphi : & \mathcal P(F_{\operatorname{spec}})& \longrightarrow& \mathcal P(F_{\operatorname{core}}) \\
& G & \longmapsto & \{\varphi(\alpha)\mid \alpha\in G\},
\end{matrix}
\]
for the induced map on power sets. The corresponding assignment of core harmonic functions is therefore
\begin{equation}
\label{eq:core-funcs}
\Phi_{\operatorname{core}} \coloneqq\varphi\circ\Phi:
Y^{(\mathbf c)}\longrightarrow\mathcal P(F_{\operatorname{core}}).
\end{equation}
\end{definition}

A harmonic function schema therefore fixes the abstract vocabulary of a harmonic practice---the function-types on which harmonic syntax operates---while a functional assignment realizes that abstract vocabulary on the concrete chords of a particular orbit cover.

\begin{definition}[Realizations of a Function]
\label{def:realizations-of-function}
For each core harmonic function $\alpha\in F_{\operatorname{core}}$, the set of chords realizing $\alpha$ is
\begin{equation}
\label{eq:realizations-of-function}
\operatorname{Real}_{\Phi}(\alpha)
\coloneqq
\{\,A\in Y^{(\mathbf c)}
\mid
\alpha\in\Phi_{\operatorname{core}}(A)
\,\}.
\end{equation}
\end{definition}

\begin{remark}
Writing
\begin{equation}
\label{eq:up-alpha}
\uparrow\alpha
\coloneqq
\{\,G\subseteq F_{\operatorname{core}}\mid\alpha\in G\,\}
\end{equation}
for the principal up-set of $\alpha$ in $\mathcal P(F_{\operatorname{core}})$, Equation~\eqref{eq:realizations-of-function} is equivalently 
\[
\operatorname{Real}_{\Phi}(\alpha) = \Phi_{\operatorname{core}}^{-1}(\uparrow\alpha).
\]
\end{remark}

\begin{definition}
\label{def:func-fibration}
Let $M=([X],Y^{(\mathbf c)},\mu_i)$ be a modal orbit cover. The category $\Func(M)$ is defined as follows.
\begin{enumerate}
\item Its objects are pairs $(\mathcal{F}, \Phi)$ where $\mathcal{F} = (F_{\operatorname{core}}, F_{\operatorname{spec}}, \varphi)$ is a harmonic function schema and $\Phi$ is a functional assignment $\Phi : Y^{(\mathbf c)} \to \mathcal{P}({F_{\spec}})$ with respect to this schema. 
\item A morphism $(\mathcal{F}, \Phi)  \to (\mathcal{F}',\Phi')$ is a set map $g_{\core}:F_{\operatorname{core}} \to F'_{\operatorname{core}}$ such that, for every $A \in Y^{(\mathbf c)}$,
\[
g_{\core}(\Phi_{\operatorname{core}}(A))
\subseteq
\Phi'_{\operatorname{core}}(A).
\]
\end{enumerate}
\end{definition}

Condition (2) asks only for an inclusion, not equality. This is because a morphism $(\mathcal{F},\Phi)\to(\mathcal{F}',\Phi')$ should be understood as an \emph{enrichment} of the harmonic functions available for each chord $A\in Y^{(\mathbf c)}$, rather than a rigid transport of it. Equality would force $g_{\core}$ to carry the functions assigned to $A$ in $\mathcal F$ onto \emph{exactly} the functions assigned to $A$ in $\mathcal F'$, so that any function recognized for $A$ in the target schema but absent from the image of the source would obstruct the morphism entirely. Under the subset condition, $\Phi'_{\operatorname{core}}(A)$ may properly contain the image of $\Phi_{\operatorname{core}}(A)$; the target schema is free to recognize additional harmonic functions for $A$ that the source schema did not distinguish. For instance, a chord assigned only the tonic function under $\Phi$ may, under a richer schema $\Phi'$, also be assigned a function such as tonic substitute, without this preventing $g_{\core}$ from being a morphism. Morphisms in $\Func(M)$ thus point from coarser schemas toward richer ones, requiring only that no functional assignment already present in the source be lost.

\begin{example}
\label{ex:motte-rohrmeier}
Fix the modal orbit cover $(Y^{(\mathbf c)}, \mu_i)$ whose chords are the diatonic triads of the major mode, labeled with Roman numerals $Y^{(\mathbf c)} = \{\mathrm I, \mathrm{II}, \dots, \mathrm{VII}\}$. Let every schema below share the core function set $F_{\operatorname{core}} = \{\mathrm T, \mathrm S, \mathrm D\}$ (tonic, subdominant, dominant), so that a functional assignment is a map $\Phi_{\operatorname{core}}: Y^{(\mathbf c)} \to \mathcal P(F_{\operatorname{core}})$.

De la Motte \cite{motte1991study} assigns
\begin{align*}
\Phi^{\mathrm{dlM}}_{\operatorname{core}}(\mathrm I) = \Phi^{\mathrm{dlM}}_{\operatorname{core}}(\mathrm{VI}) &= \{\mathrm T\}, \\
\Phi^{\mathrm{dlM}}_{\operatorname{core}}(\mathrm{III}) = \Phi^{\mathrm{dlM}}_{\operatorname{core}}(\mathrm V) &= \{\mathrm D\}, \\
\Phi^{\mathrm{dlM}}_{\operatorname{core}}(\mathrm{II}) = \Phi^{\mathrm{dlM}}_{\operatorname{core}}(\mathrm{IV}) &= \{\mathrm S\}, \\
\Phi^{\mathrm{dlM}}_{\operatorname{core}}(\mathrm{VII}) &= \varnothing.
\end{align*}
Rohrmeier \cite{rohrmeier2011towards} assigns
\begin{align*}
\Phi^{R}_{\operatorname{core}}(\mathrm I) = \Phi^{R}_{\operatorname{core}}(\mathrm{III}) = \Phi^{R}_{\operatorname{core}}(\mathrm{VI}) &= \{\mathrm T\}, \\
\Phi^{R}_{\operatorname{core}}(\mathrm{II}) = \Phi^{R}_{\operatorname{core}}(\mathrm{IV}) &= \{\mathrm S\}, \\
\Phi^{R}_{\operatorname{core}}(\mathrm V) = \Phi^{R}_{\operatorname{core}}(\mathrm{VII}) &= \{\mathrm D\}.
\end{align*}

We show that no morphism relates $\mathrm{dlM}$ and $R$. A morphism $\mathrm{dlM} \to R$ would be a map $g_{\operatorname{core}}: F_{\operatorname{core}} \to F_{\operatorname{core}}$ with $g_{\operatorname{core}}(\Phi^{\mathrm{dlM}}_{\operatorname{core}}(A)) \subseteq \Phi^{R}_{\operatorname{core}}(A)$ for every $A$. Comparing the two schemas at $\mathrm I$ and $\mathrm V$ forces the values of $g_{\operatorname{core}}$ on $\mathrm T$ and $\mathrm D$, and comparing at $\mathrm{III}$ then breaks the inclusion:
\begin{align*}
\Phi^{\mathrm{dlM}}_{\operatorname{core}}(\mathrm I) = \Phi^{R}_{\operatorname{core}}(\mathrm I) = \{\mathrm T\}
&\implies g_{\operatorname{core}}(\mathrm T) = \mathrm T, \\
\Phi^{\mathrm{dlM}}_{\operatorname{core}}(\mathrm V) = \Phi^{R}_{\operatorname{core}}(\mathrm V) = \{\mathrm D\}
&\implies g_{\operatorname{core}}(\mathrm D) = \mathrm D, \\
g_{\operatorname{core}}\big(\Phi^{\mathrm{dlM}}_{\operatorname{core}}(\mathrm{III})\big) = g_{\operatorname{core}}(\{\mathrm D\}) = \{\mathrm D\}
&\not\subseteq \{\mathrm T\} = \Phi^{R}_{\operatorname{core}}(\mathrm{III}).
\end{align*}
So no such $g_{\operatorname{core}}$ exists. Symmetrically, a morphism $R \to \mathrm{dlM}$ would need $g_{\operatorname{core}}(\mathrm T) = \mathrm T$ and $g_{\operatorname{core}}(\mathrm D) = \mathrm D$ from the same two chords, forcing
\[
g_{\operatorname{core}}\big(\Phi^{R}_{\operatorname{core}}(\mathrm{III})\big) = g_{\operatorname{core}}(\{\mathrm T\}) = \{\mathrm T\} \not\subseteq \{\mathrm D\} = \Phi^{\mathrm{dlM}}_{\operatorname{core}}(\mathrm{III}),
\]
so this direction fails too. Both schemas disagree about the function of $\mathrm{III}$.

We now provide a common enrichment that absorbs both de la Motte's and Rohrmeier's assignments. Take $\mathrm I, \mathrm{IV}, \mathrm V$ as the primary $\mathrm T, \mathrm S, \mathrm D$ representatives, and admit any chord a third above or below one of these as a substitute for it:
\begin{align*}
\Phi^{L}_{\operatorname{core}}(\mathrm I) &= \{\mathrm T\}, \\
\Phi^{L}_{\operatorname{core}}(\mathrm{II}) &= \{\mathrm S\}, \\
\Phi^{L}_{\operatorname{core}}(\mathrm{III}) &= \{\mathrm T, \mathrm D\}, \\
\Phi^{L}_{\operatorname{core}}(\mathrm{IV}) &= \{\mathrm S\}, \\
\Phi^{L}_{\operatorname{core}}(\mathrm V) &= \{\mathrm D\}, \\
\Phi^{L}_{\operatorname{core}}(\mathrm{VI}) &= \{\mathrm T, \mathrm S\}, \\
\Phi^{L}_{\operatorname{core}}(\mathrm{VII}) &= \{\mathrm D\}
\end{align*}
($\mathrm{III}$ sits a third from both $\mathrm I$ and $\mathrm V$; $\mathrm{VI}$ a third from both $\mathrm I$ and $\mathrm{IV}$). Taking $g_{\core} = \operatorname{id}_{F_{\operatorname{core}}}$, at $\mathrm I, \mathrm{II}, \mathrm{IV}, \mathrm V$ all three schemas assign the same singleton, so the inclusion is immediate. At $\mathrm{VII}$, $\Phi^{\mathrm{dlM}}_{\operatorname{core}}(\mathrm{VII}) = \varnothing$ is trivially contained in $\Phi^{R}_{\operatorname{core}}(\mathrm{VII}) = \Phi^{L}_{\operatorname{core}}(\mathrm{VII}) = \{\mathrm D\}$. The only chords where $\Phi^{L}_{\operatorname{core}}(B)$ is not a singleton are $\mathrm{III}$ and $\mathrm{VI}$---precisely where $\mathrm{dlM}$ and $R$ disagreed, and the liberal schema simply admits both readings:
\[
\Phi^{\mathrm{dlM}}_{\operatorname{core}}(\mathrm{III}) = \{\mathrm D\} \subseteq \{\mathrm T, \mathrm D\} = \Phi^{L}_{\operatorname{core}}(\mathrm{III}), \qquad
\Phi^{R}_{\operatorname{core}}(\mathrm{III}) = \{\mathrm T\} \subseteq \{\mathrm T, \mathrm D\} = \Phi^{L}_{\operatorname{core}}(\mathrm{III}).
\]
So $\operatorname{id}$ gives morphisms $\mathrm{dlM} \to L$ and $R \to L$ in $\Func(M)$; $L$ therefore enriches both schemas at once, resolving their disagreement about $\mathrm{III}$ by admitting both $\mathrm{T}$ and $\mathrm{D}$.
\end{example}

\begin{remark}
\label{rem:no-compatibility-required}
In Definition~\ref{def:func-fibration}, notice that no compatibility is required between the specialized function sets $F_{\operatorname{spec}}$ and $F'_{\operatorname{spec}}$. One could instead require the existence of a map $g_{\spec}:F_{\operatorname{spec}}\to F'_{\operatorname{spec}}$ making the diagram
\[\begin{tikzcd}
	{F_{\operatorname{spec}}} && {F'_{\operatorname{spec}}} \\
	\\
	{F_{\operatorname{core}}} && {F'_{\operatorname{core}}}
	\arrow["{g_{\operatorname{spec}}}", from=1-1, to=1-3]
	\arrow["\varphi"', from=1-1, to=3-1]
	\arrow["{\varphi'}", from=1-3, to=3-3]
	\arrow["{g_{\operatorname{spec}}}"', from=3-1, to=3-3]
\end{tikzcd}\]
commute, and further such that, for every $A \in Y^{(\mathbf c)}$,
\[
g_{\spec}(\Phi(A)) \subseteq \Phi'(A).
\]
We do not impose this stronger condition. The specialized functions merely refine the core harmonic functions by specifying possible substitutions for them, whereas harmonic syntax itself operates only on the core functions, as we will see in \S~\ref{sub:harmonic-practices}. Consequently, the induced assignment $\Phi_{\operatorname{core}}=\varphi\circ\Phi$ is the syntactically relevant datum, and only the inclusion
\[
g_{\core}(\Phi_{\operatorname{core}}(A)) \subseteq \Phi'_{\operatorname{core}}(A)
\]
is required to hold for morphisms---no comparison of $\Phi(A)$ and $\Phi'(A)$ themselves, nor any $g_{\spec}$, is called for. The map $\varphi:F_{\spec} \to F_{\core}$ then records the substitutional structure associated with a given functional assignment.
\end{remark}

We now show that $\Func$ acts as a functor
\[
\Func:\int\OrbCovTR\longrightarrow\Cat.
\]
It sends each modal orbit cover $M$ to the category $\Func(M)$. Given a morphism
\[
(\Delta,\bfT_p,\bfR_q,f):
([X],Y^{(\mathbf c)},\mu_i)
\longrightarrow
([W],Z^{(\mathbf d)},\mu_j),
\]
we define the functor
\[
\Func(\Delta,\bfT_p,\bfR_q,f):
\Func([X],Y^{(\mathbf c)},\mu_i)
\longrightarrow
\Func([W],Z^{(\mathbf d)},\mu_j)
\]
as follows:
\begin{enumerate}
\item On objects, $(\mathcal{F},\Phi)
\mapsto
(\mathcal{F}, \Phi_{\hat f_{i,j}})$, where
\begin{equation}
\label{eq:push-phi-forward}
\Phi_{\hat f_{i,j}} \coloneqq\Phi\circ (\hat f_{i,j})^{-1}:
Z^{(\mathbf d)}\longrightarrow\mathcal P(F_{\spec}),
\end{equation}
with $\hat f_{i,j}$ as in Proposition~\ref{prop:induced-chord-bijection}. Equivalently, $\Phi_{\hat f_{i,j}}$ is the unique map making the diagram
\[\begin{tikzcd}
	{Y^{(\mathbf c)}} && {Z^{(\mathbf d)}} \\
	\\
	& {\mathcal{P}(F_{\spec})}
	\arrow["{\hat f_{i,j}}", from=1-1, to=1-3]
	\arrow["\Phi"', from=1-1, to=3-2]
	\arrow["{\Phi_{\hat f_{i,j}}}", from=1-3, to=3-2]
\end{tikzcd}\]
commute.
\item On a morphism $g_{\core} : F_{\core} \to F'_{\core}$, it acts by $g_{\core} \mapsto g_{\core}$.
\end{enumerate}

It remains only to verify that the assignment on morphisms is well defined. This follows immediately, since
\begin{align*}
\Phi_{\hat f_{i,j}, \core}
& =
\varphi\circ\Phi_{\hat f_{i,j}} \\
& =
\varphi\circ\Phi\circ (\hat f_{i,j})^{-1} \\
& =
\Phi_{\core} \circ (\hat f_{i,j})^{-1},
\end{align*}
and likewise
\[
\Phi'_{\hat f_{i,j}, \core}
=
\Phi'_{\core} \circ (\hat f_{i,j})^{-1}.
\]
Hence, if
\[
g_{\core}(\Phi_{\core}(A))\subseteq\Phi'_{\core}(A)
\]
for every $A\in Y^{(\mathbf c)}$, then
\[
g_{\core}(\Phi_{\hat f_{i,j},\core}(B))
\subseteq
\Phi'_{\hat f_{i,j},\core}(B)
\]
for every $B\in Z^{(\mathbf d)}$.

Thus, applying the Grothendieck construction to $\Func$ yields a fibration
\[
\int\Func\longrightarrow\int\OrbCovTR.
\]

\begin{example}
Let $X=\{0,2,4,5,7,9,11\}$ be the C-major scale, with the Ionian mode $\mu_1$ and Lydian mode $\mu_4$. Let $\mathbf c = (2,2,3)$ and $\mathbf d = (2,1,1,3)$, so $X^{(\mathbf c)}$ is the tertian orbit cover and $X^{(\mathbf d)}$, with $\hat{\mathbf d}(x)=\{x,x+2,x+3,x+4\}\pmod 7$, consists of root-third-fourth-fifth sonorities.

Take $u=5$. We seek $b\in\Z_7$ such that $f(x)=5x+b$ satisfies $f(\hat{\mathbf c}(0))\subseteq\hat{\mathbf d}(0)$, i.e.\ $5\cdot\{0,2,4\}+b\subseteq\{0,2,3,4\}$. Since $5\cdot\{0,2,4\}=\{0,10,20\}\pmod 7=\{0,3,6\}$, checking each $b\in\Z_7$ shows only $b=4$ works: $\{0,3,6\}+4=\{4,0,3\}\subseteq\{0,2,3,4\}$, while every other value of $b$ fails to land inside $\{0,2,3,4\}$. Thus
\[
f(x) = 5x+4
\]
is the unique element of $\mathrm{Hom}_{\mathbf C(7)}(\mathbf c,\mathbf d)$ with linear part $u=5$.

By Remark~\ref{rem:induced-map-in-composition-notation},
\[
\hat f(\hat{\mathbf c}(z)) = \hat{\mathbf d}(5z-4) = \hat{\mathbf d}(5z+3).
\]
Working through all seven triads, and reading the codomain's Roman numerals relative to $\mu_4$'s own tonic gives:
\begin{align*}
\mathrm I=\{\mathrm C,\mathrm E,\mathrm G\} &\longmapsto \{\mathrm F,\mathrm A,\mathrm B,\mathrm C\}=\mathrm {I}, \\
\mathrm{II}=\{\mathrm D,\mathrm F,\mathrm A\} &\longmapsto \{\mathrm D,\mathrm F,\mathrm G,\mathrm A\}=\mathrm {VI}, \\
\mathrm{III}=\{\mathrm E,\mathrm G,\mathrm B\} &\longmapsto \{\mathrm B,\mathrm D,\mathrm E,\mathrm F\}=\mathrm {IV}, \\
\mathrm{IV}=\{\mathrm F,\mathrm A,\mathrm C\} &\longmapsto \{\mathrm G,\mathrm B,\mathrm C,\mathrm D\}=\mathrm {II}, \\
\mathrm V=\{\mathrm G,\mathrm B,\mathrm D\} &\longmapsto \{\mathrm E,\mathrm G,\mathrm A,\mathrm B\}=\mathrm {VII}, \\
\mathrm{VI}=\{\mathrm A,\mathrm C,\mathrm E\} &\longmapsto \{\mathrm C,\mathrm E,\mathrm F,\mathrm G\}=\mathrm {V}, \\
\mathrm{VII}=\{\mathrm B,\mathrm D,\mathrm F\} &\longmapsto \{\mathrm A,\mathrm C,\mathrm D,\mathrm E\}=\mathrm {III}.
\end{align*}

Consider the morphism
\[
(\operatorname{id},\bfT_0,\bfR_3,f):
([X],Y^{(\mathbf c)},\mu_1)
\longrightarrow
([X],Z^{(\mathbf d)},\mu_4)
\]
of $\int\OrbCovTR$, and let $(\mathcal F,\Phi)\in\Func([X],Y^{(\mathbf c)},\mu_1)$ be de la Motte's schema from Example~\ref{ex:motte-rohrmeier}:
\[
\begin{aligned}
F_{\core} = F_{\spec} &= \{\mathrm T,\mathrm S,\mathrm D\}, \\
\Phi(\mathrm I)=\Phi(\mathrm{VI}) &= \{\mathrm T\}, \\
\Phi(\mathrm{III})=\Phi(\mathrm V) &= \{\mathrm D\}, \\
\Phi(\mathrm{II})=\Phi(\mathrm{IV}) &= \{\mathrm S\}, \\
\Phi(\mathrm{VII}) &= \varnothing.
\end{aligned}
\]
The pushforward $\Phi_{\hat f_{1,4}}$ transports each function to the (now non-obvious) chord its triad maps to:
\begin{align*}
\Phi_{\hat f_{1,4}}(\{\mathrm F,\mathrm A,\mathrm B,\mathrm C\}) &= \{\mathrm T\}, &
\Phi_{\hat f_{1,4}}(\{\mathrm D,\mathrm F,\mathrm G,\mathrm A\}) &= \{\mathrm S\}, \\
\Phi_{\hat f_{1,4}}(\{\mathrm B,\mathrm D,\mathrm E,\mathrm F\}) &= \{\mathrm D\}, &
\Phi_{\hat f_{1,4}}(\{\mathrm G,\mathrm B,\mathrm C,\mathrm D\}) &= \{\mathrm S\}, \\
\Phi_{\hat f_{1,4}}(\{\mathrm E,\mathrm G,\mathrm A,\mathrm B\}) &= \{\mathrm D\}, &
\Phi_{\hat f_{1,4}}(\{\mathrm C,\mathrm E,\mathrm F,\mathrm G\}) &= \{\mathrm T\}, \\
\Phi_{\hat f_{1,4}}(\{\mathrm A,\mathrm C,\mathrm D,\mathrm E\}) &= \varnothing. &&
\end{align*}
The functor transports functional identity correctly even though the underlying chord correspondence is far from obvious: which chord realizes which function is determined entirely by $\hat f_{1,4}$, not by any resemblance in pitch content.
\end{example}

\subsubsection{Harmonic Functions Dependent on Deformation Families}
\label{subsub:harmonic-functions-dependent-on-deformation-families}

There are two natural ways to extend the harmonic function fibration over $\int\DefFam$. The first is to regard the functional assignment as independent of the deformation family. In this case, an object is a tuple
\[
([X],Y^{(\mathbf c)},\mu_i,\mathcal D,(\mathcal{F},\Phi)),
\]
where the functional assignment does not depend on $\mathcal D$. The assignment $\Phi:Y^{(\mathbf c)}\to\mathcal P(F_{\spec})$ then extends canonically to the deformation chords by composition with the projection
\[
\mathcal{D}(Y^{(\mathbf c)}) \xlongrightarrow{p_{Y^{(\mathbf c)}}} Y^{(\mathbf c)} \xlongrightarrow{\Phi} \mathcal{P}(F_{\spec}),
\]
giving the composite
\[
\Phi\circ p_{Y^{(\mathbf c)}}:
\mathcal D(Y^{(\mathbf c)})
\longrightarrow
\mathcal P(F_{\spec}).
\]
Thus every deformation chord lying over a chord $A$ inherits the same harmonic functions as $A$ (this inheritance will reappear in \S~\ref{subsub:the-category-of-functional-harmonic-systems}, where we form a category of functional harmonic systems). Since the functional assignment is independent of the deformation family, the resulting functor is obtained as in the preceding subsection. Applying the Grothendieck construction therefore yields a fibration
\[
\int\Func\longrightarrow\int\DefFam.
\]

The second approach is to allow substitution chords to carry harmonic functions independently of the base chords from which they arise. In this case we define $\Func([X],Y^{(\mathbf c)},\mu_i,\mathcal D)$ to be the category whose objects are pairs $(\mathcal{F},\Phi)$, where $\mathcal{F}$ is as before, but where 
\[
\Phi:\mathcal D(Y^{(\mathbf c)})\longrightarrow\mathcal P(F_{\spec}) 
\]
is defined directly on the substitution chords. Morphisms are maps $g_{\core}:F_{\core} \to F'_{\core}$
such that 
\[
g_{\core}(\Phi_{\core}(A))\subseteq\Phi'_{\core}(A)
\] 
for every $A\in\mathcal D(Y^{(\mathbf c)})$.

\begin{remark}
\label{re:realizations-of-function-defs}
As in the previous subsection (Definition~\ref{def:realizations-of-function}), each core harmonic function $\alpha\in F_{\core}$ determines its realization set of the harmonic function $\alpha$. We write
\begin{equation}
\label{eq:realizations-of-function-defs}
\operatorname{Real}_\Phi(\alpha) \coloneqq \{ A \in \mathcal{D}(Y^{(\mathbf c)}) \mid \alpha \in \Phi_{\core}(A) \}. 
\end{equation}
Unlike the previous construction, the functional assignment is now defined directly on deformation chords. Consequently, distinct substitution chords lying over the same base chord may realize different harmonic functions.
\end{remark}

The natural question is whether this construction extends canonically to a functor
\[
\Func(\Delta,\bfT_p,\bfR_q,f,\iota):
\Func([X],Y^{(\mathbf c)},\mu_i,\mathcal D)
\longrightarrow
\Func([W],Z^{(\mathbf d)},\mu_j,\mathcal E).
\]
Unlike the previous construction, the answer is, in general, no. The reason lies in the injection that a morphism of $\int\DefFam$ lying over $\Delta$ carries \eqref{eq:def-on-objects},
\[
\begin{matrix}
\iota : & \mathcal D &  \rightarrowtail & \mathcal E, \\
& \Gamma & \mapsto & \Gamma\circ\Delta^{-1},
\end{matrix}
\] 
which extends to a map on the associated chord sets,
\begin{equation}
\label{eq:iota-defs}
\begin{matrix}
\iota : & \mathcal D(Y^{(\mathbf c)})& \rightarrowtail & \mathcal E(Z^{(\mathbf d)}) \\
& \Gamma(A) & \mapsto & (\Gamma\circ\Delta^{-1})\bigl(\hat f_{i,j}(A)\bigr). \end{matrix}
\end{equation}
This extension is injective but need not be surjective. Consequently there is no canonical extension $\Phi_\iota$ making 
\[\begin{tikzcd}
	{\mathcal{D}(Y^{(\mathbf c)})} && {\mathcal{E}(Z^{(\mathbf d)})} \\
	\\
	& {\mathcal{P}(F)}
	\arrow["\iota", tail, from=1-1, to=1-3]
	\arrow["\Phi"', from=1-1, to=3-2]
	\arrow["{\Phi_\iota}", from=1-3, to=3-2]
\end{tikzcd}\]
commute. In contrast to the previous subsection (see \eqref{eq:push-phi-forward}), there is no inverse to $\iota$ along which $\Phi$ may be transported. Thus this construction does not, by itself, define a functor over $\int\DefFam$. To recover functoriality, one must specify a canonical extension of $\Phi$ to the substitution chords lying outside the image of $\iota$. We now describe one such construction, drawing on the deformation-independent pushforward $\Phi_{\hat f_{i,j}}$ of \S~\ref{subsub:harmonic-functions-independent-of-deformation-families}.

Let $\iota:\mathcal D(Y^{(\mathbf c)})\rightarrowtail\mathcal E(Z^{(\mathbf d)})$ be as in \eqref{eq:iota-defs}. Write $\operatorname{im}(\iota)\subseteq\mathcal E(Z^{(\mathbf d)})$ for its image, and let
\[
\Phi_{\operatorname{im}(\iota)}
:
\operatorname{im}(\iota)
\longrightarrow
\mathcal P(F_{\spec})
\]
be the unique map satisfying $\Phi_{\operatorname{im}(\iota)}\circ\iota=\Phi$. Define the set of \emph{new deformation morphisms} by
\[
\operatorname{new}(\iota)
\coloneqq
\mathcal E(Z^{(\mathbf d)})
\setminus
\operatorname{im}(\iota).
\]
Thus every deformation morphism in $\mathcal E(Z^{(\mathbf d)})$ lies either in $\operatorname{im}(\iota)$ or in $\operatorname{new}(\iota)$. Restricting the projection $p_{Z^{(\mathbf d)}}:\mathcal E(Z^{(\mathbf d)}) \to Z^{(\mathbf d)}$ to the new deformation morphisms yields a map
\[
p_{\operatorname{new}}
:
\operatorname{new}(\iota)
\longrightarrow
Z^{(\mathbf d)}.
\]
We assign these new deformed harmonies the deformation-independent value already available from the earlier construction: for every $B\in Z^{(\mathbf d)}$, the pushforward 
\[
\Phi_{\hat f_{i,j}}(B) = \Phi\bigl(\hat f_{i,j}^{-1}(B)\bigr)
\] 
is defined on all of $Z^{(\mathbf d)}$, independently of $\iota$. We therefore define
\[
\operatorname{new}(\Phi) \coloneqq \Phi_{\hat f_{i,j}} \circ p_{\operatorname{new}}
:
\operatorname{new}(\iota)
\longrightarrow
\mathcal P(F_{\spec}).
\]
Since $\mathcal E(Z^{(\mathbf d)}) = \operatorname{im}(\iota) \sqcup \operatorname{new}(\iota)$, the maps $\Phi_{\operatorname{im}(\iota)}$ and $\operatorname{new}(\Phi)$ assemble uniquely into a map
\[
\Phi_{\operatorname{im}(\iota)}
\sqcup
\operatorname{new}(\Phi):
\mathcal E(Z^{(\mathbf d)})
\longrightarrow
\mathcal P(F_{\spec}).
\]
This yields a canonical extension of $\Phi$ along $\iota$: every deformation morphism already tracked by $\mathcal D$ (transported via $\Delta$) keeps its specific, deformation-dependent value, while every genuinely new deformation morphism in $\mathcal E$ inherits the generic value its underlying chord would carry regardless of deformation.

Consequently, $\Func$ extends canonically to a functor
\[
\Func:\int\DefFam\longrightarrow\Cat,
\]
and applying the Grothendieck construction yields a fibration
\[
P_\Func : \int\Func\longrightarrow\int\DefFam.
\]
Thus every object of $\int\Func$ consists of all the preceding layers of structure together with a functional assignment. Explicitly, the objects of $\int\Func$ are of the form
\[
([X],Y^{(\mathbf c)},\mu_i,\mathcal D,(\mathcal{F},\Phi)).
\]

Collecting the constructions of this section, we obtain the hierarchy of fibrations as present in Figure~\ref{fig:fibration-hierarchy}.

\section{Formalizing Harmonic Practices}
\label{sec:formalizing-harmonic-practices}

Section~\ref{sec:a-hierarchy-of-fibrations-over-heptatonic-scales} organizes the available harmonic data into a hierarchy of fibrations. We now use this data to define \emph{harmonic practices}, which generalize the practice of functional harmony in common-practice tonality. Arriving at this definition takes some preliminary work; we do not state it formally until \S~\ref{sub:harmonic-practices}. 

Unlike the preceding constructions, those introduced below are not obtained as total categories of fibrations. Instead, they are assembled from the fibrational data, but their formalization is driven primarily by musical considerations rather than by the categorical organization of harmonic data. This difference in approach is intentional and reflects a change in priority from mathematical coherence and systematicity to musical relevance. The organization of a harmonic practice cannot be derived solely from the internal logic of the mathematical framework, but must instead be informed by the musical conventions it is intended to model, which are often historically and musically contingent. The hierarchy of fibrations is not itself a theory of harmonic practice; rather, it serves as a repository of musical structure. It organizes the scales, modes, orbit covers, deformation families, and harmonic function data into a coherent mathematical framework, together with the canonical structure-preserving maps between them. A harmonic practice, by contrast, is obtained by selecting and interpreting particular portions of this structure according to musical principles.

Thus, one should not expect the conventions governing a harmonic practice to arise solely from the categorical structure afforded by the fibrations. Many musical conventions are inherently local. For example, the harmonic conventions governing the major mode do not automatically transport to the minor mode, despite the existence of canonical mode isomorphisms between them (which do canonically transport objects in the categories fibered above the modal categories). The mathematical transformations encoded by the fibrations preserve abstract structure, whereas the musical transformations relevant to a harmonic practice are determined by compositional considerations that need not be functorial.

\subsection{Modal Orbit Cover Systems}
\label{sub:modal-orbit-cover-systems}

The first component of a harmonic practice is the collection of scales, mode types, and orbit covers on which it is based. We formalize this data as a \emph{modal orbit cover system}. Common-practice tonality provides the motivating example.

First, common-practice tonality takes place within the translation orbit $[X]$ of the diatonic scale. Next, the modes corresponding to keys are the Ionian and Aeolian modes, namely $\mu_1$ and $\mu_6$. Thus we first consider the full subgroupoid
\[
[X]_I \longhookrightarrow \int\ActTR, \qquad I = \{1,6\},
\]
whose objects are the 24 major and minor modes, and whose morphisms are translation-rotation pairs.

Common-practice tonality is furthermore based on tertian trichords, which arise from the interval composition $\mathbf{c} = (2,2,3) \in \mathbf{C}(7,3)$. Thus we consider the full subgroupoid
\[
[X^{(\mathbf{c})}]_I \longhookrightarrow \int \OrbCovTR,
\]
whose objects are the modal orbit covers $(Y^{(\mathbf{c})},\mu_i)$ with $i\in I$. Its morphisms are again the translation-rotation pairs $(\bfT_p, \bfR_q)$. Hence $[X^{(\mathbf{c})}]_I$ describes the major and minor keys together with their associated tertian triads. This motivates the following definition.

\begin{definition}[Modal Orbit Cover System]
\label{def:modal-orbit-cover-system}
A \emph{modal orbit cover system} is a triple $([X],\mathbf{c},I)$, where:
\begin{enumerate}
\item $[X]$ is a translation orbit of heptatonic scales, 
\item $\mathbf{c}\in\mathbf{C}(7,k)$ is an interval composition with $k\le7$, and 
\item $I\subseteq\mathbb{Z}_7$ is a collection of mode types.
\end{enumerate} 
It determines the full subgroupoid
\[
[X^{(\mathbf{c})}]_I \longhookrightarrow \int \OrbCovTR,
\]
whose objects are the modal orbit covers $(Y^{(\mathbf{c})},\mu_i)$ with $i\in I$.
\end{definition}

\subsection{Harmonic Vocabularies}
\label{sub:harmonic-vocabularies}

A modal orbit cover system specifies the scales, modes, and orbit covers available in a harmonic practice. The next ingredient specifies how tones may be altered with the use of deformation families. 

Different mode types generally admit different deformation conventions. For example, in the minor mode it is common to raise the seventh scale degree to obtain a leading tone, whereas in the major mode the corresponding alteration (now the raised fifth degree) is generally not employed except in special contexts such as tonicization. Thus, while different translations of a fixed mode type naturally share the same deformation conventions, different mode types need not.

Observe that such an assignment does not arise by lifting the modal orbit cover system into the deformation-family fibration. Morphisms in $\DefFam([X])$ are inclusions of deformation families, so a lift would require $\mathcal{D}(i)=\mathcal{D}(j)$ for every pair of mode types $i,j\in I$. Musically, this is unreasonable, as deformation conventions are determined locally within the context of each mode rather than by canonical transport along mode isomorphisms. The familiar distinction between major and minor leading tones is a simple example of this phenomenon. Accordingly, we simply assign a deformation family to each mode type independently. This equips every modal orbit cover $(Y^{(\mathbf{c})},\mu_i)$ with deformation data, producing triples $(Y^{(\mathbf{c})},\mu_i,\mathcal{D}(i))$. These are objects of $\int\DefFam$, but we deliberately do not use the morphisms inherited from that category. Instead, the transformations between such triples are simply the underlying translation-rotation pairs between the modal orbit covers, with the deformation-family assignments regarded as fixed local data. Thus we have transformations
\[
(\mathbf{T}_p,\mathbf{R}_q):
(Y^{(\mathbf{c})},\mu_i,\mathcal{D}(i))
\longrightarrow
(\mathbf{T}_p(Y^{(\mathbf{c})}),\mu_{i+q},\mathcal{D}(i+q)),
\]
without requiring any compatibility between the deformation-family components $\mathcal{D}(i)$ and $\mathcal{D}(i+q)$.

This data determines the collection of chords available within the harmonic practice and therefore constitutes its harmonic vocabulary.

\begin{definition}[Harmonic Vocabulary]
A \emph{harmonic vocabulary} is a modal orbit cover system $([X],\mathbf{c},I)$ equipped with a deformation-family assignment
\[
\mathcal{D}:I\longrightarrow
\mathcal{P}(\Def(X)).
\]
We thus specify a harmonic vocabulary as a quadruple $\mathbf{H} = ([X],\mathbf{c},I,\mathcal{D})$.

For each modal orbit cover $(Y^{(\mathbf{c})},\mu_i)$, define its \emph{local harmonic vocabulary} by
\begin{equation}
\label{eq:local-vocabulary}
V_{Y,i}(\mathbf{H})
\coloneqq
\mathcal{D}(i)\bigl(Y^{(\mathbf{c})}\bigr) = \coprod_{\Delta \in \mathcal{D}(i)} \Delta(Y^{(\mathbf c)}).
\end{equation}
Thus $V_{Y,i}(\mathbf{H})$ is the set of harmonies available in the modal orbit cover $(Y^{(\mathbf{c})},\mu_i)$ once it is equipped with the deformation family $\mathcal{D}(i)$. The \emph{global harmonic vocabulary} is then
\begin{equation}
\label{eq:global-vocabulary}
V(\mathbf{H})
\coloneqq
\coprod_{(Y,i)\in[X]\times I}
V_{Y,i}(\mathbf{H}),
\end{equation}
the disjoint union of the local harmonic vocabularies.
\end{definition}

\begin{example}
Let $[X]$ be the translation orbit of the diatonic scale, let $\mathbf c=(2,2,3)$, and let $I=\{1,6\}$ index the Ionian and Aeolian modes. We illustrate both $\mathcal D(1)$ and $\mathcal D(6)$ on a single representative 
\[
Y=\{\mathrm{C},\mathrm{D},\mathrm{E},\mathrm{F},\mathrm{G},\mathrm{A},\mathrm{B}\}\in[X],
\]
read as $(Y,\mu_1)=$ C major and $(Y,\mu_6)=$ A minor.

For a deformation altering scale degrees $d_1,\ldots,d_k$ relative to mode $i$, we write $(\sharp d_1\cdots\sharp d_k)_i$ (or with $\flat$, as appropriate). The subscript records which mode's degree-counting convention is in force, since the same scale-degree numeral names different pitch classes in different modes. For example, in the C-major diatonic collection, $(\sharp4)_1$ and $(\sharp4)_6$ raise F and D, respectively.

\textbf{Mode 1 (Ionian).} Define
\[
\mathcal D(1) = \{ \mathrm{id}_X,\ (\sharp 4)_1,\ (\flat 7)_1 \}.
\]
Both nontrivial deformations carry the scale itself into an adjacent key:
\[
(\sharp4)_1(Y)=\{ \mathrm{C},\mathrm{D},\mathrm{E},\mathrm{F}^\sharp,\mathrm{G},\mathrm{A},\mathrm{B} \}
\] 
is the G major collection, and
\[
(\flat7)_1(Y)=\{\mathrm{C},\mathrm{D},\mathrm{E},\mathrm{F},\mathrm{G},\mathrm{A},\mathrm{B}^\flat\}
\] 
is the F major collection. Correspondingly, they supply pivot chords into those keys:
\begin{itemize}
\item $(\sharp4)_1$ sends $\mathrm{ii}=\{\mathrm{D},\mathrm{F},\mathrm{A}\}$ to $\{\mathrm{D},\mathrm{F}^\sharp,\mathrm{A}\}$, which is the dominant of G major;
\item $(\flat7)_1$ sends $\mathrm{iii}=\{\mathrm{E}, \mathrm{G}, \mathrm{B}\}$ to $\{\mathrm{E},\mathrm{G},\mathrm{B}^\flat\}$, which is the diminished $\mathrm{vii}^\circ$ of F major.
\end{itemize}

\textbf{Mode 6 (Aeolian).} Define
\[
\mathcal D(6) = \{ \mathrm{id}_X,\ (\sharp 7)_6,\ (\sharp 6\sharp 7)_6,\ (\sharp 3)_6,\
(\sharp 4\sharp 6)_6\}.
\]
The first two furnish the harmonic- and melodic-minor variants: $(\sharp7)_6$ sends $\mathrm v=\{\mathrm E, \mathrm G, \mathrm B\}$ to the major triad $\mathrm V = \{\mathrm E, \mathrm G^\sharp, \mathrm B\}$, while $(\sharp6\sharp7)_6$ additionally raises the sixth degree, carrying $Y$ to the ascending melodic-minor collection $\{\mathrm A, \mathrm B, \mathrm C, \mathrm D, \mathrm E, \mathrm F^\sharp, \mathrm G^\sharp\}$. The remaining two generate secondary dominants within the key:
\begin{itemize}
\item $(\sharp3)_6$ sends $\mathrm i=\{\mathrm A, \mathrm C, \mathrm E\}$ to $\{\mathrm A, \mathrm C^\sharp, \mathrm E\}$, the major triad, the dominant of $\mathrm{iv}$;
\item $(\sharp4\sharp6)_6$ sends $\mathrm{ii^\circ}=\{\mathrm B, \mathrm D, \mathrm F\}$ to $\{\mathrm B, \mathrm D^\sharp, \mathrm F^\sharp\}$, the dominant of $\mathrm v$.
\end{itemize}

Thus $([X],\mathbf c,I,\mathcal D)$ furnishes, for the major mode, pivot chords toward the dominant and subdominant keys, and, for the minor mode, the harmonic and melodic variants together with secondary dominants of the subdominant and dominant scale degrees.
\end{example}

In practice, deformation families are often specified indirectly through general musical principles rather than by listing deformations individually. For example, the principle ``introduce a leading tone'' applies uniformly to every mode of the diatonic scale, although the resulting deformation $\Delta \in \Def(X)$ depends on the mode under consideration. More generally, one may regard a musical principle as assigning, to each mode type, a collection of deformations that realize the desired musical effect. Such principles therefore naturally give rise to maps $I\to\mathcal{P}(\Def(X))$, and hence determine harmonic vocabularies.

\begin{example}
\label{ex:tonicize}
Consider the principle \emph{tonicize}---i.e., deform a chord so that it becomes a chromatic translate of the tonic triad, allowing it to serve as the tonic of another orbit cover within the same translation orbit. This foreshadows, in a simplified single-mode setting, the formalism in \S~\ref{para:the-scale-degree-level} (in particular, Definition~\ref{def:secondary-scale-degree-symbols}); here we work out the underlying idea concretely.

Let
\[
X = \{ \mathrm G^\flat, \mathrm A, \mathrm B^\flat, \mathrm C, \mathrm D^\flat,
\mathrm E^{\flat \flat}, \mathrm F \},
\]
let $[X]$ be its translation orbit, and let $\mathbf c=(3,3,1)$, forming the modal orbit cover $(X^{(\mathbf c)},\mu_4)$, whose tonic is $\mathrm G^\flat$, as represented in Figure~\ref{fig:tonicize-cover}.

The tonic triad is $\mathrm I=\{\mathrm G^\flat,\mathrm C,\mathrm F\}$. For each remaining chordal degree $R\in\{\mathrm{II},\ldots,\mathrm{VII}\}$, we seek the deformations $\Delta$ making $\Delta(R)$ a chromatic translate of $\mathrm I$---that is, the tonic of some transposed orbit cover $(\mathbf T_p(X)^{(\mathbf c)}, \mu_4)$. This defines the map
\[
\operatorname{Piv}: \{\mathrm{II},\ldots,\mathrm{VII}\}\longrightarrow\mathcal P(\Def(X)),
\]
\begin{equation}
\label{eq:piv}
\operatorname{Piv}(R)
\coloneqq
\bigl\{
\Delta\in\Def(X)
\bigm|
\exists\, p
\text{ such that }
\Delta(R) = \mathbf T_p(\mathrm I)
\bigr\}.
\end{equation}

For instance, consider the following two cases.
\begin{itemize}
\item Let $R=\mathrm{IV}=\{\mathrm C,\mathrm F,\mathrm B^\flat\}$, and let $\Delta=(\flat4)_4$
lower the fourth scale degree of mode $4$. Then
\[
\Delta(\mathrm{IV}) = \{\mathrm C^\flat,\mathrm F,\mathrm B^\flat\} = \mathbf T_5(\mathrm I),
\]
so $(\flat4)_4\in\operatorname{Piv}(\mathrm{IV})$: flatting the root of the IV chord turns it into a pivot toward the mode transposed up a perfect fourth.
\item Let $R=\mathrm{VI}=\{\mathrm E^{\flat\flat},\mathrm A,\mathrm D^\flat\}$, and let
$\Gamma=(\flat2)_4$ lower the second scale degree of mode $4$. Then
\[
\Gamma(\mathrm{VI}) = \{\mathrm E^{\flat\flat},\mathrm A^\flat,\mathrm D^\flat\}
= \mathbf T_8(\mathrm I),
\]
so $(\flat2)_4\in\operatorname{Piv}(\mathrm{VI})$.
\end{itemize}
The remaining values of $\operatorname{Piv}$ are computed similarly.
\end{example}

\begin{figure}
\centering
        \includegraphics[width=0.55\linewidth]{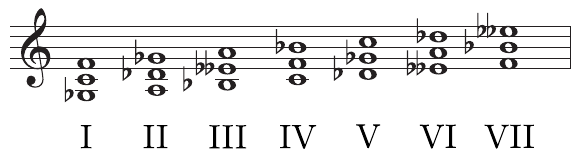}
\caption{The modal orbit cover $(X^{(\mathbf c)},\mu_4)$ of Example~\ref{ex:tonicize}, generated by the interval composition $\mathbf c=(3,3,1)$ over $X=\{\mathrm G^\flat,\mathrm A,\mathrm B^\flat,\mathrm C,\mathrm D^\flat, \mathrm E^{\flat\flat},\mathrm F\}$, with tonic $\mathrm G^\flat$.}
\label{fig:tonicize-cover}
\end{figure}

\begin{remark}
The deformation component of a harmonic vocabulary naturally admits refinement. For example, the historical development of common-practice tonality is accompanied by the gradual introduction of new chromatic alterations. Thus, the lowered second and sixth scale degree deformation is largely absent from earlier tonal practice, but later becomes commonplace through the use of the Neapolitan chord. Accordingly, the set of deformation-family assignments $I\to\mathcal P(\Def(X))$ carries the pointwise partial order
\[
\mathcal D\le\mathcal D'
\quad\Longleftrightarrow\quad
\mathcal D(i)\subseteq\mathcal D'(i)
\qquad\text{for all }i\in I.
\]
This induces a partial order on harmonic vocabularies:
\[
([X],\mathbf c,I,\mathcal D)
\le
([X],\mathbf c,I,\mathcal D').
\]
According to this ordering, larger harmonic vocabularies admit a greater collection of permissible deformations, and hence a richer collection of available harmonic substitutions.
\end{remark}

\subsection{Functional Harmonic Systems}
\label{sub:functional-harmonic-systems}

Next, we assign harmonic functions to the harmonic vocabularies. Once again, common-practice tonality serves as the motivating example. In its simplest form, it employs the three principal harmonic functions T, S, and D, corresponding to tonic, subdominant, and dominant.

As with deformation families, it is natural for harmonic-function assignments to be invariant under translation within a fixed mode type. Thus, every translation of a major dominant triad is again assigned dominant function, and similarly for the other harmonic functions.

Unlike translations, however, mode rotations need not preserve harmonic function, although in some harmonic systems it is reasonable for them to. For example, suppose the I chord is assigned tonic function, the IV chord subdominant function, and the V chord dominant function, and further suppose that every chord sharing two tones with one of these principal representatives is also assigned the corresponding function (as in the assignment $\Phi_{\core}^L$ of Example~\ref{ex:motte-rohrmeier}). Since chords may carry more than one harmonic function, these assignments need not be mutually exclusive. Because common-tone relationships are preserved under mode rotation, the resulting harmonic-function assignment is likewise preserved under rotation.

In common-practice tonality, however, harmonic functions are generally not preserved under mode rotation. For instance, the tonic counter-parallel function Tcp is realized by the iii chord in major but by the VI chord in minor. Thus the assignment of harmonic functions depends on the mode type rather than arising by canonical transport between modes. Consequently, harmonic functions should be assigned independently for each mode type.

To describe such assignments, let $F_{\core}$ denote a set of core harmonic functions and let $F_{\spec}$ denote a collection of specialized harmonic functions, as discussed in \S~\ref{sub:harmonic-function-fibrations}. Now recall the projection
\[
p_{Y^{(\mathbf c)}}:
\mathcal D(i)(Y^{(\mathbf c)})
\longrightarrow
Y^{(\mathbf c)}
\]
from~\eqref{eq:proj-orb-cov}. Via the canonical identification
\[
\begin{matrix}
\mu_i : & Y^{(\mathbf c)} & \xlongrightarrow{\cong} & \mathbb Z_7^{(\mathbf c)} \\
& A & \longmapsto & \{\mu_i(a)\mid a\in A\},
\end{matrix}
\]
this projection transports to the scale-degree representation as the projection
\[
p_i : \coprod_{\Delta\in\mathcal D(i)} \mathbb Z_7^{(\mathbf c)}
\longrightarrow
\mathbb Z_7^{(\mathbf c)},
\]
which restricts to the identity on each copy of $\Z_7^{(\mathbf{c})}$ indexed by a deformation $\Delta \in \mathcal{D}(i)$, making the following diagram commute.
\[\begin{tikzcd}
	{\mathcal{D}(i)(Y^{(\mathbf{c})})} && {Y^{(\mathbf{c})}} \\
	\\
	{\coprod\limits_{\Delta \in \mathcal{D}(i)} \Z_7^{(\mathbf{c})}} && {\Z_7^{(\mathbf{c})}}
	\arrow["{p_{Y^{(\mathbf{c})}}}", from=1-1, to=1-3]
	\arrow["{\mu_i}"', from=1-1, to=3-1]
	\arrow["{\mu_i}", from=1-3, to=3-3]
	\arrow["{p_i}"', from=3-1, to=3-3]
\end{tikzcd}\]
The functional data for the mode type $i$ is then specified by a map
\[
\Phi_i:
\coprod_{\Delta\in\mathcal D(i)}
\mathbb Z_7^{(\mathbf c)}
\longrightarrow
\mathcal P(F_{\spec}),
\]
assigning to each chord and chord substitution in the scale-degree representation a collection of specialized harmonic functions. Since the domain is expressed in terms of the abstract scale-degree representation $\mathbb Z_7^{(\mathbf c)}$ rather than concrete pitch classes, every translation of the same mode type inherits the same functional assignment.

Now let
\begin{equation}
\label{eq:def-degree-rep}
\mathcal{D}(i)(\Z_7^{(\mathbf c)}) \coloneqq \coprod_{\Delta \in \mathcal{D}(i)}\Z_7^{(\mathbf c)}.
\end{equation}
For each mode type $i\in I$, define
\begin{equation}
\label{eq:F(i)}
\mathbf F(i)
\coloneqq
(\mathcal P(F_{\spec}))^{\mathcal{D}(i)(\Z_7^{(\mathbf c)})}.
\end{equation}
Thus $\mathbf F(i)$ is the set of all possible functional assignments for the mode type $i$. A functional assignment for the entire harmonic vocabulary is therefore a family
\begin{equation}
\label{eq:functional-assignment-family}
\Phi\in\prod_{i\in I}\mathbf F(i),
\end{equation}
whose component at $i$ is a map
\[
\Phi_i:
\mathcal{D}(i)(\Z_7^{(\mathbf c)})
\longrightarrow
\mathcal P(F_{\spec}).
\]
Equivalently, $\Phi$ specifies, for each mode type, a function assigning a collection of specialized harmonic functions to every chord and chord substitution in an orbit cover. Again, since the domain is expressed in terms of the scale-degree representation $\mathbb Z_7^{(\mathbf c)}$ rather than concrete pitch classes, every translation of the same mode type inherits the same functional assignment.

\begin{definition}[Functional Harmonic System]
Let $\mathcal{F} = (F_{\core}, F_{\spec}, \varphi)$ be a harmonic function schema. A \emph{functional harmonic system} is a harmonic vocabulary  $\mathbf{H} = ([X], \mathbf{c}, I, \mathcal{D})$ equipped with a functional assignment 
\[
\Phi \in \prod_{i \in I} \mathbf{F}(i). 
\]
\end{definition}

\begin{example}[Secondary Functions]
\label{ex:secondary-functions}
We now revisit Example~\ref{ex:secondary-chords}, where the left-sided deformation pivots $\mathcal D(Y^{(\mathbf c)})\cap Z^{(\mathbf c)}$ recover the secondary chords of the target key.

Let
\[
C=(Y^{(\mathbf c)},\mu_1,\mathcal D)
\qquad\text{and}\qquad
G=(Z^{(\mathbf c)},\mu_1,\mathcal D)
\]
be the deformation-equipped modal orbit covers for C major and G major, where
$\mathcal D(1)$ consists of the single deformation that raises the fourth degree of the major mode. Define the harmonic function schema $\mathcal F=(F_{\core},F_{\spec},\varphi)$, where
\[
F_{\core}
=
\{\mathrm T,\mathrm S,\mathrm D\},
\qquad
F_{\spec}
=
\{\mathrm T,\mathrm{Tp},
\mathrm S,\mathrm{Sp},
\mathrm D,\mathrm{Dp}\},
\]
and $\varphi:F_{\spec}\to F_{\core}$ sends each principal function and its relative counterpart to the corresponding core function. Let
\[
\Phi_1:\Z_7^{(\mathbf c)}\longrightarrow\mathcal P(F_{\spec})
\]
be an ordinary common-practice functional assignment for the major mode. 

Then the dominant realization set
\[
\operatorname{Real}_{\Phi_1}(\mathrm D)
\subseteq
Z^{(\mathbf c)}
\]
consists of all chords in G major realizing the dominant function. Intersecting this realization set with the deformation pivots selects precisely those substitution chords that realize a dominant function in the target key. Equivalently, the secondary dominant chords are characterized by the pullback
\[\begin{tikzcd}
	{\mathrm{SecDom}} && {\mathcal{D}(Y^{(\mathbf c)}) \cap Z^{(\mathbf c)}} \\
	\\
	{\operatorname{Real}_{\Phi_1}(\mathrm{D})} && {Z^{(\mathbf c)}}
	\arrow[hook, from=1-1, to=1-3]
	\arrow[hook, from=1-1, to=3-1]
	\arrow[hook, from=1-3, to=3-3]
	\arrow[hook, from=3-1, to=3-3]
\end{tikzcd}\]
Thus secondary dominant chords are precisely those deformation pivots whose harmonic function in the target key is dominant (we will see a slightly different and more systematic method for deriving secondary chords in \S~\ref{para:the-scale-degree-level}).
\end{example}

\subsection{Harmonic Practices}
\label{sub:harmonic-practices}

The final component of a harmonic practice is a syntax, specifying which harmonic progressions are permitted among the harmonies provided by a given harmonic vocabulary. Since harmonic progressions are finite sequences of harmonies, the natural ambient object is the free monoid on the alphabet of available harmonies.

Let $\mathbf H = ([X],\mathbf c,I,\mathcal D,(\mathcal F,\Phi))$ be a functional harmonic system, and therefore $V(\mathbf{H})$ is the set of all harmonies in the harmonic vocabulary. The free monoid $(V(\mathbf{H}))^\ast$ then consists of all finite harmonic progressions. 

At the most general level, a harmonic syntax is simply a subset $\Sigma \subseteq (V(\mathbf{H}))^\ast$, whose elements are the admissible harmonic progressions. In practice, however, one wishes to describe such a syntax by means of a comparatively small collection of generative rules rather than by specifying an arbitrary subset of $(V(\mathbf{H}))^\ast$. For example, Rohrmeier~\cite{rohrmeier2011towards} showed that the syntax of common-practice tonality admits a presentation as a context-free grammar. We take his grammar as the starting point for a generalization to arbitrary functional harmonic systems. We therefore develop the general theory in parallel with Rohrmeier's presentation.  

\subsubsection{The Four Grammatical Levels}
\label{subsub:the-four-grammatical-levels}

Rohrmeier's grammar is organized into four principal levels:
\begin{enumerate}
\item the phrase level,
\item the function level,
\item the scale-degree level, and
\item the surface level.
\end{enumerate}
We now introduce each of these in turn, together with its corresponding generalization in our framework.

\paragraph{The Phrase Level}

In Rohrmeier's grammar, the phrase level employs phrase symbols
\[
\mathbb P=\{piece,\mathrm P\},
\]
together with key symbols
\[
\mathbb K
=
\{\mathrm{Cmaj},\mathrm{Cmin},
\mathrm{C\sharp maj},\mathrm{C\sharp min},
\mathrm{D\flat maj},\mathrm{D\flat min},\ldots\}.
\]

In our generalized setting, the phrase symbols remain unchanged. The key symbols, however, are replaced by symbols representing the modal orbit covers of the harmonic vocabulary. Hence, we redefine
\[
\mathbb K
\coloneqq
\{(Y,i)\mid Y\in[X],\ i\in I\},
\]
where $(Y,i)$ denotes the modal orbit cover $(Y^{(\mathbf c)},\mu_i)$.

In the special case where $[X]$ is the translation orbit of the diatonic scale, $\mathbf c=(2,2,3)$, and $I=\{1,6\}$, the resulting set $\mathbb K$ agrees with Rohrmeier's key symbols, modulo enharmonic equivalence.

\paragraph{The Function Level}
\label{para:the-function-level}

The function level of Rohrmeier employs functional-region symbols
\[
\mathbb R
=
\{\mathrm{T}_\mathrm{R},\mathrm{S}_\mathrm{R},\mathrm{D}_\mathrm{R}\},
\]
together with functional terms
\[
\mathbb F
=
\{\mathrm T,\mathrm S,\mathrm D,
\mathrm{Tp},\mathrm{Sp},\mathrm{Dp},\mathrm{Tcp}\}.
\]

This information is already present in a functional harmonic system, in the following way. For each core harmonic function $\alpha\in F_{\core}$ we introduce a corresponding functional-region symbol $\alpha_{\mathrm R}$, and define
\[
\mathbb R
\coloneqq
\{\alpha_{\mathrm R}\mid\alpha\in F_{\core}\}.
\]
Likewise, the functional terms are simply the specialized harmonic functions:
\[
\mathbb F
\coloneqq
F_{\spec}.
\]

\paragraph{The Scale-Degree Level}
\label{para:the-scale-degree-level}

The scale-degree level of Rohrmeier consists of symbols
\[
\mathbb{S}
=
\{\mathrm{I},\mathrm{II},\ldots,\mathrm{VII},
\ldots,
\mathrm{V}/\mathrm{II},
\ldots,
\mathrm{VII}/\mathrm{II}\},
\]
including both primary and secondary scale-degree symbols. Our goal now is to construct an analogous collection of abstract scale-degree symbols for an arbitrary functional harmonic system.

Via the canonical identification
\[
Y^{(\mathbf c)}
\xlongrightarrow{\cong}
\Z_7^{(\mathbf c)},
\]
the degree map (see \eqref{eq:chordal-degree})
\[
\deg_i:Y^{(\mathbf c)}
\longrightarrow
\Z_7
\]
for mode $i \in I$ transports to the abstract degree map
\[
\deg_i:\Z_7^{(\mathbf c)}
\longrightarrow
\Z_7.
\]
For readability, we identify the elements of $\Z_7$ with the Roman numeral symbols
\[
\mathbf S
=
\{\mathrm I,\mathrm{II},\ldots,\mathrm{VII}\},
\]
writing $\mathrm I$ in place of $0$, $\mathrm{II}$ in place of $1$, and so forth. Thus $\mathbf S$ consists of the primary (i.e.\ non-secondary) scale-degree symbols, without any additional harmonic information.

For every modal orbit cover $(Y^{(\mathbf c)},\mu_i)$, the corresponding chord of degree $R\in\mathbf S$ is obtained by the evaluation map
\[
\operatorname{ev}_{Y,i}
=
\mu_i^{-1}\circ\deg_i^{-1} :
\mathbf S
\longrightarrow
Y^{(\mathbf c)},
\]
so that $\operatorname{ev}_{Y,i}(R)$ is the chord of degree $R$ in the mode $(Y,\mu_i)$.

A bare Roman numeral, however, does not yet determine a generalized scale-degree symbol. Two additional pieces of information are required.

First, the harmonic quality of a Roman numeral depends on the mode. In common-practice tonality this distinction is encoded typographically by writing major triads with uppercase Roman numerals and minor triads with lowercase Roman numerals. This convention is no longer adequate in the present setting, since an orbit cover may contain more than two chord qualities. Accordingly, we regard the mode type as part of the scale-degree symbol itself. Thus, instead of writing $\mathrm{I}$ and $\mathrm{i}$ for the tonic triads in major and minor, respectively, we write $(1,\mathrm I)$ and $(6,\mathrm I)$.

Second, scale-degree symbols may carry chromatic alteration data via the deformations. For example, the Neapolitan chord in major is represented by a deformation lowering the second and sixth scale degrees. Hence, define
\begin{equation}
\label{eq:s-hat}
\hat{\mathbf S}
\coloneqq
\left(\sum_{i\in I}\mathcal D(i)\right)\times\mathbf S,
\end{equation}
as the product of the dependent sum\footnote{For a more thorough treatment of dependent sums, see \cite[\S2.5.5]{flieder2025typetheory}.} $\sum_{i\in I}\mathcal D(i)$ with $\mathbf S$, whose elements are triples $(i,\Delta,R)$, consisting of a mode type $i$, a deformation $\Delta \in \mathcal{D}(i)$, and a Roman numeral $R$. Such a triple is realized in a modal orbit cover via the evaluation map
\begin{equation}
\label{eq:ev}
\begin{matrix}
\operatorname{ev} : & [X]\times\hat{\mathbf S} & \longrightarrow & V(\mathbf H) \\
& (Y,(i,\Delta,R)) & \longmapsto & \Delta(\operatorname{ev}_{Y,i}(R)).
\end{matrix}
\end{equation}

Different deformations may nevertheless realize the same chord. Thus we declare the equivalence relation 
\[
(i,\Delta,R)\sim(i,\Gamma,R)
\]
whenever
\[
\Delta(\operatorname{ev}_{Y,i}(R))
=
\Gamma(\operatorname{ev}_{Y,i}(R))
\]
for some (equivalently, every) translation $Y\in[X]$. The resulting equivalence class records only the restriction of $\Delta$ to the scale degrees occurring in the chord $\operatorname{ev}_{Y,i}(R)$; two deformations that agree there are identified regardless of how they act on scale degrees the chord does not contain. For example, the harmonic-minor and melodic-minor deformations both realize the major V chord in minor: both raise the seventh scale degree, and V contains no other altered degree between them, so their restrictions to V's scale degrees coincide even though melodic minor additionally raises the sixth. These therefore determine the same generalized scale-degree symbol. The generalized primary scale-degree symbols are therefore the equivalence classes $\hat{\mathbf S}/\sim$. 

To obtain a symbolic alphabet, choose a bijection
\[
\hat{\mathbf S}/\sim \;
\xlongrightarrow{\cong}
\mathbb S_{\mathrm{pri}},
\]
where $\mathbb S_{\mathrm{pri}}$ denotes the set of primary scale-degree labels. For example, one may label the class of $(i,\Delta,R)$ by the symbol $R^\Delta_i$.

We now enlarge the alphabet $\mathbb S_{\mathrm{pri}}$ to incorporate secondary scale-degree symbols, thereby obtaining the full scale-degree alphabet corresponding to Rohrmeier's set~$\mathbb S$. Henceforth we identify an equivalence class of $\hat{\mathbf S}/\sim$ with its chosen label in $\mathbb S_{\mathrm{pri}}$, writing $R_i^\Delta$ for the class of $(i,\Delta,R)$.

Secondary scale-degree symbols arise from pivot chords, in the following way. Suppose
\[
(Y,R_i^\Delta),\,
(Z,S_j^\Gamma)
\in
[X]\times\mathbb S_{\mathrm{pri}}
\]
satisfy
\[
\operatorname{ev}(Y,R_i^\Delta)
=
\operatorname{ev}(Z,S_j^\Gamma).
\]
Then both local vocabularies $V_{Y,i}(\mathbf H)$ and $V_{Z,j}(\mathbf H)$ contain this same chord, so that it serves simultaneously as the degree $R_i^\Delta$ in $(Y,i)$ and the degree $S_j^\Gamma$ in $(Z,j)$. We are interested in the special case where the shared chord is the tonic degree of $(Z,j)$. Thus suppose
\[
\operatorname{ev}(Y,R_i^\Delta)
=
\operatorname{ev}(Z,\mathrm I_j).
\]
Then $R_i^\Delta$ also occurs as the tonic of the mode $(Z,j)$. Consequently, every primary scale-degree symbol of mode $j$ may be interpreted as a secondary scale-degree symbol of $R_i^\Delta$.

To formalize this, define
\begin{equation}
\label{eq:tonic}
\operatorname{Tonic} : \mathbb S_{\mathrm{pri}} \longrightarrow \mathcal P(I)
\end{equation}
by
\[
R_i^\Delta \longmapsto \left\{ j\in I \;\middle|\; \exists\,Y,Z\in[X] \text{ such that } \operatorname{ev}(Y,R_i^\Delta) = \operatorname{ev}(Z,\mathrm I_j) \right\}.
\]
Thus $\operatorname{Tonic}(R_i^\Delta)$ is the set of mode types for which $R_i^\Delta$ may serve as the tonic.  Now for each $j\in I$, write
\begin{equation}
\label{eq:pri,i}
\mathbb S_{\mathrm{pri},j}
\coloneqq
\{
S_j^\Gamma
\in
\mathbb S_{\mathrm{pri}}
\},
\end{equation}
the subset of primary scale-degree symbols belonging to mode $j$.

\begin{definition}[Secondary Scale-Degree Symbols]
\label{def:secondary-scale-degree-symbols}
For each primary scale-degree symbol $R_i^\Delta\in\mathbb S_{\mathrm{pri}}$, define
\[
\operatorname{Sec}(R_i^\Delta)
\coloneqq
\coprod_{j\in\operatorname{Tonic}(R_i^\Delta)}
\mathbb S_{\mathrm{pri},j}.
\]
The elements of $\operatorname{Sec}(R_i^\Delta)$ are called the \emph{secondary scale-degree symbols} of $R_i^\Delta$.
\end{definition}

The collection of all secondary scale-degree symbols is therefore the dependent sum
\begin{equation}
\label{eq:secondary-symbols}
\sum_{R_i^\Delta\in\mathbb S_{\mathrm{pri}}}
\operatorname{Sec}(R_i^\Delta),
\end{equation}
pairing each $R_i^\Delta$ with all of its secondary scale-degree symbols. Choosing a bijection
\[
\left(
\sum_{R_i^\Delta\in\mathbb S_{\mathrm{pri}}}
\operatorname{Sec}(R_i^\Delta)
\right)
\xlongrightarrow{\cong}
\mathbb S_{\mathrm{sec}}
\]
produces an alphabet of secondary scale-degree labels. For example, one may label the pair $(R_i^\Delta, S_j^\Gamma)$ by the symbol $S_j^\Gamma/R_i^\Delta$.

Finally, the complete scale-degree alphabet is the disjoint union
\[
\mathbb S
\coloneqq
\mathbb S_{\mathrm{pri}}
\sqcup
\mathbb S_{\mathrm{sec}},
\]
which generalizes Rohrmeier's scale-degree alphabet.

\begin{example}
Consider the class of diatonic scales $[X]$, with all of its modes $I = \Z_7$, and let $\mathbf c = (2,2,3)$. Write $I_{\mathrm{maj}} \coloneqq \{1,4,5\} \subset I$ for the mode types whose (undeformed) tonic triad is major. Now suppose we are in Ionian ($\mu_1$), and we have the deformation $\flat 2 \flat 6$ that flats the second and sixth scale degrees. Then the Roman numeral $\mathrm{II}_1^{\flat 2 \flat 6}$ specifies the Neapolitan chord in Ionian---concretely, in C major this is the $\mathrm{D}\flat$-F-$\mathrm{A}\flat$ triad. Since this chord is major, and $I_{\mathrm{maj}}$ is exactly the set of mode types whose tonic triad is major, it follows that
\[
\operatorname{Tonic}(\mathrm{II}_1^{\flat 2 \flat 6}) = I_{\mathrm{maj}}.
\]
Thus $\operatorname{Sec}(\mathrm{II}_1^{\flat 2 \flat 6})$ consists of the full primary alphabets $\mathbb S_{\mathrm{pri},j}$ for $j \in I_{\mathrm{maj}} = \{1,4,5\}$: every primary scale-degree symbol of Ionian, Lydian, and Mixolydian is thereby reinterpreted as a secondary scale-degree symbol of the Neapolitan chord.
\end{example}

\begin{remark}
In certain harmonic systems one may be interested only in certain classes of secondary chords. For example, in common-practice tonality the most familiar secondary chords are the secondary dominants, although more general secondary chords arise naturally in prolonged tonicizations and local tonal regions. The harmonic-function assignment allows such restrictions to be made uniformly. The map
\[
\Phi_i:\mathbb S_{\mathrm{pri},i}\to\mathcal P(F_{\spec})
\]
induces
\[
\Phi_{i,\core}:
\mathbb S_{\mathrm{pri},i}
\longrightarrow
\mathcal P(F_{\core}).
\]
For each core function $\alpha\in F_{\core}$, define
\[
\mathbb S_{\mathrm{pri},i}^{\alpha}
\coloneqq
\{
R_i^\Delta
\in
\mathbb S_{\mathrm{pri},i}
\mid
\alpha\in\Phi_{i,\core}(R_i^\Delta)
\}.
\]
Accordingly, one may restrict Definition~\ref{def:secondary-scale-degree-symbols} by defining
\[
\operatorname{Sec}_{\alpha}(R_i^\Delta)
\coloneqq
\coprod_{j\in\operatorname{Tonic}(R_i^\Delta)}
\mathbb S_{\mathrm{pri},j}^{\alpha},
\]
whose elements are the secondary $\alpha$-realizing scale-degree symbols of $R_i^\Delta$.
\end{remark}

\paragraph{The Surface Level}

Finally, Rohrmeier's surface level consists of chord symbols such as
\[
\mathbb O
=
\{\mathrm{Cmaj},\mathrm{Cmin},
\mathrm C^\circ,\ldots\}.
\]

In our setting, we instead assign a unique surface symbol to every harmony in the global harmonic vocabulary. Equivalently, we choose a bijection
\begin{equation}
\label{eq:lambda}
\lambda : V(\mathbf{H}) \longrightarrow\mathbb O,
\end{equation}
where $\mathbb O$ denotes the set of surface chord symbols.

\subsubsection{Rewrite Rules}
\label{subsub:rewrite-rules}

Having introduced these symbol sets, we now define a grammar over them. A formal grammar consists of a set of non-terminal symbols, a set of terminal symbols, a distinguished start symbol, and a collection of rewrite rules. The preceding subsections have supplied the raw material for the first three ingredients.

To ensure that rewrite rules are interpreted within a definite harmonic context, every non-terminal symbol, except for the distinguished start symbol, $piece$, is indexed by an ambient modal orbit cover, that is, by an element of $\mathbb{K}$. This indexing becomes indispensable once modulation is introduced in \S~\ref{para:scale-degree-level-rewrite-rules}.

Hence we define the set of non-terminal symbols by
\[
N
\coloneqq
\{piece\}
\;\cup\;
\bigl((\mathbb{P}\setminus\{piece\})
      \cup
      \mathbb{R}
      \cup
      \mathbb{F}
      \cup
      \mathbb{S}\bigr)
\times
\mathbb{K},
\]
where we write an element $(X,K)$ of the product as $X_K$, with
\[
X\in
(\mathbb{P}\setminus\{piece\})
\cup
\mathbb{R}
\cup
\mathbb{F}
\cup
\mathbb{S},
\qquad
K\in\mathbb{K}.
\]
The distinguished start symbol is
\[
piece\in\mathbb{P}\subseteq N.
\]
The terminal symbols are the surface chords $\mathbb{O}$.

To complete the specification of the grammar, it remains to define the set of rewrite rules. A \emph{rewrite rule} is an expression of the form
\[
A \longrightarrow X,
\]
where $A\in N$ is a non-terminal symbol and $X$ is a finite sequence of symbols from
\[
(N\setminus\{piece\})\cup\mathbb{O}.
\]

\paragraph{Phrase-Level Rewrite Rules}
\label{para:phrase-level-rewrite-rules}

A derivation begins with the distinguished start symbol $piece$. The first rewrite rule is
\begin{equation}
piece
\longrightarrow
\mathrm{P}_{K_1}\cdots \mathrm{P}_{K_n},
\qquad
n\ge 1,
\end{equation}
indicating that a piece consists of a finite sequence of phrases, each indexed by a (not necessarily distinct) modal orbit cover. Since each phrase symbol $\mathrm{P}_{K_i}$ is indexed by a modal orbit cover $K_i$, we assume, until the introduction of modulation (\S~\ref{para:scale-degree-level-rewrite-rules}), that every subsequent rewrite rule is interpreted relative to the same modal orbit cover $K_i$. Accordingly, we suppress the subscript $K_i$ throughout the remainder of this section.

At the phrase level, Rohrmeier's formalization of common-practice syntax \cite{rohrmeier2011towards} uses the rewrite rule
\[
\mathrm{P} \longrightarrow \mathrm{T}_\mathrm{R},
\]
indicating that a phrase is rewritten as a tonic region. This is appropriate for common-practice tonality, where every phrase is grounded in a distinguished tonic region. For full generality, however, a modal orbit cover need not admit a privileged functional center analogous to the tonic region. We therefore do not require that every phrase expands to a tonic region. Instead, we specify a subset $\mathcal{R}\subseteq\mathbb{R}$ of admissible functional regions and include the rewrite rules
\begin{equation}
\label{eq:phrase-rewrite}
\mathrm{P} \longrightarrow \alpha_{\mathrm{R}},
\qquad
\alpha_{\mathrm{R}}\in\mathcal{R}.
\end{equation}
We also allow phrases to subdivide recursively by including the family of rewrite rules
\[
\mathrm{P}
\longrightarrow
\underbrace{\mathrm{P}\cdots\mathrm{P}}_{n},
\qquad
n\ge1.
\]
This rule becomes important once modulation is introduced. Since modulation is triggered at the scale-degree level when a scale-degree symbol is reinterpreted as a pivot chord in a new modal orbit cover, the resulting modulation may naturally initiate one or more new phrases.

\paragraph{Functional Level Rewrite Rules}
\label{para:functional-level-rewrite-rules}

At the functional level we have functional expansion rules—how to expand a functional region—and substitution rules—how to substitute specialized functions for core harmonic functions. 

\subparagraph{Functional Expansion Rules}

We start by recalling Rohrmeier's functional expansion rules, which are the following:
\begin{align*}
  \mathrm{T}_\mathrm{R} &\longrightarrow  \mathrm{D}_\mathrm{R}\;\mathrm{T},\\
  \mathrm{D}_\mathrm{R} &\longrightarrow  \mathrm{S}_\mathrm{R}\;\mathrm{D},\\
  \mathrm{T}_\mathrm{R} &\longrightarrow  \mathrm{T}_\mathrm{R}\;\mathrm{D}_\mathrm{R},\\
  \mathrm{X}_\mathrm{R} &\longrightarrow  \mathrm{X}_\mathrm{R}\;\mathrm{X}_\mathrm{R}
    \quad\text{for any }\mathrm{X}_\mathrm{R}\in\mathbb{R},\\
  \mathrm{T}_\mathrm{R} &\longrightarrow  \mathrm{T},\\
  \mathrm{D}_\mathrm{R} &\longrightarrow  \mathrm{D},\\
  \mathrm{S}_\mathrm{R} &\longrightarrow  \mathrm{S}.
\end{align*}
We now give a slight generalization of these rules. The guiding principle is that every functional expansion should preserve the identity of the parent functional region. Thus, if a functional region $\alpha_{\mathrm R}$ is rewritten as a sequence of symbols, then either the region $\alpha_{\mathrm R}$ itself or its underlying harmonic function $\alpha$ should appear on the right-hand side.\footnote{This assumes that $F_{\core}\subseteq F_{\spec}$, as discussed in \S~\ref{subsub:harmonic-functions-independent-of-deformation-families}.}

First, every functional region may be rewritten as its underlying harmonic function,
\[
\alpha_{\mathrm R}\longrightarrow\alpha,
\qquad
\alpha_{\mathrm R}\in\mathbb R.
\]

The remaining rewrite rules have right-hand side consisting of two symbols. If both symbols are functional regions, then one of them must be the parent region $\alpha_{\mathrm R}$, so that the parent region is preserved under expansion. Likewise, if one symbol is a functional region and the other is a core harmonic function, then either the functional region must be $\alpha_{\mathrm R}$ or the core harmonic function must be $\alpha$. Finally, as in Rohrmeier's fourth rule above, every functional region may be duplicated. Accordingly, define 
\begin{equation}
\label{eq:expopt}
\mathrm{ExpOpt}(\alpha_\mathrm{R}) \coloneqq
(\{\alpha_{\mathrm R}\}\times\mathbb R)
\ \sqcup \
(\mathbb R\times\{\alpha_{\mathrm R}\})
\ \sqcup \
(\{\alpha_{\mathrm R}\}\times F_{\core})
\ \sqcup \
( F_{\core}\times\{\alpha\})
\ \sqcup \
(\{\alpha_{\mathrm R}\}\times\{\alpha_{\mathrm R}\}).
\end{equation}
Thus $\mathrm{ExpOpt}(\alpha_{\mathrm R})$ consists precisely of the admissible binary right-hand sides for rewrite rules with left-hand side $\alpha_{\mathrm R}$.

Now define
\begin{equation}
\label{eq:exprules}
\mathrm{ExpRules}(\alpha_{\mathrm R})
\coloneqq
\mathcal P(\mathrm{ExpOpt}(\alpha_{\mathrm R}))
\times
\{\alpha\}.
\end{equation}
Thus an element of $\mathrm{ExpRules}(\alpha_{\mathrm R})$ consists of a choice of admissible binary expansion rules for $\alpha_{\mathrm R}$ together with the unary rule
\[
\alpha_{\mathrm R}\longrightarrow\alpha.
\]

Finally, define
\begin{equation}
\label{eq:funcexp}
\mathrm{FuncExp}
\coloneqq
\prod_{\alpha_{\mathrm R}\in\mathbb R}
\mathrm{ExpRules}(\alpha_{\mathrm R}).
\end{equation}
An element of $\mathrm{FuncExp}$ therefore specifies a complete system of functional expansion rules by choosing, for every functional region, which admissible binary expansions are permitted.

\subparagraph{Substitution Rules}

The substitution rules specify how a core harmonic function may be replaced by one of its specialized realizations. Recall that the harmonic function schema $\mathcal{F} = (F_{\core}, F_{\spec}, \varphi)$ associates to each core function $\alpha\in F_{\core}$ the fiber
\[
\varphi^{-1}(\alpha)\subseteq F_{\spec}=\mathbb{F}
\]
of its specialized functions. Accordingly, for every $\alpha\in F_{\core}$ and every $\alpha' \in\varphi^{-1}(\alpha)$, we include the rewrite rule
\begin{equation}
\alpha \longrightarrow \alpha'.
\end{equation}
(For instance, with the schema of Example~\ref{ex:harmonic-function-schema}, this yields rewrite rules $\mathrm{T} \longrightarrow \mathrm{Tp}$, $\mathrm{T} \longrightarrow \mathrm{Tcp}$, $\mathrm{S} \longrightarrow \mathrm{Sp}$, and $\mathrm{D} \longrightarrow \mathrm{Dp}$.)

\paragraph{Scale-Degree-Level Rewrite Rules}
\label{para:scale-degree-level-rewrite-rules}

We now come to three kinds of rewrite rules: those for converting functional terms to scale-degree symbols, those for generating secondary chords of a scale degree, and finally those governing modulation via pivot chords. This is the most intricate part of the grammar, since it is here that the harmonic structure encoded by modal orbit covers is involved most directly.

\subparagraph{Function-to-Scale-Degree Rules}

We first formulate the rewrite rules that rewrite functional terms as scale-degree symbols. Recall that every non-terminal symbol, aside from the distinguished start symbol $piece$, is indexed by a modal orbit cover $K=(Y,i)\in\mathbb K$. These rewrite rules depend on the mode $i$.

Recall the set $\mathbb{S}_{\mathrm{pri},i}$ from \eqref{eq:pri,i}, consisting of the primary scale-degree symbols belonging to mode $i$. We continue to denote by
\[
\Phi_i:\mathbb{S}_{\mathrm{pri},i}\longrightarrow\mathcal{P}(\mathbb{F})
\]
the map assigning to each primary scale-degree symbol the set of harmonic functions it may realize.

Given a harmonic function $\alpha\in\mathbb F$, our goal is to identify every primary scale-degree symbol whose associated set of harmonic functions contains $\alpha$. To this end, define, as in \eqref{eq:up-alpha},
\[
\uparrow\alpha
\coloneqq
\{A\subseteq\mathbb F\mid\alpha\in A\}.
\]
Then
\[
\Phi_i^{-1}(\uparrow\alpha)
=
\{R\in\mathbb{S}_{\mathrm{pri},i}\mid\alpha\in\Phi_i(R)\}.
\]
Accordingly, for every $\alpha\in\mathbb F$ and every $R\in\Phi_i^{-1}(\uparrow\alpha)$, we include the rewrite rule
\begin{equation}
\alpha\longrightarrow R.
\end{equation}

\subparagraph{Secondary-Function Rules}

We now give rewrite rules for secondary functions. Recall from Definition~\ref{def:secondary-scale-degree-symbols} that every primary scale-degree symbol $R\in\mathbb{S}_{\mathrm{pri}}$ has an associated set $\operatorname{Sec}(R)$ of secondary scale-degree symbols. Hence, for every $S\in\operatorname{Sec}(R)$, we include the rewrite rule
\begin{equation}
R\longrightarrow S \; R.
\end{equation}
Thus a primary scale degree may be expanded by prefixing one of its admissible secondary functions.

\begin{remark}[Sequences]
Rohrmeier's grammar also includes a rewrite rule generating descending fifth sequences. Namely, for each scale-degree symbol $R$, one has the rewrite rule
\[
R\longrightarrow \operatorname{Desc}_{5}(R) \; R,
\]
where $\operatorname{Desc}_{5}(R)$ denotes the scale-degree symbol a diatonic fifth below $R$. Repeated application of this rule produces a descending-fifths sequence.

More generally, one could introduce rewrite rules
\[
R\longrightarrow\operatorname{Desc}_{x}(R),
\]
where $x$ is an arbitrary scalar interval. Recursive application then generates scalar sequences of the corresponding scalar interval.
\end{remark}

\subparagraph{Modulation Rules}

In our framework, modulation is realized through pivot chords. That is, modulation occurs by reinterpreting a scale-degree symbol $R\in\mathbb S_{\mathrm{pri}}$ as the tonic of a new mode. This is possible precisely when $\operatorname{Tonic}(R)\neq\varnothing$, where $\operatorname{Tonic}$ was defined in \eqref{eq:tonic}.

We must now make use of the modal orbit cover indexing inherited by scale-degree symbols from the phrase level. Thus we write $R_{(Y,i)}$ for the scale-degree symbol $R$ belonging to the modal orbit cover $(Y,i)$.

Let $j\in\operatorname{Tonic}(R)$. By the definition of $\operatorname{Tonic}$ \eqref{eq:tonic}, there is a unique translate $Z\in[X]$ such that
\[
\operatorname{ev}(R_{(Y,i)})
=
\operatorname{ev}(\mathrm I_{(Z,j)}),
\]
where $\operatorname{ev}$ denotes the evaluation map of \eqref{eq:ev}. Hence we include the rewrite rule
\begin{equation}
R_{(Y,i)}
\longrightarrow
\mathrm P_{(Z,j)}.
\end{equation}

Thus a scale-degree symbol in one modal orbit cover is reinterpreted as the tonic of another, thereby initiating a new phrase in the modal orbit cover $(Z,j)$. Modulation therefore occurs as a rewrite from a scale-degree symbol to a phrase symbol, with the pivot chord serving as the tonic of the new phrase.

\paragraph{Surface-Level Rewrite Rules}

Finally, the remaining scale-degree symbols are rewritten as surface chords by means of the evaluation map. Thus, for every indexed scale-degree symbol $R_{(Y,i)}$, we include the rewrite rule
\[
R_{(Y,i)}
\longrightarrow
\operatorname{ev}(R_{(Y,i)}),
\]
where $\operatorname{ev}(R_{(Y,i)})\in\mathbb O$ is the corresponding surface chord.

Since the symbols of $\mathbb O$ are terminal symbols, no further harmonic rewrites are possible. 

\begin{figure}[h!]
    \centering
    \includegraphics[width=0.8\linewidth]{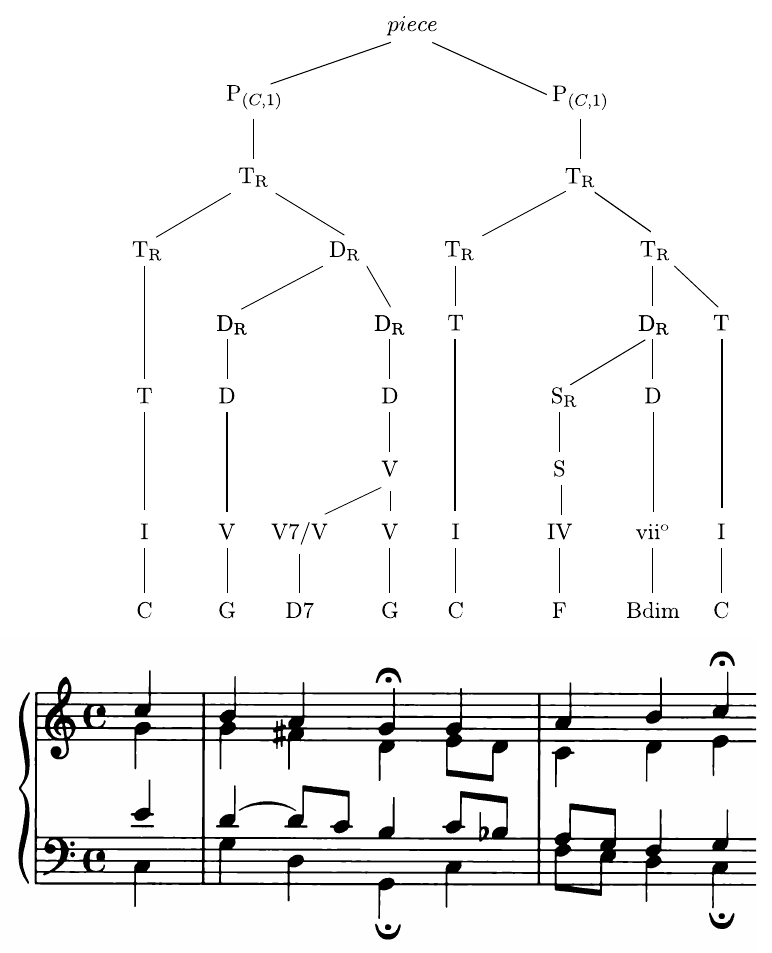}
    \caption{A derivation tree, in Rohrmeier's syntax, for the opening two phrases of Bach's four-part chorale \emph{Ach Gott und Herr}, BWV 255.}
    \label{fig:Bach-derivation}
\end{figure}

\begin{example}
Figure~\ref{fig:Bach-derivation} shows a derivation tree, in Rohrmeier's syntax, for the opening two phrases of Bach's four-part chorale \emph{Ach Gott und Herr}, BWV 255 (the tree covers only this excerpt, not the full chorale). Starting from the start symbol $piece$, the derivation branches into two phrases, each a diatonic instance $\mathrm{P}_{(C,1)}$ of the grammar developed above. 
\end{example}

\subsubsection{Formal Definition of Harmonic Practice}
\label{subsub:formal-definition-of-harmonic-practice}

Given the data of a functional harmonic system $\mathbf{H} = ([X], \mathbf c, I, \mathcal{D}, (\mathcal{F},\Phi))$, recall that the symbol sets $\mathbb K,\mathbb R,\mathbb F,\mathbb S,\mathbb P,\mathbb O$ were each canonically determined by $\mathbf H$ in \S\ref{subsub:the-four-grammatical-levels} (up to a labeling choice for $\mathbb S_{\mathrm{pri}}$, $\mathbb S_{\mathrm{sec}}$, and $\mathbb O$). Consequently, the preceding construction is almost entirely canonical. Indeed, the only genuine choices made in specifying the grammar are:
\begin{enumerate}
\item the choice of a subset $\mathcal R\subseteq\mathbb R$ of admissible phrase-initial functional regions for the rewrite rule \eqref{eq:phrase-rewrite} (see \S~\ref{para:phrase-level-rewrite-rules}); and

\item the choice of a system of functional expansion rules, that is, an element $\mathcal E\in\mathrm{FuncExp}$, where $\mathrm{FuncExp}$ is defined in \eqref{eq:funcexp} (see \S~\ref{para:functional-level-rewrite-rules}).
\end{enumerate}
All remaining rewrite rules are determined canonically by these choices together with the underlying functional harmonic system $\mathbf H$. 

This motivates the following definition.

\begin{definition}[Syntactic Specification]
\label{def:syntactic-specification}
Let $\mathbf H=([X],\mathbf c,I,\mathcal D,(\mathcal F,\Phi))$ be a functional harmonic system. A \emph{syntactic specification} for
$\mathbf H$ is a pair
\[
\Sigma=( R, E)
\in
\mathcal P(\mathbb R)\times\mathrm{FuncExp}.
\]
\end{definition}

The rewrite rules described above are therefore induced canonically from the pair $(\mathbf H,\Sigma)$.

\begin{definition}[Harmonic Practice]
A \emph{harmonic practice} is a pair $(\mathbf H,\Sigma)$, consisting of a functional harmonic system together with a syntactic specification for it.
\end{definition}

The preceding constructions therefore determine, for every harmonic practice $(\mathbf H,\Sigma)$, a unique context-free grammar $\mathcal G(\mathbf H,\Sigma)$. We denote by $L(\mathbf H,\Sigma)$ the language generated by this grammar, that is, the set of well-formed terminal strings derivable from the distinguished start symbol. Since the terminal symbols are the surface chords $\mathbb O$, we have $L(\mathbf H,\Sigma)\subseteq\mathbb O^*$, where $\mathbb O^*$ denotes the free monoid on $\mathbb O$.

\begin{example}
\label{ex:piano-vars}
Recall from Example~\ref{ex:tonicize} the scale
\[
X = \{ \mathrm G^\flat, \mathrm A, \mathrm B^\flat, \mathrm C, \mathrm D^\flat,
\mathrm E^{\flat \flat}, \mathrm F \}
\]
and the interval composition $\mathbf q=(3,3,1)$, forming the modal orbit cover $(X^{(\mathbf q)},\mu_4)$ whose tonic is $\mathrm G^\flat$. We translate the diatonic tertian cover into this one via an isomorphism, inherit its functional assignment, and use the result to build a derivation tree that closely parallels a common-practice derivation in the sense of Rohrmeier's syntax, but with the harmonic flavor of $X^{(\mathbf q)}$.

Let $\mathbf{t} = (2,2,3)$ be the tertian interval composition, so $\hat{\mathbf t}(0)=\{0,2,4\}$ and $\hat{\mathbf q}(0)=\{0,3,6\}$. Take $u=5$: since
\[
5\cdot\{0,2,4\} = \{0,10,20\} \pmod 7 = \{0,3,6\} = \hat{\mathbf q}(0)
\]
exactly, the map $f(x)=5x$ (i.e.\ $b=0$) lies in $\mathrm{Hom}_{\mathbf C(7)}(\mathbf t,\mathbf q)$, and in fact is an isomorphism, since $f(\hat{\mathbf t}(0))=\hat{\mathbf q}(0)$ exactly---consistent with $\mathbf t$ and $\mathbf q$ lying in the same isomorphism class of Figure~\ref{fig:isoclasses}. This witness will supply the isomorphism carrying the tertian cover to $X^{(\mathbf q)}$.

To furnish the functional assignment to be transported, take $C= \{ \mathrm{C}, \mathrm{D}, \ldots, \mathrm{B}\}$, the C major scale, and define $\Phi$, with specialized functions $F_{\spec}=\{\mathrm{T}, \mathrm{Tp}, \mathrm{S}, \mathrm{Sp}, \mathrm{D}, \mathrm{Dp} \}$, as follows:
\[
\begin{tikzcd}[column sep=large, row sep=small]
C^{(\mathbf t)} \arrow[r, "\Phi"]
  & \mathcal P(F_{\spec}) \\[6pt]
\mathrm I=\{\mathrm C,\mathrm E,\mathrm G\} \arrow[r, |->]
  & \{\mathrm T\} \\
\mathrm{II}=\{\mathrm D,\mathrm F,\mathrm A\} \arrow[r, |->]
  & \{\mathrm{Sp}\} \\
\mathrm{III}=\{\mathrm E,\mathrm G,\mathrm B\} \arrow[r, |->]
  & \{\mathrm{Dp}\} \\
\mathrm{IV}=\{\mathrm F,\mathrm A,\mathrm C\} \arrow[r, |->]
  & \{\mathrm S\} \\
\mathrm V=\{\mathrm G,\mathrm B,\mathrm D\} \arrow[r, |->]
  & \{\mathrm D\} \\
\mathrm{VI}=\{\mathrm A,\mathrm C,\mathrm E\} \arrow[r, |->]
  & \{\mathrm{Tp}\} \\
\mathrm{VII}=\{\mathrm B,\mathrm D,\mathrm F\} \arrow[r, |->]
  & \{\mathrm D\}
\end{tikzcd}
\]
The witness $f : \mathbf t \to \mathbf q$ determines the isomorphism
\[
\hat{f}_{1,4} : C^{(\mathbf t)} \longrightarrow X^{(\mathbf q)},
\]
along which we push $\Phi$ forward via \eqref{eq:push-phi-forward}, so that each chord of $C^{(\mathbf t)}$ carries its function across to the corresponding chord of $X^{(\mathbf q)}$:
\[
\begin{tikzcd}[column sep=large, row sep=small]
C^{(\mathbf t)} \arrow[r, "\hat f_{1,4}"]
  & X^{(\mathbf q)} \arrow[r, "\Phi_{\hat f_{1,4}}"]
  & \mathcal P(F_{\spec}) \\[6pt]
\mathrm I=\{\mathrm C,\mathrm E,\mathrm G\} \arrow[r, |->]
  & \{\mathrm G^\flat,\mathrm C,\mathrm F\}=\mathrm I \arrow[r, |->]
  & \{\mathrm T\} \\
\mathrm{II}=\{\mathrm D,\mathrm F,\mathrm A\} \arrow[r, |->]
  & \{\mathrm E^{\flat\flat},\mathrm A,\mathrm D^\flat\}=\mathrm{VI} \arrow[r, |->]
  & \{\mathrm{Sp}\} \\
\mathrm{III}=\{\mathrm E,\mathrm G,\mathrm B\} \arrow[r, |->]
  & \{\mathrm C,\mathrm F,\mathrm B^\flat\}=\mathrm{IV} \arrow[r, |->]
  & \{\mathrm{Dp}\} \\
\mathrm{IV}=\{\mathrm F,\mathrm A,\mathrm C\} \arrow[r, |->]
  & \{\mathrm A,\mathrm D^\flat,\mathrm G^\flat\}=\mathrm{II} \arrow[r, |->]
  & \{\mathrm S\} \\
\mathrm V=\{\mathrm G,\mathrm B,\mathrm D\} \arrow[r, |->]
  & \{\mathrm F,\mathrm B^\flat,\mathrm E^{\flat\flat}\}=\mathrm{VII} \arrow[r, |->]
  & \{\mathrm D\} \\
\mathrm{VI}=\{\mathrm A,\mathrm C,\mathrm E\} \arrow[r, |->]
  & \{\mathrm D^\flat,\mathrm G^\flat,\mathrm C\}=\mathrm V \arrow[r, |->]
  & \{\mathrm{Tp}\} \\
\mathrm{VII}=\{\mathrm B,\mathrm D,\mathrm F\} \arrow[r, |->]
  & \{\mathrm B^\flat,\mathrm E^{\flat\flat},\mathrm A\}=\mathrm{III} \arrow[r, |->]
  & \{\mathrm D\}
\end{tikzcd}
\]

To obtain a harmony capable of being tonicized, consider the deformation $\musDoubleSharp 6$, which raises the sixth scale degree of $X$ in mode $4$ by two semitones and so takes $\mathrm{E} \flat \flat$ to $\mathrm{E} \natural$. Applying it to $\mathrm{III}$ gives
\[
\mathrm{III}^{\musDoubleSharp 6} = \text{B}\flat\text{-E-A} = \bfT_4(\mathrm{I}),
\]
a transposition of the tonic by four semitones. Since $\mathrm{III}^{\musDoubleSharp 6}$ therefore functions locally as a tonic, it can itself support secondary chords: for instance, $\mathrm{III10}/\mathrm{III}^{\musDoubleSharp 6}$ functions as its secondary dominant (with an added tenth).

With this harmonic vocabulary in hand, we construct a derivation tree using the same syntactic specification as Rohrmeier's, shown in Figure~\ref{fig:piano-vars-derivation} together with a musical realization for piano (a slight adjustment, made for pedagogical purposes, of the first phrase of the theme of my \emph{Variations for Piano}, Op.~31). The example illustrates how much of the structure of common-practice tonality can be inherited even while its harmonic content is substantially altered.
\end{example}

\begin{sidewaysfigure}
    \includegraphics[width=1\linewidth]{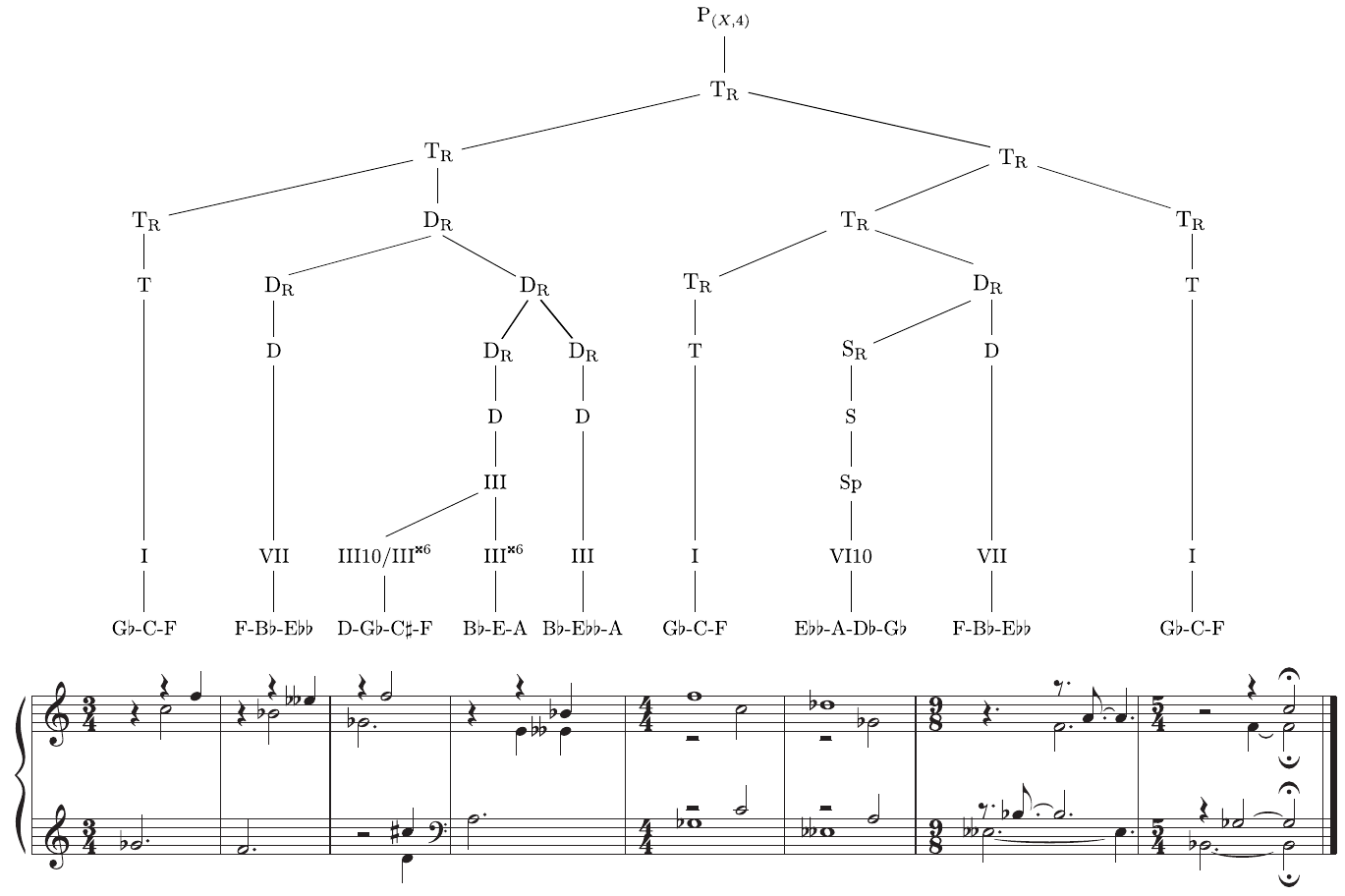}
    \caption{A derivation of $\mathrm P_{(X,4)}$ in the modal orbit cover system of Example~\ref{ex:piano-vars}, together with a musical realization for piano---a slight adjustment, made for pedagogical purposes, of the opening phrase of the theme of my \emph{Variations for Piano}, Op.~31. The tree parallels a common-practice derivation in the sense of Rohrmeier's syntax, but is built from the functional assignment $\Phi$ pushed forward onto $X^{(\mathbf q)}$ via $\hat f_{1,4}$ (see the accompanying text). Roman numerals label scale degrees of $(X^{(\mathbf q)},\mu_4)$; the chord $\mathrm{III}^{\musDoubleSharp 6}$ is the double-sharp-6 deformation of degree III, equal to $\bfT_4(\mathrm I)$, over which $\mathrm{III10}$ functions as a secondary dominant.}
    \label{fig:piano-vars-derivation}
\end{sidewaysfigure}

\subsection{Inheritance of Harmonic Practices}
\label{sub:inheritance-of-harmonic-practices}

In \S~\ref{sub:modal-orbit-cover-systems}–\ref{sub:harmonic-practices} we introduced several successive layers of harmonic structure: modal orbit cover systems, harmonic vocabularies, functional harmonic systems, and harmonic practices. It is natural to ask whether each of these forms a category, whose morphisms describe how harmonic structure may be inherited from one system to another.

The amount of structure that should be preserved depends on the layer under consideration. Modal orbit cover systems admit a natural notion of morphism that preserves much of their structure. Harmonic vocabularies, by contrast, specify deformation choices that are often highly dependent on local harmonic context, and therefore it is generally unreasonable to require deformation data to be preserved under morphisms, for the reasons discussed in \S~\ref{sub:harmonic-vocabularies}. Functional harmonic systems and harmonic practices allow for substantially more structure to be inherited.

We begin with modal orbit cover systems.

\subsubsection{The Category of Modal Orbit Cover Systems}

Recall from Definition~\ref{def:modal-orbit-cover-system} that a modal orbit cover system is a triple $([X],\mathbf c,I)$, consisting of a scale orbit, an interval composition determining the orbit-cover type, and a chosen collection of modes. Such a system determines the full subgroupoid
\[
[X^{(\mathbf c)}]_I
\longhookrightarrow
\int\OrbCovTR.
\]

A morphism between modal orbit cover systems should preserve not only the underlying orbit-cover structure, but also the relationships among the chosen modal realizations. These two requirements are governed by the same automorphism $h : \Z_7 \to \Z_7$ on scale degrees, as we now show.

Note that automorphisms $h:\mathbb Z_7\to\mathbb Z_7$ play two closely related roles. First, they transform interval compositions, thereby changing the structure of the generating chord of an orbit cover. For example, multiplication by $5$ sends the interval composition corresponding to tertian triads to that corresponding to quartal triads. Second, since the modes are canonically indexed by the elements of $\mathbb Z_7$, the same $h$ determines how the modal realizations themselves should correspond.

\begin{example}
Consider the tertian interval composition $\mathbf{c}=(2,2,3)$ and the quartal interval composition $\mathbf{d}=(3,3,1)$. Over the C major scale $X$, these determine the tertian orbit cover $X^{(\mathbf c)}$ and the quartal orbit cover $X^{(\mathbf d)}$, respectively.

Now consider the automorphism
\[
\begin{matrix}
 h : &\mathbb Z_7&\longrightarrow&\mathbb Z_7 \\
& x & \longmapsto &5x \pmod 7.
\end{matrix}
\]
Since
\[
h(1)=5\equiv -2 \pmod 7,
\]
this map replaces each ascending scale step by a descent of two scale steps. Consequently, $h(\{0,2,4\}) = \{0,3,6\}$, so the generating tertian chord is transformed into the generating quartal chord.

Because the modes are canonically indexed by the elements of $\mathbb Z_7$, the same $h$ also determines how the modal realizations correspond. For example,
\[
 h(1) \equiv 5 \pmod 7,
\qquad
 h(3) \equiv 1 \pmod 7,
\]
so the relationship between the Ionian mode $\mu_1$ and the Phrygian mode $\mu_3$ is carried to the relationship between the Mixolydian mode $\mu_5$ and the Ionian mode $\mu_1$. For instance, the tonic chord of $(X^{(\mathbf c)},\mu_1)$ is $\{0,4,7\}$, while that of $(X^{(\mathbf c)},\mu_3)$ is $\{4,7,11\}$, sharing its first two tones with the last two tones of the former. Likewise, the tonic chord of $(X^{(\mathbf d)},\mu_5)$ is $\{7,0,5\}$, while that of $(X^{(\mathbf d)},\mu_1)$ is $\{0,5,11\}$, which exhibits the same incidence relationship. Thus the automorphism $h$ simultaneously transforms the orbit-cover structure and the corresponding relationships among its modal realizations.
\end{example}

Previously we only considered morphisms between individual modes. Given modal orbit covers $(X^{(\mathbf c)},\mu_i)$ and $(Y^{(\mathbf d)},\mu_j)$ together with an automorphism $h$ on $\Z_7$, there is a unique map $h_{i,j} : X \to Y$ satisfying $\mu_j\circ h_{i,j} = h\circ\mu_i$, as in~\eqref{eq:h_ij} (which extends to an orbit-cover morphism when the condition of Definition~\ref{def:morphisms-Scaorb} is satisfied).

To define morphisms of modal orbit cover systems, however, we must relate entire collections of modes. Thus the automorphism $h$ must restrict to an injective map
\[
\hat{\mathcal M} : I \rightarrowtail J,
\]
so that the correspondence between modal realizations varies coherently with the transformation of the orbit-cover structure. Accordingly, a morphism of modal orbit cover systems is given by a quadruple
\[
\mathcal M
=
(\Delta,\mathbf T_p, \bfR_q, f):
[X^{(\mathbf c)}]_I
\longrightarrow
[Y^{(\mathbf d)}]_J,
\]
where the first two components describe the transformation of the underlying scales, while $\bfR_q$ and $f$ determine the induced correspondence between the orbit-cover structure and the chosen modal realizations; here $h$ denotes the automorphism induced by the linear part of $f$, i.e.\ $h(z) \coloneqq \ell(f)\cdot z$, so that $h_{i,j}$ above is understood relative to this $f$. We now make this construction precise.

\begin{definition}[Morphisms of Modal Orbit Cover Systems]
\label{def:morphisms-of-modal-orbit-cover-systems}
Let $([X],\mathbf c,I)$ and $([Y],\mathbf d,J)$ be modal orbit cover systems. A morphism
\[
\mathcal M
=
(\Delta,\mathbf T_p, \bfR_q, f) :
[X^{(\mathbf c)}]_I
\longrightarrow
[Y^{(\mathbf d)}]_J
\]
consists of a deformation $\Delta\in\mathrm{Hom}_{\HDef}([X],[Y])$, a translation, a rotation, and a morphism witness $f\in\mathrm{Hom}_{\mathbf{C}(7)}(\mathbf c,\mathbf d)$ (Definition~\ref{def:intcomp-category}), satisfying the following condition. The operations $\Delta$, $\bfR_q$, and the linear part $h\coloneqq\ell(f)$ induce the map on modal indices
\[
\hat{\mathcal{M}} = \delta \circ \bfR_q \circ h : I \longrightarrow J,
\]
where $\delta$ is the bijection induced by $\Delta$ as in~\eqref{eq:delta}. Thus $\mathcal M$ is defined only when $\hat{\mathcal{M}}(I)\subseteq J$. An object $(W^{(\mathbf c)},\mu_i)$ is sent to
\[
\left(
(\Delta(\mathbf T_p(W)))^{(\mathbf d)},
\mu_{\hat{\mathcal M}(i)}
\right),
\]
with associated modal orbit cover morphism $h_{i,\hat{\mathcal M}(i)} : W \longrightarrow \Delta(\mathbf T_p(W))$, the map of~\eqref{eq:h_ij} satisfying $\mu_{\hat{\mathcal M}(i)}\circ h_{i,\hat{\mathcal M}(i)} = h\circ\mu_i$. This extends to the canonical induced chord map
\[
\hat f_{i,\hat{\mathcal M}(i)} : W^{(\mathbf c)} \longrightarrow (\Delta(\mathbf T_p(W)))^{(\mathbf d)}
\]
of Proposition~\ref{prop:induced-chord-bijection}, well-defined since $f\in\mathrm{Hom}_{\mathbf{C}(7)}(\mathbf c,\mathbf d)$.
\end{definition}

\begin{example}
\label{ex:modal-orbit-cover-system-map}
Let $[X]$ be the translation orbit of the diatonic scale, let $\mathbf c=(2,2,3)$, and let $I=\{1,6\}$, corresponding to the Ionian and Aeolian modes. Thus $([X],\mathbf c,I)$ is the modal orbit cover system consisting of the tertian orbit covers of the Ionian and Aeolian modes.

Let $\Delta$ be the deformation which, relative to the Ionian mode, flattens the third, sixth, and seventh scale degrees. Then $[\Delta(X)]=[X]$. Now let $\mathbf d=(3,2,1,1)$, $J=\{1,4,6\}$, and define
\[
\mathbf T_p=\mathbf T_9,
\qquad
\bfR_q = \bfR_1,
\qquad
 f(x)=5x + 0. 
\]
We verify that $\mathcal M=(\Delta,\mathbf T_9, \bfR_1, f)$ is a morphism of modal orbit cover systems.

Before checking these conditions, we first confirm that $f\in\mathrm{Hom}_{\mathbf C(7)}(\mathbf c,\mathbf d)$. We have $\hat{\mathbf c}(0)=\{0,2,4\}$, and
\[
f(\hat{\mathbf c}(0)) = 5\cdot\{0,2,4\} = \{0,10,20\} \pmod 7 = \{0,3,6\}.
\]
Next, we have $\hat{\mathbf d}(0)=\{0,3,5,6\}$, which contains the subset $\{0,3,6\}$; thus $f(\hat{\mathbf c}(0))\subseteq\hat{\mathbf d}(0)$, confirming $f\in\mathrm{Hom}_{\mathbf C(7)}(\mathbf c,\mathbf d)$.

First, $\delta(i)=i-2 \pmod7$, so
\[
\hat{\mathcal M}
=
\delta\circ \bfR_1 \circ h
=
5x-1 \pmod7.
\]
Consequently, $\hat{\mathcal M}(1)=4$, $\hat{\mathcal M}(6)=1$, and therefore $\hat{\mathcal M}(I) = \{4,1\} \subseteq J$, verifying the first condition of Definition~\ref{def:morphisms-of-modal-orbit-cover-systems}.

Now consider the object $(W^{(\mathbf c)},\mu_1)$, where $W$ is the C-major scale in the Ionian mode. Since $\Delta(\mathbf T_9(W))=W$, we have $\mathcal{M}(W^{(\mathbf c)},\mu_1) = (W^{(\mathbf d)},\mu_4)$, and the corresponding modal orbit cover morphism is $h_{1,4}$, which extends to the chordal degree map $\hat f_{1,4}$, as so:

\begin{align*}
W^{(\mathbf c)} & \xlongrightarrow{\hat f_{1,4}} (\Delta(\bfT_9(W)))^{(\mathbf d)} \\
\mathrm{I} = \{0,4,7\} & \longmapsto \{5, 11, 2, 4\} = \mathrm{I} \\
\mathrm{II} = \{2,5,9\} & \longmapsto   \{2,  7, 11, 0\} = \mathrm{VI} \\
\mathrm{III} = \{4,7,11\} & \longmapsto \{11, 4, 7, 9\} = \mathrm{IV} \\
\mathrm{IV} = \{5,9,0\} & \longmapsto   \{7, 0, 4, 5\} = \mathrm{II} \\
\mathrm{V} = \{7,11,2\} & \longmapsto  \{4, 9, 0, 2\} = \mathrm{VII} \\
\mathrm{VI} = \{9,0,4\} & \longmapsto    \{0, 5, 9, 11\} = \mathrm{V} \\
\mathrm{VII} = \{11,2,5\} & \longmapsto  \{9, 2, 5, 7\} = \mathrm{III} \\
\end{align*}

The induced functor also acts on morphisms. A groupoid morphism
\[
(\mathbf T_s,\mathbf R_{j-i}) :
(W^{(\mathbf c)},\mu_i)
\longrightarrow
(V^{(\mathbf c)},\mu_j)
\]
is sent to
\[
(\mathbf T_s,\mathbf R_{\hat{\mathcal M}(j)-\hat{\mathcal M}(i)}) :
\left(
(\Delta(\mathbf T_9(W)))^{(\mathbf d)},
\mu_{\hat{\mathcal M}(i)}
\right)
\longrightarrow
\left(
(\Delta(\mathbf T_9(V)))^{(\mathbf d)},
\mu_{\hat{\mathcal M}(j)}
\right),
\]
so that the diagram
\[
\begin{tikzcd}
	{(W^{(\mathbf c)},\mu_i)} &&& {(V^{(\mathbf c)},\mu_j)} \\
	\\
	{\left((\Delta(\mathbf T_9(W)))^{(\mathbf d)},\mu_{\hat{\mathcal M}(i)}\right)}
	&&&
	{\left((\Delta(\mathbf T_9(V)))^{(\mathbf d)},\mu_{\hat{\mathcal M}(j)}\right)}
	\arrow["{(\mathbf T_s,\mathbf R_{j-i})}", from=1-1, to=1-4]
	\arrow["{h_{i,\hat{\mathcal M}(i)}}"', from=1-1, to=3-1]
	\arrow["{h_{j,\hat{\mathcal M}(j)}}", from=1-4, to=3-4]
	\arrow["{(\mathbf T_s,\mathbf R_{\hat{\mathcal M}(j)-\hat{\mathcal M}(i)})}"', from=3-1, to=3-4]
\end{tikzcd}
\]
commutes.

For example, the rotation
\[
\mathbf R_5:\mu_1\longrightarrow\mu_6
\]
is sent to
\[
\mathbf R_{\hat{\mathcal M}(6)-\hat{\mathcal M}(1)}
=
\mathbf R_{1-4}
=
\mathbf R_4,
\]
which carries the image of the Ionian mode, namely $\mu_4$, to the image of
the Aeolian mode, namely $\mu_1$. Likewise,
\[
\mathbf R_2:\mu_6\longrightarrow\mu_1
\]
is sent to
\[
\mathbf R_{\hat{\mathcal M}(1)-\hat{\mathcal M}(6)}
=
\mathbf R_{4-1}
=
\mathbf R_3,
\]
which carries $\mu_1$ back to $\mu_4$. Thus $\mathcal M$ transports both the individual modal orbit covers and the transformations between them into the target modal orbit cover system.
\end{example}

\paragraph{The Category of Modal Orbit Cover Systems}
We are now in a position to define the category of modal orbit cover systems.
\begin{definition}
\label{def:category-of-modal-orbit-cover-systems}
The category $\ModeOrbCovSys$ has:
\begin{enumerate}
\item as objects, modal orbit cover systems $([X],\mathbf{c},I)$;
\item as morphisms, quadruples
\[
\mathcal M
=
(\Delta,\mathbf T_p, \bfR_q, f) :
[X^{(\mathbf c)}]_I
\longrightarrow
[Y^{(\mathbf d)}]_J,
\]
as in Definition~\ref{def:morphisms-of-modal-orbit-cover-systems}. 
\end{enumerate}
\end{definition}

By Definition~\ref{def:morphisms-of-modal-orbit-cover-systems}, every such morphism $\mathcal M$ acts as a functor $[X^{(\mathbf c)}]_I \to [Y^{(\mathbf d)}]_J$, giving $\ModeOrbCovSys$ the structure of a category.

\begin{remark}
Unlike a typical functor, a morphism $\mathcal{M}$ of modal orbit cover systems canonically determines not only the image of each object, but also a distinguished modal orbit cover morphism from the object to its image. This additional information reflects the fact that the functor is induced by concrete musical transformations—namely a deformation, a translation, a rotation, and a linear automorphism coupled with a rotation of the source interval composition.
\end{remark}

\subsubsection{The Category of Harmonic Vocabularies}

Equipping a modal orbit cover system with deformation families produces a harmonic vocabulary $([X], \mathbf{c}, I, \mathcal{D})$. As discussed earlier, deformation families are naturally specified only locally within a given modal orbit cover system, and there is therefore no \emph{canonical}\footnote{There is, of course, a mathematically meaningful notion of preserving deformation data, as discussed in the deformation-family fibration (\S~\ref{subsub:the-deformation-family-fibration}).} musically meaningful notion of preserving deformation data between different harmonic vocabularies. Consequently, the category $\HarmVocab$ has harmonic vocabularies as its objects, while its morphisms are simply the morphisms of the underlying modal orbit cover systems. Thus, $\HarmVocab$ differs from $\ModeOrbCovSys$ only in its objects, which are now equipped with deformation data.

\subsubsection{The Category of Functional Harmonic Systems}
\label{subsub:the-category-of-functional-harmonic-systems}

Recall from \S~\ref{subsub:harmonic-functions-dependent-on-deformation-families} that there are two ways to endow a deformation-equipped modal orbit cover $([X], Y^{(A)}, \mu_i, \mathcal{D})$ with a functional assignment. The first approach regards the functional assignment as \emph{independent} of the deformation family. In this case, one defines
\[
\Phi : Y^{(A)} \longrightarrow \mathcal{P}(F_{\spec}),
\]
and each deformation inherits this assignment by composition with the canonical projection
\[
\mathcal{D}(Y^{(A)})
\xlongrightarrow{p_{Y^{(A)}}}
Y^{(A)}
\xlongrightarrow{\Phi}
\mathcal{P}(F_{\spec}).
\]
The second approach regards the functional assignment as \emph{dependent} on the deformation family, so that the assignment is instead given directly by
\[
\Phi : \mathcal{D}(Y^{(A)}) \longrightarrow \mathcal{P}(F_{\spec}).
\]

Recall that morphisms in $\HarmVocab$ do not include any maps between deformation families. Consequently, only the first notion of functional assignment is compatible with the morphisms of $\HarmVocab$. For, if functional assignments are defined directly on deformation families, then transporting them along morphisms would require corresponding maps between the deformation families themselves, but such maps are deliberately excluded from the definition of $\HarmVocab$.

To see this concretely, suppose
\[
\mathbf{H}_{\mathcal D}
=
([X], \mathbf{c}, I, \mathcal D),
\qquad
\mathbf{H}_{\mathcal E}
=
([X], \mathbf{c}, I, \mathcal E),
\]
are harmonic vocabularies differing only in their deformation families. Let
\[
\mathcal{M} :
\mathbf{H}_{\mathcal D}
\longrightarrow
\mathbf{H}_{\mathcal E}
\]
denote the identity morphism on the underlying modal orbit cover system. Since $\mathcal{M}$ contains no information relating $\mathcal D$ and $\mathcal E$, it induces no maps
\[
\mathcal D(i)(Y^{(\mathbf c)})
\longrightarrow
\mathcal E(i)(Y^{(\mathbf c)})
\]
for any mode $i \in I$. Consequently, there is no natural way to transport functional assignments
\[
\Phi :
\mathcal D(i)(Y^{(\mathbf c)})
\longrightarrow
\mathcal P(F_{\spec})
\]
to corresponding assignments on $\mathcal E(i)(Y^{(\mathbf c)})$.

For this reason, we now define functional harmonic systems using the first approach, in which functional assignments are made on the underlying modal orbit covers and inherited by the deformation families through the fibers of the canonical projections.


Let $\mathbf{V}=([X],\mathbf{c},I,\mathcal D)$ be a harmonic vocabulary, and let $\mathcal F=(F_{\core},F_{\spec},\varphi)$ be a harmonic function schema. A functional assignment for $\mathbf V$ is an element
\[
\Phi\in
\prod_{i\in I}
\left(\mathcal P(F_{\spec})\right)^{\mathbb Z_7^{(\mathbf c)}}
\]
(compare \eqref{eq:F(i)} and \eqref{eq:functional-assignment-family}). Thus, for each mode $i\in I$, there is a map
\[
\Phi_i:\mathbb Z_7^{(\mathbf c)}
\longrightarrow
\mathcal P(F_{\spec}).
\]
A functional harmonic system whose functional assignment is independent of the deformation family is therefore a tuple
\[
\mathbf H=
([X],\mathbf c,I,\mathcal D,(\mathcal F,\Phi)).
\]

\begin{definition}[Morphisms of Functional Harmonic Systems]
\label{def:morphisms-of-functional-harmonic-systems}
Let
\[
\mathbf H=
([X],\mathbf c,I,\mathcal D,(\mathcal F,\Phi))
\qquad\text{and}\qquad
\mathbf H'
=
([Y],\mathbf d,J,\mathcal E,(\mathcal F',\Phi'))
\]
be functional harmonic systems, and let
\[
\mathcal M
=
(\Delta,\mathbf T_p,\mathbf R_q, f) :
([X],\mathbf c,I,\mathcal D)
\longrightarrow
([Y],\mathbf d,J,\mathcal E)
\]
be a morphism of harmonic vocabularies. Writing $\hat{\mathcal M}:I\to J$ for the induced mode map, a morphism
\[
(\mathcal M,g_{\core}) :
\mathbf H
\longrightarrow
\mathbf H'
\]
consists of a set map
\[
g_{\core}:F_{\core}\longrightarrow F'_{\core},
\]
from the core harmonic functions of $\mathcal F$ to those of $\mathcal F'$
(see Definition~\ref{def:func-fibration}), such that, for every $i\in I$, the following condition holds:
\[
g_{\core}(\Phi_{i,\core}(A)) \subseteq \Phi'_{\hat{\mathcal M}(i),\core}(\hat f(A)),
\]
for every $A \in \Z_7^{(\mathbf{c})}$. 
\end{definition}

We denote the category of functional harmonic systems by $\FuncHarmSys$.

\subsubsection{The Category of Harmonic Practices}
Deriving the category of harmonic practices from the category $\FuncHarmSys$ of functional harmonic systems is straightforward.
Recall from \S~\ref{para:the-function-level} that every functional harmonic system $\mathbf{H}$ with core function set $F_{\core}$ determines the set of functional region symbols
\[
\mathbb{R}_{\mathbf H}
\coloneqq
\{\alpha_{\mathrm R}\mid \alpha\in F_{\core}\}
\cong
F_{\core}.
\]
Recall also that the definition of $\mathrm{ExpOpt}$ in~\eqref{eq:expopt} determines, for each $\alpha_{\mathrm R}\in\mathbb R_{\mathbf H}$, the set $\mathrm{ExpOpt}_{\mathbf H}(\alpha_{\mathrm R})$, while the definition of $\mathrm{ExpRules}$ in~\eqref{eq:exprules} determines the set
\[
\mathrm{ExpRules}_{\mathbf H}(\alpha_{\mathrm R})
\coloneqq
\mathcal P\!\left(
\mathrm{ExpOpt}_{\mathbf H}(\alpha_{\mathrm R})
\right)
\times
\{\alpha\}.
\]
Consequently,
\[
\mathrm{FuncExp}_{\mathbf H}
\coloneqq
\prod_{\alpha_{\mathrm R}\in\mathbb R_{\mathbf H}}
\mathrm{ExpRules}_{\mathbf H}(\alpha_{\mathrm R})
\]
is the set of functional expansion specifications for $\mathbf H$. A \emph{syntactic specification} (Definition~\ref{def:syntactic-specification}) for $\mathbf H$ is therefore a pair
\[
(R,E)
\in
\mathbb R_{\mathbf H}
\times
\mathrm{FuncExp}_{\mathbf H}.
\]

Now let
\[
\mathbf H=
([X],\mathbf c,I,\mathcal D,(\mathcal F,\Phi)),
\qquad
\mathbf H'
=
([Y],\mathbf d,J,\mathcal E,(\mathcal F',\Phi'))
\]
be functional harmonic systems, and let
\[
(\mathcal M,g_{\core}) :
\mathbf H
\longrightarrow
\mathbf H'
\]
be a morphism. Since the constructions $\mathbb R_\mathbf{H}$, $\mathrm{ExpOpt}$, $\mathrm{ExpRules}$, and $\mathrm{FuncExp}$ are defined canonically from the core harmonic function set, the map
\[
g_{\core} :
F_{\core}
\longrightarrow
F'_{\core}
\]
induces maps
\[
\mathbb R_{\mathbf H}
\longrightarrow
\mathbb R_{\mathbf H'},
\qquad
\mathrm{ExpOpt}_{\mathbf H}(\alpha_{\mathrm R})
\longrightarrow
\mathrm{ExpOpt}_{\mathbf H'}\!\left(g_{\core}(\alpha)_{\mathrm R}\right),
\]
for each $\alpha_{\mathrm R}\in\mathbb R_{\mathbf H}$, and hence induces maps
\[
\mathrm{ExpRules}_{\mathbf H}(\alpha_{\mathrm R})
\longrightarrow
\mathrm{ExpRules}_{\mathbf H'}\!\left(g_{\core}(\alpha)_{\mathrm R}\right),
\]
and
\[
\mathrm{FuncExp}_{\mathbf H}
\longrightarrow
\mathrm{FuncExp}_{\mathbf H'}.
\]
Therefore $g_{\core}$ induces a map
\[
\mathbb R_{\mathbf H}
\times
\mathrm{FuncExp}_{\mathbf H}
\longrightarrow
\mathbb R_{\mathbf H'}
\times
\mathrm{FuncExp}_{\mathbf H'}
\]
between syntactic specifications.

\begin{definition}[Morphisms of Harmonic Practices]
Let
\[
\mathbf P=(\mathbf H,\Sigma=(R,E))
\qquad\text{and}\qquad
\mathbf P'=(\mathbf H',\Sigma'=(R',E'))
\]
be harmonic practices. A \emph{morphism of harmonic practices}
\[
(\mathcal M,g_{\core}) :
\mathbf P
\longrightarrow
\mathbf P'
\]
is a morphism
\[
(\mathcal M,g_{\core}) :
\mathbf H
\longrightarrow
\mathbf H'
\]
of functional harmonic systems such that
\[
g_{\core}(R)\subseteq R'
\qquad\text{and}\qquad
g_{\core}(E)\subseteq E'.
\]
\end{definition}

Functional harmonic systems and their morphisms form a category, which we denote by $\HarmPractice$.

\begin{remark}
Rather than requiring equality, the inclusion condition allows the codomain harmonic practice to extend the syntactic specification of the domain. Thus every functional region and expansion rule of $(R,E)$ is preserved under $g_{\core}$, while additional functional regions and expansion rules may be introduced in $(R',E')$. Morphisms of harmonic practices therefore allow for syntactic enrichments of one harmonic practice by another.
\end{remark}

\begin{example}
\label{ex:morphism-of-harmonic-practices}
We close with an example of a morphism of harmonic practices. Let $\mathbf{P}$ be a harmonic practice defined by
\begin{enumerate}
\item $[X]$ the diatonic translation orbit;
\item $\mathbf{c} = (2,2,3)$ the tertian interval composition;
\item $I = \{1,6\}$ the Ionian and Aeolian modes;
\item $\mathcal D$ arbitrary (we do not need to specify morphism data on this component);
\item $(\mathcal{F}, \Phi)$ given by $\mathcal{F} = (F_{\core}, F_{\spec}, \varphi)$ where
\[
F_{\core} = \{\mathrm{T}, \mathrm{S}, \mathrm{D}\}, \qquad
F_{\spec} = \{\mathrm{T}, \mathrm{Tp}, \mathrm{S}, \mathrm{Sp}, \mathrm{S}^*, \mathrm{D}, \mathrm{Dp}, \mathrm{D}^*\},
\]
and $\varphi : F_{\spec} \to F_{\core}$ in the obvious way. The functional assignment $\Phi_1$ on the Roman numerals is given by
\begin{align*}
\Phi_1(\mathrm{I})   &= \{\mathrm{T}\},  & \Phi_1(\mathrm{II})  &= \{\mathrm{Sp}\}, & \Phi_1(\mathrm{III}) &= \{\mathrm{Dp}\}, \\
\Phi_1(\mathrm{IV})  &= \{\mathrm{S}\},  & \Phi_1(\mathrm{V})   &= \{\mathrm{D}\},  & \Phi_1(\mathrm{VI})  &= \{\mathrm{Tp}\}, \\
\Phi_1(\mathrm{VII}) &= \{\mathrm{D}^*\}, & & &
\end{align*}
and $\Phi_6$ by
\begin{align*}
\Phi_6(\mathrm{I})   &= \{\mathrm{T}\},   & \Phi_6(\mathrm{II})  &= \{\mathrm{S}^*\}, & \Phi_6(\mathrm{III}) &= \{\mathrm{Tp}\}, \\
\Phi_6(\mathrm{IV})  &= \{\mathrm{S}\},   & \Phi_6(\mathrm{V})   &= \{\mathrm{D}\},   & \Phi_6(\mathrm{VI})  &= \{\mathrm{Sp}\}, \\
\Phi_6(\mathrm{VII}) &= \{\mathrm{Dp}\},  & & &
\end{align*}
\item $\Sigma = (R,E)$ as in Rohrmeier's original (\S~\ref{para:phrase-level-rewrite-rules}–\ref{para:functional-level-rewrite-rules}).
\end{enumerate}
This specifies the data for the domain of an ordinary common-practice harmonic practice.

Now define $X' = \{\mathrm{A}, \mathrm{B}, \mathrm{C}, \mathrm{D}, \mathrm{E}, \mathrm{F}^\sharp, \mathrm{G}^\sharp\}$, and define the harmonic practice $\mathbf{P}'$ given by
\begin{enumerate}
\item $[X']$;
\item $\mathbf{c}' = (2, 3, 1, 1)$;
\item $I' = \{1,2\}$;
\item $\mathcal{D}'$ (which we will specify later);
\item $(\mathcal{F}', \Phi')$, where $\mathcal{F}' = \mathcal{F}$ is as above, and $\Phi'$ is defined as follows. The assignment $\Phi'_1$ is given by
\begin{align*}
\Phi'_1(\mathrm{I})   &= \{\mathrm{T}\},  & \Phi'_1(\mathrm{II})  &= \{\mathrm{Tp}\}, & \Phi'_1(\mathrm{III}) &= \{\mathrm{S}\}, \\
\Phi'_1(\mathrm{IV})  &= \{\mathrm{Sp}\}, & \Phi'_1(\mathrm{V})   &= \{\mathrm{D}^*\}, & \Phi'_1(\mathrm{VI})  &= \{\mathrm{D}\}, \\
\Phi'_1(\mathrm{VII}) &= \{\mathrm{Dp}\}, & & &
\end{align*}
and $\Phi'_2$ by
\begin{align*}
\Phi'_2(\mathrm{I})   &= \{\mathrm{T}\},  & \Phi'_2(\mathrm{II})  &= \{\mathrm{Sp}\}, & \Phi'_2(\mathrm{III}) &= \{\mathrm{S}\}, \\
\Phi'_2(\mathrm{IV})  &= \{\mathrm{S}^*\}, & \Phi'_2(\mathrm{V})   &= \{\mathrm{Dp}\}, & \Phi'_2(\mathrm{VI})  &= \{\mathrm{D}\}, \\
\Phi'_2(\mathrm{VII}) &= \{\mathrm{Tp}\}, & & &
\end{align*}
\item $\Sigma' = \Sigma = (R, E)$ again as in Rohrmeier's original.
\end{enumerate}

Define the identity morphism
\[
g_{\core} : F_{\core} \longrightarrow F'_{\core}
\]
on the respective components (since $F'_{\core} = F_{\core}$), and define the morphism $\mathcal M = (\Delta, \bfT_p, \bfR_q, f)$ of modal orbit cover systems by:
\begin{enumerate}
\item $\Delta = (0,0,0,0,1,1,0) \in \Def(X)$;
\item $p = 0$;
\item $q = 3$;
\item $f(x) = 3x$.
\end{enumerate}
We verify that $(\mathcal{M}, g_{\core})$ constitutes a morphism
\[
(\mathcal{M}, g_{\core}) : \mathbf{P} \longrightarrow \mathbf{P}'
\]
of harmonic practices.

First, $\Delta$ extends to a map $\Delta : [X] \to [X']$: relative to the Aeolian mode, it raises the sixth and seventh scale degrees, so that $\Delta(X)$ is the corresponding melodic minor scale. Next we verify the condition of Definition~\ref{def:morphisms-of-modal-orbit-cover-systems} that
\[
\hat{\mathcal M} = \delta \circ \bfR_q \circ h : I \longrightarrow I'
\]
satisfies $\hat{\mathcal{M}}(I) \subseteq I'$, where $h \coloneqq \ell(f)$, i.e.\ $h(x) = f(x) = 3x$. Here $\delta(x) = x + 2$, so
\[
\hat{\mathcal M}(x) = \delta \circ \bfR_q \circ h(x) = 3x + (3 + 2) = 3x + 5,
\]
and hence $\hat{\mathcal{M}}(\{1,6\}) = \{1,2\} = I' \subseteq I'$, as required.

Turning to the chord-level data: since $\hat{\mathbf c}(0) = \{0,2,4\}$ and $\hat{\mathbf c}'(0) = \{0,2,5,6\}$, we have
\[
f(\hat{\mathbf c}(0)) = 3\cdot\{0,2,4\} = \{0,6,12\} \pmod 7 = \{0,5,6\} \subseteq \{0,2,5,6\},
\]
confirming $f \in \mathrm{Hom}_{\mathbf C(7)}(\mathbf c, \mathbf c')$. Thus we have chord maps
\[
\hat{f}_{i, \hat{\mathcal{M}}(i)} : W^{(\mathbf c)} \longrightarrow (\Delta(W))^{(\mathbf c')}
\]
for every mode $(W, \mu_i) \in [X^{(\mathbf c)}]_I$. As a concrete example, take $(W, \mu_1)$ the C Ionian mode. Then
\[
h_{1, \hat{\mathcal M}(1)} = h_{1,1} : W \longrightarrow \Delta(X)
\]
is given by
\[
\begin{aligned}
\mathrm C &\mapsto \mathrm A, \\
\mathrm D &\mapsto \mathrm D, \\
\mathrm E &\mapsto \mathrm G^\sharp, \\
\mathrm F &\mapsto \mathrm C, \\
\mathrm G &\mapsto \mathrm F^\sharp, \\
\mathrm A &\mapsto \mathrm B, \\
\mathrm B &\mapsto \mathrm E.
\end{aligned}
\]
The induced chord map $\hat{f}_{1, \hat{\mathcal{M}}(1)}$ is given by
\begin{align*}
\mathrm{I}   = \{\mathrm C,\mathrm E,\mathrm G\} &\longmapsto \{\mathrm A,\mathrm C,\mathrm F^\sharp,\mathrm G^\sharp\} = \mathrm{I}', \\
\mathrm{II}  = \{\mathrm D,\mathrm F,\mathrm A\} &\longmapsto \{\mathrm D,\mathrm F^\sharp,\mathrm B,\mathrm C\} = \mathrm{IV}', \\
\mathrm{III} = \{\mathrm E,\mathrm G,\mathrm B\} &\longmapsto \{\mathrm G^\sharp,\mathrm B,\mathrm E,\mathrm F^\sharp\} = \mathrm{VII}', \\
\mathrm{IV}  = \{\mathrm F,\mathrm A,\mathrm C\} &\longmapsto \{\mathrm C,\mathrm E,\mathrm A,\mathrm B\} = \mathrm{III}', \\
\mathrm{V}   = \{\mathrm G,\mathrm B,\mathrm D\} &\longmapsto \{\mathrm F^\sharp,\mathrm A,\mathrm D,\mathrm E\} = \mathrm{VI}', \\
\mathrm{VI}  = \{\mathrm A,\mathrm C,\mathrm E\} &\longmapsto \{\mathrm B,\mathrm D,\mathrm G^\sharp,\mathrm A\} = \mathrm{II}', \\
\mathrm{VII} = \{\mathrm B,\mathrm D,\mathrm F\} &\longmapsto \{\mathrm E,\mathrm G^\sharp,\mathrm C,\mathrm D\} = \mathrm{V}',
\end{align*}
where the target Roman numerals are now read relative to the target mode's own tonic $\mathrm A$.

We also observe that for each $i \in I$,
\[
\Phi'_{\hat{\mathcal M}(i)} = \Phi_{\hat f_{i, \hat{\mathcal M}(i)}}
\]
is the pushforward of $\Phi_i$ along $\hat f_{i, \hat{\mathcal M}(i)}$. For $i=1$, this pushforward is given by
\begin{align*}
\Phi_{\hat f_{1,1}}(\{\mathrm A,\mathrm C,\mathrm F^\sharp,\mathrm G^\sharp\}) &= \{\mathrm T\},  & \Phi_{\hat f_{1,1}}(\{\mathrm B,\mathrm D,\mathrm G^\sharp,\mathrm A\}) &= \{\mathrm{Tp}\}, \\
\Phi_{\hat f_{1,1}}(\{\mathrm C,\mathrm E,\mathrm A,\mathrm B\}) &= \{\mathrm S\},   & \Phi_{\hat f_{1,1}}(\{\mathrm D,\mathrm F^\sharp,\mathrm B,\mathrm C\}) &= \{\mathrm{Sp}\}, \\
\Phi_{\hat f_{1,1}}(\{\mathrm E,\mathrm G^\sharp,\mathrm C,\mathrm D\}) &= \{\mathrm D^*\}, & \Phi_{\hat f_{1,1}}(\{\mathrm F^\sharp,\mathrm A,\mathrm D,\mathrm E\}) &= \{\mathrm D\}, \\
\Phi_{\hat f_{1,1}}(\{\mathrm G^\sharp,\mathrm B,\mathrm E,\mathrm F^\sharp\}) &= \{\mathrm{Dp}\}, & &
\end{align*}
which agrees exactly with $\Phi'_1$ chord-by-chord, so $g_{\core}$ trivially satisfies the inclusion condition of Definition~\ref{def:morphisms-of-functional-harmonic-systems} on this component.

Let us spell out the other mode $(W, \mu_6) \in [X^{(\mathbf c)}]_I$. We have
\[
h_{6, \hat{\mathcal M}(6)} = h_{6,2} : W \longrightarrow \Delta(W)
\]
given by
\[
\begin{aligned}
\mathrm A &\mapsto \mathrm B, \\
\mathrm B &\mapsto \mathrm E, \\
\mathrm C &\mapsto \mathrm A, \\
\mathrm D &\mapsto \mathrm D, \\
\mathrm E &\mapsto \mathrm G^\sharp, \\
\mathrm F &\mapsto \mathrm C, \\
\mathrm G &\mapsto \mathrm F^\sharp,
\end{aligned}
\]
The chord map $\hat{f}_{6, 2}$ is given by
\begin{align*}
\mathrm{I}   = \{\mathrm A,\mathrm C,\mathrm E\} &\longmapsto \{\mathrm B,\mathrm D,\mathrm G^\sharp,\mathrm A\} = \mathrm{I}', \\
\mathrm{II}  = \{\mathrm B,\mathrm D,\mathrm F\} &\longmapsto \{\mathrm E,\mathrm G^\sharp,\mathrm C,\mathrm D\} = \mathrm{IV}', \\
\mathrm{III} = \{\mathrm C,\mathrm E,\mathrm G\} &\longmapsto \{\mathrm A,\mathrm C,\mathrm F^\sharp,\mathrm G^\sharp\} = \mathrm{VII}', \\
\mathrm{IV}  = \{\mathrm D,\mathrm F,\mathrm A\} &\longmapsto \{\mathrm D,\mathrm F^\sharp,\mathrm B,\mathrm C\} = \mathrm{III}', \\
\mathrm{V}   = \{\mathrm E,\mathrm G,\mathrm B\} &\longmapsto \{\mathrm G^\sharp,\mathrm B,\mathrm E,\mathrm F^\sharp\} = \mathrm{VI}', \\
\mathrm{VI}  = \{\mathrm F,\mathrm A,\mathrm C\} &\longmapsto \{\mathrm C,\mathrm E,\mathrm A,\mathrm B\} = \mathrm{II}', \\
\mathrm{VII} = \{\mathrm G,\mathrm B,\mathrm D\} &\longmapsto \{\mathrm F^\sharp,\mathrm A,\mathrm D,\mathrm E\} = \mathrm{V}',
\end{align*}
The pushforward of $\Phi_6$ along $\hat f_{6,2}$ is given by
\begin{align*}
\Phi_{\hat f_{6,2}}(\{\mathrm B,\mathrm D,\mathrm G^\sharp,\mathrm A\}) &= \{\mathrm T\},  & \Phi_{\hat f_{6,2}}(\{\mathrm C,\mathrm E,\mathrm A,\mathrm B\}) &= \{\mathrm{Sp}\}, \\
\Phi_{\hat f_{6,2}}(\{\mathrm D,\mathrm F^\sharp,\mathrm B,\mathrm C\}) &= \{\mathrm S\},   & \Phi_{\hat f_{6,2}}(\{\mathrm E,\mathrm G^\sharp,\mathrm C,\mathrm D\}) &= \{\mathrm S^*\}, \\
\Phi_{\hat f_{6,2}}(\{\mathrm F^\sharp,\mathrm A,\mathrm D,\mathrm E\}) &= \{\mathrm{Dp}\}, & \Phi_{\hat f_{6,2}}(\{\mathrm G^\sharp,\mathrm B,\mathrm E,\mathrm F^\sharp\}) &= \{\mathrm D\}, \\
\Phi_{\hat f_{6,2}}(\{\mathrm A,\mathrm C,\mathrm F^\sharp,\mathrm G^\sharp\}) &= \{\mathrm{Tp}\}, & &
\end{align*}
which again agrees exactly with $\Phi'_2$ chord-by-chord. Having verified the inclusion condition of Definition~\ref{def:morphisms-of-functional-harmonic-systems} for both $i=1$ and $i=6$, we conclude that $(\mathcal{M}, g_{\core})$ is indeed a morphism
\[
(\mathcal{M}, g_{\core}) : \mathbf{P} \longrightarrow \mathbf{P}'
\]
of harmonic practices.

We have not specified the deformation data $\mathcal{D}'$ for $\mathbf{P}'$; as with much of this data, it may often develop gradually through the course of composing in a harmonic practice, and we specify here only what the derivation below (Figure~\ref{fig:syntax-preservation}) requires. Now the chord with the dominant function D in $(\Delta(W)^{\mathbf c'}, \mu_1)$ is the chord $\mathrm{VI}' = \{ \mathrm{F}^\sharp, \mathrm{A}, \mathrm{D}, \mathrm{E}\}$. To get a secondary dominant, we first need a deformation of $\mathrm{VI}'$ that puts it in a transpositional relationship with $\mathrm{I}'$. We do this by raising the fourth and fifth degrees, so $(\sharp 4 \sharp 5) \in \mathcal{D}'(1)$, and we have
\[
(\sharp 4 \sharp 5)(\mathrm{VI}') = \bfT_9(\mathrm{I}'). 
\]
Hence we can take secondary chords of $(\sharp 4 \sharp 5)(\mathrm{VI}')$. 

To illustrate then how we can translate actual harmonic progressions from $\mathbf P$ to $\mathbf P'$, recall the derivation for the Bach progression of Figure~\ref{fig:Bach-derivation}. In common-practice tonality, a secondary dominant of $\mathrm V$ requires no preliminary deformation, since $\mathrm V$ is already a transposition of $\mathrm I$ (both being major triads); a single application of the dominant relation therefore suffices to reach $\mathrm{V/V}$, which resolves directly to the ordinary $\mathrm V$. In $\mathbf P'$, by contrast, $\mathrm{VI}'$ is not itself a transposition of $\mathrm I'$, so before we can take its secondary dominant we must first deform it into one via $(\sharp 4 \sharp 5)$. The resulting secondary dominant $\mathrm{VI}'/\mathrm{VI}'^{\sharp 4 \sharp 5}$ should then resolve to $\mathrm{VI}'^{\sharp 4 \sharp 5}$---the chord of which it is literally the dominant---rather than to the ordinary $\mathrm{VI}'$ that the original progression actually targets. Realizing the same musical outcome in $\mathbf P'$ therefore costs one additional step: a second duplication of the functional region $\mathrm D_{\mathrm R}$, carrying the derivation from the formally correct resolution $\mathrm{VI}'^{\sharp 4 \sharp 5}$ back to the underlying $\mathrm{VI}'$. Using the same derivation as a template, adjusted for this one additional duplication, we obtain the nearly syntactically identical progression of Figure~\ref{fig:syntax-preservation} in $\mathbf P'$.
\end{example}

\begin{sidewaysfigure}[htbp]
\centering
\includegraphics[width=0.85\linewidth]{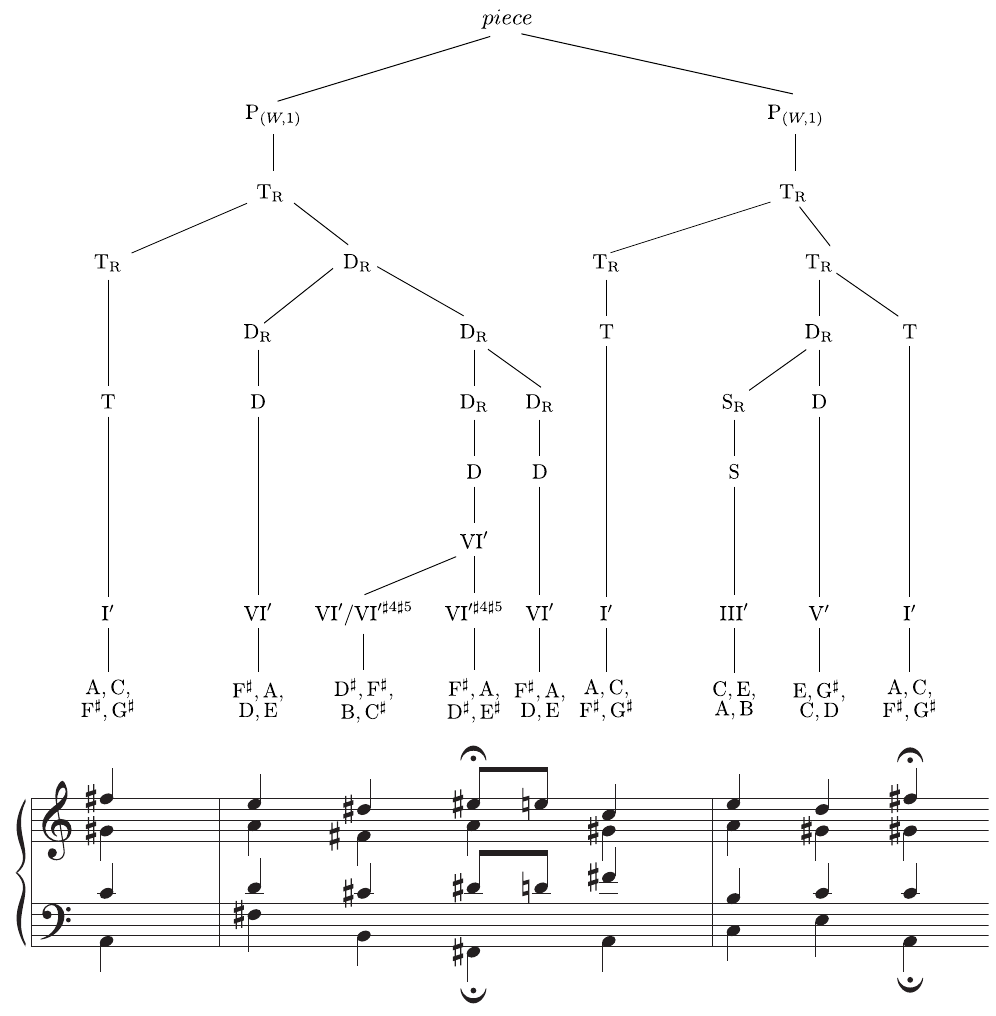}
\caption{The derivation of Figure~\ref{fig:Bach-derivation}, transported along the morphism $(\mathcal M, g_{\core}) : \mathbf P \to \mathbf P'$, with the single additional duplication of the third derived $\mathrm D_{\mathrm R}$ discussed in the text, carrying the derivation from the secondary dominant's formal resolution $\mathrm{VI}'^{\sharp 4 \sharp 5}$ to the underlying $\mathrm{VI}'$.}
\label{fig:syntax-preservation}
\end{sidewaysfigure}

\section{Conclusion}
\label{sec:conclusion}

We have now achieved what we set out to do: provide mathematical principles of a generalized tonal practice. We began with an algebraic theory of modes and scales (\S~\ref{sec:modes-and-scales}), the basis of the whole theory, and then considered three special classes of mode isomorphisms---translations, rotations, and deformations---which supplied the fundamental building blocks for everything that followed. No other morphisms were considered at the mode level in this paper; these three constitute a musically prominent class, one already familiar from much common-practice-period tonality, as each preserves one piece of musically relevant data: translations preserve mode type, rotations preserve scale, and deformations (often, although not always) preserve tonic (Table~\ref{tab:mode-isomorphisms}).

Using deformations, we defined the category $\HDef$ of translation orbits of heptatonic scales, with morphisms the closure of deformations under composition, forming the base of the hierarchy of fibrations constructed in \S~\ref{sec:a-hierarchy-of-fibrations-over-heptatonic-scales}. We then presented categories of orbit covers (\S~\ref{sec:orbit-covers}) and showed that two orbit covers are isomorphic exactly when their nerves are (Theorem~\ref{th:nerve-isomorphism})---a musically salient result.

The hierarchy of fibrations in \S~\ref{sec:a-hierarchy-of-fibrations-over-heptatonic-scales} is the core mathematical machinery of the paper. By associating to each translation orbit in $\HDef$ its translation-rotation action groupoid (\S~\ref{subsub:the-translation-rotation-action-groupoid-fibration}), we obtain the category of heptatonic modes with translations, rotations, and deformations as morphisms, but with deformations now recast as transportations of one action groupoid onto another. Above this we took the orbit covers over each mode, generalizing the classical covering of diatonic scales by tertian triads. Over that, we took deformation families, generalizing the process of altering tones in a key to obtain tones outside it---the source, in this framework, of secondary chords and of the pivot chords that support modulation. Finally, assigning harmonic functions to the complete set of chords in an orbit cover equipped with a deformation family culminates the hierarchy (Figure~\ref{fig:fibration-hierarchy}).

This hierarchy supplies the mathematical groundwork for the ultimate object of our theory---\emph{harmonic practices} (\S~\ref{sec:formalizing-harmonic-practices}). It is here that the concrete musical payoff appears: a harmonic practice is a direct generalization of common-practice tonality, especially as formalized by Rohrmeier~\cite{rohrmeier2011towards}. Recall that a harmonic practice is given by six pieces of data: a choice of scale class, an orbit-cover structure, a choice of modes, a family of chromatic alterations, an assignment of harmonic functions, and a generative syntax. This is the formal object promised in the introduction (\S~\ref{sec:introduction}): a harmonic vocabulary equipped with a functional assignment and a generative syntax, together with a notion of morphism that lets one harmonic practice inherit and enrich the structure of another.

What should be emphasized is not only that harmonic practices have been formally defined, but that they form a category. This is not merely an interesting formal result; it is one of the core justifications for our claim that the definition generalizes common-practice tonality. A morphism of harmonic practices constitutes an inheritance and enrichment of the structure of one harmonic practice in another. A morphism out of common-practice tonality into another harmonic practice, for instance, preserves the structure of common-practice tonality. It bears reiterating that such morphisms were used as the basis of the author's compositional practice well before the formal definition was worked out. This is therefore a case of theory serving the practice of composition, rather than composition proceeding as a proof of concept for the theory. I hope other composers find here a tool useful for their own aesthetic aims.

As concrete evidence of the framework's compositional use, Example~\ref{ex:piano-vars} demonstrates an adaptation of a phrase from the theme of my \emph{Variations for Piano}, Op.~31, adjusted slightly from the actual piece for pedagogical demonstration. Example~\ref{ex:morphism-of-harmonic-practices} shows a further case: a morphism of harmonic practices that translates a Bach progression (Figure~\ref{fig:Bach-derivation}) into a foreign harmonic practice while preserving a nearly identical syntactic derivation (Figure~\ref{fig:syntax-preservation}).

This bears on a point raised in the introduction (\S~\ref{sec:introduction}): common-practice tonality is economical. It asks comparatively little of the listener at the point of listening while returning a great deal, and that is an achievement most post-tonal harmonic practices have not particularly tried to reproduce.\footnote{Feldman~\cite[174]{feldman2000give} treated this as a deliberate refusal: functional harmony, on his account, ``hears for us''---it does the listener's work in advance. His own music withholds that convenience.} The aim of this theory has been to keep that economy without also keeping the vocabulary that historically carried it.
Common-practice tonality's tertian, diatonic vocabulary is one way of supporting a functional syntax, but it need not be the only one. Harmonic practices preserve the underlying machinery---cadence, secondary chords, modulation via pivot chords, the nesting of functional regions, prolongation, and so on---while generalizing the harmonic vocabulary a composer can use to write with it. This realizes the original aesthetic motivation of the thesis: the economy of common-practice tonality, now available across a far wider range of harmonic palettes.

\subsection*{Future Directions}

The most immediate direction for future work is simply composing with the theory as it now stands. The framework has already guided my own compositional practice implicitly, prior to its explicit formulation; now that it has one, my hope is that other composers will find it useful for theirs too.

A more practical direction is implementation in software. Computing all the data a harmonic practice makes available---data a composer needs in order to write with it efficiently---is tedious and error-prone by hand. Software would also make other tasks tractable that are not, by hand: enumerating the morphisms that exist between orbit covers, harmonic practices, and the other objects discussed throughout the paper, for
instance.

The most substantial direction, and the one untouched in this document, is a formalism for melody within this framework: melodic syntax, voice-leading syntax, and the interaction between the two. What we have covered here are the mathematical principles of a generalized tonal practice with respect to its \emph{harmonic} syntax; a Part~II would develop the analogous principles for its \emph{melodic} syntax.

\pagebreak

\printbibliography

\end{document}